\documentclass{amsart}
\usepackage{amssymb}
\usepackage[T1]{fontenc}
\usepackage{mathrsfs}
\usepackage{enumerate}
 \usepackage{hyperref}
 \usepackage{cleveref}
\usepackage{tikz}

\newtheorem{thm}{Theorem}
\newtheorem{defn}{Definition}
\newtheorem{fact}{Fact}
\newtheorem{lem}[thm]{Lemma}
\newtheorem{prop}[thm]{Proposition}
\newtheorem{cor}[thm]{Corollary}
\newtheorem{question}{Question}

\newcommand{\pp}{\mathcal{P}}
\newcommand{\up}{\upharpoonright}
\newcommand{\ra}{\rightarrow}
\newcommand{\mc}{\mathcal}
\newcommand{\ms}{\mathscr}
\newcommand{\msa}{\mathscr{A}}
\newcommand{\msb}{\mathscr{B}}

\newcommand{\bfe}{\mathbf{E}}
\newcommand{\mct}{\mathcal{T}}
\newcommand{\mcc}{\mathcal{C}}
\newcommand{\mcr}{\mathcal{R}}
\newcommand{\mcx}{\mathcal{X}}

\makeatletter
\@addtoreset{case}{thm}
\@addtoreset{case}{lem}
\@addtoreset{case}{prop}
\makeatother

\begin{document}
 \title{Fragments of Martin's axiom}
  \author{Yinhe Peng}
  \address{Academy of Mathematics and Systems Science, Chinese Academy of Sciences\\ East Zhong Guan Cun Road No. 55\\Beijing 100190\\China}
\email{pengyinhe@amss.ac.cn}



  \begin{abstract}
We show that Martin's axiom for $\omega_1$ dense sets is equivalent to its fragment asserting that every ccc poset has the Knaster property K$_3$. On the other hand, we show that the dimension 3 in K$_3$ is in some sense minimal.
   \end{abstract}
  
  \begingroup
  \def\uppercasenonmath#1{}
    \maketitle
    \endgroup

    \section{Introduction}
    
    One of the most important developments of the   forcing method, introduced by Cohen \cite{Cohen}, is the area of so-called internal forcing, more commonly known as Forcing Axioms. This area was introduced in the seminal paper by Solovay and Tennenbaum \cite{ST}, which developed the method of iterated forcing resulting in the first forcing axiom known today under the name Martin's axiom (MA) asserting that MA$_\kappa$ holds for all $\kappa$ less than the continuum.
    
     For a cardinal $\kappa$,  {\rm MA}$_{\kappa}$ is the   assertion that
  \begin{itemize}
  \item for every ccc poset $\mc{P}$, if $\mc{X}$ is a collection  of $\kappa$ many   dense  subsets of $\mc{P}$, then there is a  filter $G\subseteq \mc{P}$ such that $G\cap D\neq \emptyset$ for all $D\in \mc{X}$.\footnote{$D\subseteq\pp$ is dense if it is downward cofinal. $G\subseteq\pp$ is a filter if it is upward closed and any finite elements of $G$ have a common lower bound in $G$.}
  \end{itemize}  
  A poset (partially ordered set) $\mc{P}$ has the \emph{countable chain condition} (ccc)  if every uncountable subset of $\mc{P}$ contains two elements with a common lower bound.
  
  Numerous applications of MA in   mathematics were found in the decade following its discovery. The 1984 monograph by Fremlin \cite{Fremlin} collects these applications up to that date. They are organized according to the corresponding fragments of MA needed for their application, which reveals otherwise hidden information about the relationships between problems in different areas of mathematics. 
  
It is not surprising that fragments of MA correspond directly to the strengthenings of the countable chain condition that have been considered in mathematics long before the discovery of forcing.   For example, the assertion that
every ccc poset has the Knaster's property K$_2$,
   considered by Knaster and Szpilrajn in Problem 192 of the Scottish Book in the 1940s (see \cite{The Scottish Book}), arises naturally when attempting to solve the famous Suslin Problem \cite{Souslin},  which asks whether the unit interval can be characterized among ordered continua by the countable chain condition.

 There is   another reason for considering fragments of MA and their relative strengths. This reason is drawn from the phenomenon that some problems in mathematics are solved through the combination of   corresponding fragments derived from two contradictory principles. For example, a positive answer to Katetov's problem follows from the combination of a fragment of MA$_{\omega_1}$ and a weak form of the Continuum Hypothesis (CH) \cite{LT2002}.  
  Kunen and Tall analyzed   fragments of MA$_{\omega_1}$   related to problems of this kind in \cite{KT}. For example, it is asked in \cite{KT} whether the weak diamond $2^\omega<2^{\omega_1}$ is consistent with the fragment  of MA$_{\omega_1}$ (denoted there by H) asserting that every ccc poset has precaliber $\omega_1$, or equivalently that every uncountable ccc poset has an uncountable centered subset. This problem was answered     negatively by Todorcevic and Velickovic in  \cite{TV} where the higher-dimensional versions of Knaster's axioms are considered as important fragments of MA$_{\omega_1}$. Before introducing these fragments, we need unsurprisingly to introduce the corresponding high-dimensions of Knaster's original property K$_2$.

  For $n\geq 2$, a poset $\mathcal{P}$ has \emph{property {\rm K}$_n$} (K for Knaster) if every uncountable subset of $\mathcal{P}$ has an uncountable subset that is $n$-linked where a subset $X\subseteq \mathcal{P}$ is \emph{$n$-linked} if every $n$-element subset of $X$ has a common lower bound. We may omit $n$ if $n=2$.
  
  It is easy to see that 
  \[\cdots\Rightarrow\text{ K$_{n+1}\Rightarrow$ K$_n\Rightarrow\cdots\Rightarrow$ K$_2$}.\] 
  And it is independent of ZFC, the Zermelo-Fraenkel axiomatic set theory with the Axiom of Choice, that the implications are reversible. Moreover, there are three well-known patterns on reversibility of implications.
  \begin{itemize}
  \item No implication is reversible under the Continuum Hypothesis (see, e.g., \cite{KT}, \cite{Galvin}, \cite{Todorcevic86}).
  \item All implications are reversible under MA$_{\omega_1}$.
  \item For $n\geq 2$,   there is a switch at $n$ from non-reversible to reversible, i.e.,
  \[\text{K$_i\not\Rightarrow$ K$_{i+1}$ for $2\leq i<n$  and K$_j\Rightarrow$ K$_{j+1}$ for $j\geq n$},\]
  by, e.g., a  finite support iteration of  sufficient posets with property K$_n$  from a model of CH.
  \end{itemize}
  We may also view the first two patterns as special cases of the third pattern by taking $n$ to be $\infty$ and 2 respectively. No new pattern has ever been found.
  
Then our investigation on fragments of Martin's axiom confirms that no other pattern exists. More precisely, by Theorem \ref{thm K3toK4},  for   $n\geq 2$,
  \[\text{ if K$_n\Rightarrow$ K$_{n+1}$, then K$_{n}\Rightarrow$ K$_{m}$ for all $m>n$}.\]


Fragments with the full strength of MA$_{\omega_1}$ have their own special significance. In the 1980s, Todorcevic and Velickovic  \cite{TV} proved that the following fragment of MA$_{\omega_1}$,
  \[\text{every uncountable ccc poset has an uncountable centered subset},\]
is  equivalent to MA$_{\omega_1}$. This result, along with  the underlying ideas  (\cite{Bell} and \cite{Todorcevic87}), plays an important role in later investigations.
  
       A subset $X$ of a poset $\mc{P}$ is \emph{centered} if every finite subset of $X$ has a common lower bound.
       
       Ever since, no seemingly weaker fragment  of this type was proved or disproved to have the full strength of MA$_{\omega_1}$. In particular, whether the following fragment   is equivalent to MA$_{\omega_1}$ was left open.
      \[\text{Every uncountable ccc poset has an uncountable $n$-linked subset for every }n\geq 2,\]
       \[\text{or equivalently, every ccc poset has property K$_n$ for every }n\geq 2.\footnote{For a ccc poset $\mc{P}$ and an uncountable $X\subseteq \mc{P}$, the collection of finite centered subsets of $X$ ordered by reverse inclusion is ccc and whose uncountable $n$-linked subset induces an uncountable $n$-linked subset of $X$.}\] 
       For progresses  and partial results, see a recent survey \cite{Bagaria}.


    \subsection{Strengthening ccc to property {\rm K}$_n$}
    
    We will use the following notation   introduced in  \cite{TV} to simplify our expression. 
      \begin{itemize}
      \item  For $n\geq 2$, $\ms{K}_n$ is the assertion that every ccc poset has property K$_n$.
      \end{itemize}

      A fundamental  type of questions for $\ms{K}_n$'s is: 
      \begin{itemize}
      \item What are the exact strengths of $\ms{K}_n$'s?
      \end{itemize}
      The investigation of these questions leads to another fundamental type of questions for $\ms{K}_n$'s:
       \begin{itemize}
      \item  What consequences do  $\ms{K}_n$'s have?
      \end{itemize}
       More specific questions of these types  have been asked and investigated, see e.g., \cite{TV},  \cite{Todorcevic91}, \cite{Moore2000}, \cite{LT2001}, \cite{Yorioka10}, \cite{Yorioka11}, \cite{Bagaria}.
      
      Progresses have been made to the second type of questions. We list some well-known consequences. First, the following statements follow from $\ms{K}_2$.
      \begin{itemize}
      \item (\cite{Todorcevic86}) Every collection $\mc{F}$ of $\omega_1$ many functions from natural numbers to natural numbers has a $<^*$ upper bound.\footnote{I.e., there is a function $g$ such that for every $f\in \mc{F}$, $f(n)<g(n)$ for all but finitely many $n$.}
      \item (\cite{Todorcevic91}) Every Aronszajn tree is special.
      \end{itemize}
      Then, the following statements follow from $\ms{K}_3$.
      \begin{itemize}
\item (\cite{TV}; \cite{Todorcevic89}) $2^\omega=2^{\omega_1}$.
\item (\cite{Moore2000}) The union of $\omega_1$ many Lebesgue null sets is null.
      \end{itemize}
      Finally, there is a consequence of $\ms{K}_{n+1}$ for general $n\geq 2$.
        \begin{itemize}
     \item (\cite{TV}) Every ccc poset of size $\omega_1$ is $\sigma$-$n$-linked.
       \end{itemize}
       A poset is \emph{$\sigma$-$n$-linked} if it is a countable union of $n$-linked subsets.
      
      However,  no approach to the first type of questions seems to be satisfactory. In other words, there is a large gap between the known  upper bound and lower bound of $\ms{K}_n$'s. On one hand, the discovered lower bounds, i.e., consequences of $\ms{K}_n$'s not in form of fragments of MA$_{\omega_1}$, have strengths much weaker than MA$_{\omega_1}$. On the other hand, no upper bound other than MA$_{\omega_1}$ itself was known.
      
     A possible proof of MA$_{\omega_1}$ from $\ms{K}_n$ seems to need a coding of centered subsets of a ccc poset by $n$-linked subsets of another ccc poset. In other words, a coding that codes 0-homogeneous subsets of a finitary coloring by 0-homogeneous subsets of another $n$-ary coloring is   needed.     
     An intuitive  guess  is that such coding does not exist and that $\ms{K}_n$'s are strictly weaker than MA$_{\omega_1}$. 
     However, our investigation ends up being just the opposite of this expectation.
     \begin{thm}\label{thm K3toMA}
     $\ms{K}_3$ implies {\rm MA}$_{\omega_1}$.
     \end{thm}
       The proof of above theorem consists of the following steps.
    \begin{enumerate}[($\bigstar$1)]
    \item (\cite{TV}) $\ms{K}_3$ implies that every ccc poset of size $\omega_1$ is $\sigma$-linked.
    \item Assume ${\rm MA}_{\omega_1}(\sigma$-centered). For every $\sigma$-linked poset $\mc{P}$, there is another $\sigma$-linked poset $\mc{Q}$ such that every uncountable 3-linked subset of $\mc{Q}$ induces an uncountable centered subset of $\mc{P}$.
    \item (\cite{TV}) MA$_{\omega_1}$ is equivalent to the statement that every uncountable ccc poset has an uncountable centered subset.
    \item $\ms{K}_3$ implies $\mathfrak{t}>\omega_1$.
    \item (\cite{Bell}; \cite{Rothberger}, \cite{MS}) ${\rm MA}_{\omega_1}(\sigma$-centered) is equivalent to $\mathfrak{t}>\omega_1$.
    \end{enumerate}

  ($\bigstar$1) and ($\bigstar$3).  The proof is based on Todorcevic's minimal walk technique \cite{Todorcevic87}  that codes all $\omega_1$ colors by pairs of an arbitrary uncountable subset. Then in \cite{TV}, a coloring is designed for the coding structure of ccc posets.
    
    ($\bigstar$5). This is a well-known consequence of the following famous theorems:   Bell's Theorem \cite{Bell} that ${\rm MA}_{\kappa}(\sigma$-centered) is equivalent to $\mathfrak{p}>\kappa$; Rothberger's Theorem \cite{Rothberger} that $\mathfrak{p}>\omega_1$ iff $\mathfrak{t}>\omega_1$ (or   Malliaris-Shelah's Theorem \cite{MS} that $\mathfrak{p}=\mathfrak{t}$).
    
   ($\bigstar$2).  This property follows from a more general fact (see Proposition \ref{prop code fin by 4}) indicating a general phenomenon.  
   A general damage control structure explored in Section 7 plays an essential role in the discovery of this general property.

($\bigstar$4). This result follows from the combinatorial analysis of $\omega_1$-towers given  in Section 3.
This investigation   leads us to a new structure theory necessary for filling the tower. 

The proofs of ($\bigstar$2) and ($\bigstar$4) are likely to be useful in further investigations of fragments of Martin's axiom. For example, the proof of ($\bigstar$4) can be generalized to cardinal $\kappa$ greater than $\omega_1$. Together with Bell's Theorem and Malliaris-Shelah's Theorem, an appropriate fragment of MA$_\kappa$ would imply MA$_\kappa$($\sigma$-centered).

\subsection{Fragments above property {\rm K}}

By Theorem \ref{thm K3toMA}, strengthening ccc to K$_3$ is as strong as strengthening ccc to precaliber $\omega_1$. This seems to indicate   that K$_3$ is so powerful that can fulfill the same goals as precaliber $\omega_1$. 

To explore this speculation,  we look into the general case. Then further investigation actually indicates
 that, the strengths between K$_n$  and K$_{n+1}$ is so significant that establishing equivalence between them is as powerful as having the corresponding forcing axiom.  For example, if we view    the fragment of $\ms{K}_3$ that 
\[\text{every poset with property K has property K$_3$}\]
 as a fragment of  MA$_{\omega_1}$(K), then this fragment is so powerful that it implies the full MA$_{\omega_1}$(K). Then $\ms{K}_3$ implies $\ms{K}_2$+MA$_{\omega_1}$(K) which is clearly equivalent to  MA$_{\omega_1}$.
In fact, this phenomenon is   generally true.

\begin{thm}\label{thm K3toK4}
Assume $n\geq 2$. If every poset with property {\rm K}$_n$ has property {\rm K}$_{n+1}$, then {\rm MA$_{\omega_1}$(K$_n$)} holds.
\end{thm}

To clarify the general phenomenon, we list   well-known properties stronger than property K  in the following figure (see Definition \ref{poset} for precise definitions).
 The arrows denote implications and no implication is reversible (see, e.g., \cite[Section 3]{Bagaria}).

    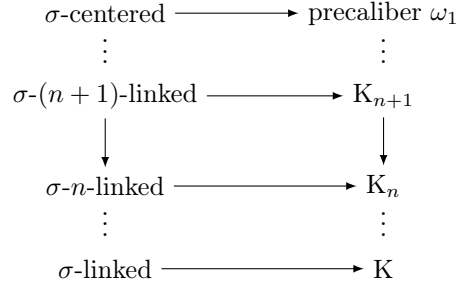
\begin{figure}[h]
\begin{tikzpicture}[auto, node distance=1.2 cm, >=latex]
\node  (centered) {$\sigma$-centered};
\node  (precaliber) [right of=centered, xshift=2.5cm] {precaliber $\omega_1$};
\node (dot11) [below of=centered, yshift=0.8cm] {$\vdots$};
\node (dot12) [below of=precaliber, yshift=0.8cm] {$\vdots$};
\node (n+1linked) [below of=dot11, yshift=0.5cm] {$\sigma$-$(n+1)$-linked};
\node (Kn+1) [below of=dot12, yshift=0.5cm] {K$_{n+1}$};
\node (nlinked) [below of=n+1linked]  {$\sigma$-$n$-linked};
\node (Kn) [below of=Kn+1] {K$_n$};
\node (dot21) [below of=nlinked, yshift=0.8cm] {$\vdots$};
\node (dot22) [below of=Kn, yshift=0.8cm] {$\vdots$};
\node (linked) [below of=dot21, yshift=0.5cm] {$\sigma$-linked};
\node (K2) [below of=dot22, yshift=0.5cm] {K};

\draw [->] (centered) -- (precaliber);
\draw [->] (n+1linked) -- (Kn+1);
\draw [->] (n+1linked) -- (nlinked);
\draw [->] (Kn+1) -- (Kn);
 \draw [->] (nlinked) -- (Kn);
 \draw [->] (linked) -- (K2);

\end{tikzpicture}
\caption{Forcing properties}\label{figure1}
\end{figure}

The  $\sigma$-$n$-linked property plays an important role in the proof of Theorem \ref{thm K3toK4}. Moreover,  Theorem \ref{thm K3toK4} remains true if we replace property K$_n$ by $\sigma$-$n$-linked.
\begin{thm}\label{thm sK3toK4}
Assume $n\geq 2$. If every $\sigma$-$n$-linked poset  has property {\rm K}$_{n+1}$, then {\rm MA}$_{\omega_1}$($\sigma$-$n$-linked) holds.
\end{thm}

We will use the following notations for fragments of MA$_{\omega_1}$.
\begin{defn}
 {\rm P}($\Phi\ra \Psi$) is the assertion that every poset with property $\Phi$ has property $\Psi$ where   $\Phi$ and $\Psi$ are properties for posets.
 
  If $\Psi$ is $\sigma$-centered or $\sigma$-$n$-linked, then {\rm P}$_{\omega_1}$($\Phi\ra \Psi$) is the assertion that every poset, of size $\omega_1$, with property $\Phi$ has property $\Psi$.
\end{defn}
Note that P(precaliber $\omega_1\ra \sigma$-linked) is false: Precaliber $\omega_1$ is closed under arbitrary product with finite support and more than continuum many product of non-trivial posets is not $\sigma$-linked. So the subscript $\omega_1$ in above notation is necessary when $\Psi$ is $\sigma$-centered or $\sigma$-$n$-linked.

On the other hand, strengthening K$_n$ to $\sigma$-$n$-linked is not so strong.
\begin{thm}\label{thm K3tosK3}
For $n\geq 2$,  {\rm P$_{\omega_1}$(K$_n\ra \sigma$-$n$-linked)} does not imply {\rm MA$_{\omega_1}$(K$_n$)}.
\end{thm}

Finally, we show that strengthening precaliber $\omega_1$ to $\sigma$-centered induces the corresponding forcing axiom.
\begin{thm}\label{thm ptosc}
{\rm P$_{\omega_1}$(precaliber $\omega_1\ra \sigma$-centered)} implies {\rm MA$_{\omega_1}$(precaliber $\omega_1$)}.
\end{thm}
This is not surprising since by Theorem \ref{thm BRMS} (see ($\bigstar$5)),     above theorem   is equivalent to the following property on cardinal invariant $\mathfrak{t}$ (see Theorem \ref{thm ptosK3 t}).
\begin{itemize}
\item {\rm P}$_{\omega_1}$(precaliber $\omega_1\ra \sigma$-centered) implies $\mathfrak{t}>\omega_1$.
\end{itemize}

We summarize the results mentioned above. Assume $n\geq 2$.
\begin{enumerate}[(i)]
\item  {\rm P}({\rm K}$_n\ra$ K$_{n+1}$) implies {\rm MA}$_{\omega_1}$({\rm K}$_n$).
\item   {\rm P}$(\sigma$-$n$-linked$\ra{\rm K}_{n+1}$) implies {\rm MA}$_{\omega_1}$($\sigma$-$n$-linked). In particular, {\rm P}$_{\omega_1}(\sigma$-$n$-linked$\ra\sigma$-$(n+1)$-linked) implies {\rm MA}$_{\omega_1}$($\sigma$-$n$-linked).
\item {\rm P}$_{\omega_1}$(precaliber $\omega_1\ra \sigma$-centered) implies {\rm MA}$_{\omega_1}$(precaliber $\omega_1$).
\item   {\rm P}$_{\omega_1}$({\rm K}$_n\ra \sigma$-$n$-linked) does not imply {\rm MA}$_{\omega_1}$({\rm K}$_n$).
\end{enumerate}

We indicate above results in     Figure \ref{figure2}  where a dashed arrow from $\Phi$ to $\Psi$ denotes that the fragment P($\Phi\ra\Psi$) (or P$_{\omega_1}$($\Phi\ra\Psi$)) of MA$_{\omega_1}(\Phi)$ has the same strength as MA$_{\omega_1}(\Phi)$.

    \begin{figure}[h]
\begin{tikzpicture}[auto, node distance=1.2 cm, >=latex]
\node  (centered) {$\sigma$-centered};
\node  (precaliber) [right of=centered, xshift=2.5cm] {precaliber $\omega_1$};
\node (dot11) [below of=centered, yshift=0.8cm] {$\vdots$};
\node (dot12) [below of=precaliber, yshift=0.8cm] {$\vdots$};
\node (n+1linked) [below of=dot11, yshift=0.5cm] {$\sigma$-$(n+1)$-linked};
\node (Kn+1) [below of=dot12, yshift=0.5cm] {K$_{n+1}$};
\node (nlinked) [below of=n+1linked]  {$\sigma$-$n$-linked};
\node (Kn) [below of=Kn+1] {K$_n$};
\node (dot21) [below of=nlinked, yshift=0.8cm] {$\vdots$};
\node (dot22) [below of=Kn, yshift=0.8cm] {$\vdots$};
\node (linked) [below of=dot21, yshift=0.5cm] {$\sigma$-linked};
\node (K2) [below of=dot22, yshift=0.5cm] {K};

\draw [->] (centered) -- (precaliber);
\draw [->] (n+1linked) -- (Kn+1);
\draw [->] (n+1linked) -- (nlinked);
\draw [->] (Kn+1) -- (Kn);
 \draw [->] (nlinked) -- (Kn);
 \draw [->] (linked) -- (K2);
 
 \draw[->, dashed] (2.5, 0.15) -- (0.9, 0.15);
 \draw[->,dashed] (0.2, -2) -- (0.2, -1.4);
  \draw[->, dashed] (0.3, -2) -- (3.4, -1.4);
 \draw[->, dashed] (3.9, -2) -- (3.9, -1.4);

\end{tikzpicture}
\caption{Forcing properties' strength}\label{figure2}
\end{figure}
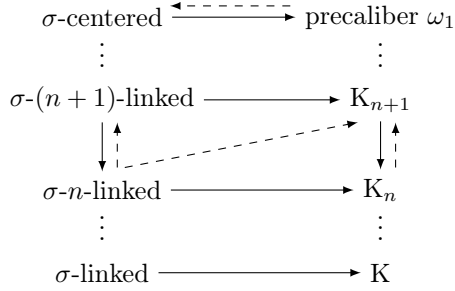

By   results (i)-(iv) summarized above, the diagram is complete.

\subsection{Ramsey-type  fragments}

Ramsey-type fragments of MA$_{\omega_1}$     are also well studied in the literature (see, e.g., \cite{TV}, \cite{Todorcevic91}, \cite{LT2001}).

A coloring $c: [\omega_1]^{<\omega}\ra 2$ is \emph{ccc} if $c^{-1}\{0\}$ is downward closed, uncountable and every uncountable collection of finite 0-homogeneous subsets contains two elements whose union is 0-homogeneous. 
 For $n\geq 2$, a coloring $c: [\omega_1]^n\ra 2$ is \emph{ccc} if $c^{-1}\{0\}$ is uncountable and every uncountable collection of finite 0-homogeneous subsets contains two elements whose union is 0-homogeneous.

We use the notation from \cite{LT2002} to denote the Ramsey type fragments.
For $n\geq 2$, $\mc{K}_n$ is the assertion that every ccc coloring $c: [\omega_1]^n\ra 2$ has an uncountable 0-homogeneous subset.

For   the exact strengths   of Ramsey-type fragments,  a natural proposal for $\mc{K}_n$'s is that the following question (see \cite{Yorioka10})  may have a positive answer.
\begin{itemize}
\item Is a Ramsey-type fragment of MA$_{\omega_1}$ equivalent to the corresponding version of poset-type fragment?
\end{itemize}
This proposal is supported by the following transformation between ccc posets  and the corresponding finitary ccc colorings.
\begin{itemize}
\item $\mc{P}$ has an uncountable centered subset iff $c: [\mc{P}]^{<\omega}\ra 2$ has an uncountable 0-homogeneous subset where $c^{-1}\{0\}$ is the collection of finite centered subsets.
\end{itemize}
This equivalence induces the equivalence between the following two statements.
\begin{itemize}
\item Every uncountable ccc poset has an uncountable centered subset.
\item Every ccc coloring $c: [\omega_1]^{<\omega}\ra 2$ has an uncountable 0-homogeneous subset.
\end{itemize}
By \cite{TV} (see ($\bigstar$3)), each of above two statements is equivalent to MA$_{\omega_1}$.

However, for a fixed $n\geq 2$, there is no obvious equivalence between the Ramsey-type fragments and poset-type fragments. In fact, no equivalence, for a fixed natural number $n\geq 2$, was proved or disproved.


Recall the productivity of ccc which has its own origin and interest. We use the notation from \cite{Galvin}.\footnote{We would like to reserve the notation $\mcc^2$ for other purposes.}
\[C(\omega_1)\text{ is the assertion that the product of two ccc posets is ccc}.\] 
 According to the computation in \cite[pp 837]{Todorcevic91} (or the proof of Lemma \ref{lem 6.2}), $\ms{K}_n$ is equivalent to $\mc{K}_n+C(\omega_1)$ for $n\geq 2$. So the following specific question might be a first step of the proposal.
\begin{itemize}
\item Does $\mc{K}_2$ imply $C(\omega_1)$?
\end{itemize}
This question is closely related to  questions \cite[(7)-(8)]{Todorcevic91}   about the exact strength of $\mc{K}_2$:
\begin{itemize}
\item Does $\mc{K}_2$ imply MA$_{\omega_1}$?
\item Does $\mc{K}_2$ imply $\mc{K}_3$?
\end{itemize}
We answer three of the above questions in negative by proving the following theorem.\footnote{Note that $\mc{K}_2$+$C(\omega_1)$+MA$_{\omega_1}$(K) is equivalent to MA$_{\omega_1}$.}
\begin{thm}\label{thm k2 not k3}
It is relative consistent with {\rm ZFC} that $\mc{K}_2$ and {\rm MA$_{\omega_1}$(K)} hold, and $\mc{K}_3$ fails.
\end{thm} 
In \cite{Peng25}, we introduced an iteration method of  minimizing damage to a strong binary coloring.  We take out the structure that controls damage and introduce its general form for higher arity colorings in Section 7. The general damage control structure follows the framework but has one major difference from the specific one in \cite{Peng25}. This major difference eventually leads to a proof of above theorem.


We also explore  the general damage control structure's properties    in Section 7 since this structure and its variations are likely to be useful in further investigations. 



\subsection{Contents and organization}

The proof of Theorem \ref{thm K3toMA} that $\ms{K}_3$ implies MA$_{\omega_1}$, consists of two parts and are presented in Subsection 3.1 and Section 4. 

In Section 3, we prove $\mathfrak{t}>\omega_1$ from $\ms{K}_3$ and several other assumptions. 
A new structure is  transformed from a tower to fill the tower.  The forcing   and coloring methods are explored to address different issues.

In Section 4, we show that under MA$_{\omega_1}$($\sigma$-centered), if every $\sigma$-$n$-linked poset has property K$_{n+1}$, then every $\sigma$-$n$-linked poset has precaliber $\omega_1$. Moreover, a more general phenomenon holds: MA$_{\omega_1}$($\sigma$-centered)
reduces arity for $\sigma$-$n$-linked colorings (see Proposition \ref{prop code fin by 4}).

In Section 5, we generalize the construction in \cite{TV}. The generalization and results in Sections 3-4 prove Theorem \ref{thm K3toK4} and Theorem \ref{thm sK3toK4} that describe general phenomenons on fragments of Martin's axiom. 

In Section 6, we prove Theorem \ref{thm K3tosK3} that, the fragment P$_{\omega_1}$(K$_n\ra\sigma$-$n$-linked) of MA$_{\omega_1}$(K$_n$) is strictly weaker than MA$_{\omega_1}$(K$_n$). 


In Section 7, we introduce the general form of a damage control structure   generalizing a specific form   in \cite{Peng25}. 
 We then analyze   properties of the damage control structure. In Section 8, we apply the damage control structure to distinguish P$_{\omega_1}$(K$_n\ra\sigma$-$n$-linked) and P$_{\omega_1}$(K$_{n+1}\ra\sigma$-$(n+1)$-linked). In Section 9, we apply the damage control structure to prove Theorem \ref{thm k2 not k3}. In particular, the Ramsey-type fragment $\mc{K}_2$ is strictly weaker than MA$_{\omega_1}$.

      \section{Preliminary}

      We first describe notations throughout this paper.
      
       $\subset$ and $\supset$ are strict and denote $\subsetneqq$ and $\supsetneqq$ respectively.     For $x, y\subseteq \omega$, $x\subseteq^* y$ if $y\setminus x$ is finite.        
       For a set $X$ and a cardinal $\kappa$, 
       \[[X]^\kappa=\{Y\subseteq X: |Y|=\kappa\}\text{ and }[X]^{<\kappa}=\{Y\subseteq X: |Y|<\kappa\}.\]
       
       A \emph{tower} is a family $T=\{t_\alpha\in [\omega]^\omega: \alpha<\kappa\}$ for some cardinal $\kappa$ such that $t_\beta\subseteq^* t_\alpha$ for all $\alpha<\beta<\kappa$.  A tower $T=\{t_\alpha\in [\omega]^\omega: \alpha<\kappa\}$ is \emph{filled} if for some $t\in [\omega]^\omega$, $t\subseteq^* t_\alpha$ for all $\alpha<\kappa$. The \emph{tower number} $\mathfrak{t}$ is the least size of an unfilled tower.

    For a partial order $(P, \leq_P)$, say $p\in P$ is the \emph{$\leq_P$-least} (\emph{$\leq_P$-largest}) element if $p\leq_P q$ ($p\geq_P q$) for all $q\in P$. Say $p\in P$ is a \emph{$\leq_P$-minimal} (\emph{$\leq_P$-maximal}) element if there is no $q\in P$ with $p>_P q$ ($p<_P q$) where $<_P$ is the strict part of $\leq_P$. 
    
        For a function $f$ and a set $X\subseteq dom(f)$, denote $f[X]=\{f(x): x\in X\}$.
        
   A set of ordinals $X$ is identified with the increasing sequence which enumerates $X$.  For $\alpha$ less than $otp(X)$, the order type of $X$, and a subset $I\subseteq otp(X)$,   
   \[X(\alpha)\text{ is the $\alpha$th element  in the increasing enumeration of }X\]
   \[\text{and }X[I]=\{X(\beta): \beta\in I\}.\]

The following notions of posets are standard (see, e.g.,  \cite{Barnett}, \cite{Bagaria}).
\begin{defn}\label{poset}
Suppose $\pp$ is a poset and $2\leq n<\omega$.
\begin{itemize}
\item $\pp$ has \emph{property {\rm K}$_n$} if every uncountable subset of $\pp$ has an uncountable subset that is $n$-linked.   A subset $X$ of $\pp$ is \emph{$n$-linked} if every $n$-element subset of $X$ has a common lower bound.  
\item $\pp$ is \emph{$\sigma$-$n$-linked} if $\pp$ is a countable union of $n$-linked subsets.  
\item $\pp$ has \emph{precaliber $\omega_1$} if every uncountable subset of $\pp$ has an uncountable centered subset. A subset $X$ of $\pp$ is \emph{centered} if every finite subset of $X$ has a common lower bound.
\item $\pp$ is \emph{$\sigma$-centered} if $\pp$ is a countable union of centered subsets.
\end{itemize}
In above concepts, we omit $n$ if $n=2$.
\end{defn}

\begin{defn}\label{defn iso}
Suppose  $\theta$ is an ordinal, $X, Y\subseteq \theta $, $k<\omega$ and $\langle f_i: i<k\rangle$, $\langle g_i: i<k\rangle$ are sequences of functions  such that all $f_i: [\theta]^{<\omega}\ra \omega$ and all $g_i: [\theta]^{<\omega}\ra \omega$ are partial maps. Say $(X,   f_i: i<k)$ \emph{is isomorphic to} $(Y,  g_i: i<k)$ if $X, Y$ have the same order type and for every $i<k$ and $I\in [otp(X)]^{<\omega}$,  $f_i(X[I])=g_i(Y[I])$. 
\end{defn}
In above definition, $f_i(X[I])=g_i(Y[I])$ asserts that $f_i(X[I])$ and $g_i(Y[I])$ are either both undefined or both defined and equal each other.

For a function $c: [\omega_1]^2\ra \omega$,  
\begin{itemize}
\item we write $c(\alpha, \beta)$ instead of $c(\{\alpha, \beta\})$ when $\alpha<\beta$ for notation simplicity;\item  for $\beta<\omega_1$,
$c_\beta$ is the function from $\beta$ to $\omega$ defined by $c_\beta(\alpha)=c(\alpha, \beta)$.
\end{itemize}

Coherent functions have played   important roles in analyzing combinatorial properties of structures  since Todorcevic introduced the technique of minimal walk in \cite{Todorcevic87}. 
\begin{defn}
A function $c: [\omega_1]^2\ra \omega$ is \emph{coherent} if for every $\alpha<\beta<\omega_1$, 
\[D_{\alpha\beta}=\{\xi<\alpha: c_\alpha(\xi)\neq c_\beta(\xi)\} \text{ is finite.}\]

\end{defn}

Throughout the paper, we will fix a coherent function $e: [\omega_1]^2\ra \omega$ with additional properties and use the function to construct different colorings.
The function defined in \cite[Definition 3.2.1]{Todorcevic07} is a witness of the following lemma (see \cite[Lemma 3.2.2, Lemma 3.2.3]{Todorcevic07}).
\begin{lem}[\cite{Todorcevic07}]\label{lem coh}
There is a coherent  function $e: [\omega_1]^2\ra \omega$ satisfying the following properties.
\begin{enumerate}[({\rm coh}1)]
\item For every $\beta<\omega_1$, $e_\beta$ is one-to-one.
\item For every $\alpha<\beta<\gamma<\omega_1$, 
\[e(\alpha, \beta)\leq \max\{e(\alpha, \gamma), e(\beta, \gamma)\}\text{ and }e(\alpha, \gamma)\leq \max\{e(\alpha, \beta), e(\beta, \gamma)\},\]
 or equivalently, $e(\beta, \gamma)\geq \max\{e(\alpha, \beta), e(\alpha, \gamma)\}$ whenever $e(\alpha, \beta)\neq e(\alpha, \gamma)$.
 \item For $\alpha<\beta<\gamma<\omega_1$, $e(\alpha, \beta)\neq e(\beta, \gamma)$.
\end{enumerate}
\end{lem}

Note that the coherence of $e$ follows from properties (coh1)-(coh2). To see this, fix $\beta<\gamma$. Then $D_{\beta\gamma}=\{\alpha<\beta: e(\alpha, \beta)\neq e(\alpha, \gamma)\}$ is contained in $e_\beta^{-1}[e(\beta, \gamma)+1]$ by (coh2) and hence is finite by (coh1).

Another way to get a witness of Lemma \ref{lem coh} is to first construct a function $e'$ satisfying (coh1)-(coh2): Construct $e'_\gamma$, by induction on $\gamma$, such that (coh1)-(coh2) and the following additional property are satisfied.
\begin{itemize}
\item For every $\gamma<\omega_1$, $rang(e'_\gamma)\cap [2^m, 2^{m+1})=\emptyset$ for infinitely many $m$.
\end{itemize}
Then a slight modification of $e'$ will induce a coherent function $e$ satisfying (coh1)-(coh3).

The following property induced from the function $e$ is straightforward and will be frequently used.
\begin{lem}\label{lem initial}
Suppose $a, b\in [\omega_1]^{<\omega}$ and $(a, e)$ is isomorphic to $(b, e)$. Then $|a|=|b|$ and $a\cap b$ is an initial segment of both $a$ and $b$.
\end{lem}
\begin{proof}
$|a|=|b|$ follows immediately from the isomorphism. We now show that $a\cap b$ is an initial segment of both $a$ and $b$. We may assume that $a\cap b\neq \emptyset$. Let 
\[\beta=\max(a\cap b)\text{ and $i, j<|a|$ be such that }\beta=a(i)=b(j).\]

We first show that $i=j$. 
Suppose otherwise. Say $i<j$. Then $b(i)<b(j)=a(i)<a(j)$. By isomorphism, $e(b(i), b(j))=e(a(i), a(j))$. But this contradicts (coh3) and the fact that $b(j)=a(i)$.

Now it suffices to show that $a(k)=b(k)$ for $k<i$. Fix $k<i$. By isomorphism, $e(a(k), a(i))=e(b(k), b(i))$.
By (coh1) and the fact $a(i)=\beta=b(i)$, $a(k)=b(k)$.
\end{proof}

The following two simple facts will be frequently used.
\begin{lem}\label{lem initial2}
Suppose for some $n<\omega$ and $\{a_i\in [\omega_1]^{<\omega}: i<n\}$, $(a_i, e)$ is isomorphic to $(a_j, e)$ for all $i<j<n$.
\begin{enumerate}
\item For $I, J\subseteq |a_0|$, $a_0[I]\cap a_1[J]=a_0[I\cap J\cap |a_0\cap a_1|]$. In particular, $a_0[I]\cap a_1[J]=\emptyset$ if $I\cap J=\emptyset$.
\item For $I\subseteq |a_0|$ and  $\langle J_j\subseteq |a_0|: j<n\rangle$, if $a_0[I]\subseteq \bigcup_{j<n} a_j[J_j]$, then $I\subseteq \bigcup_{j<n} J_j$.
\end{enumerate}
\end{lem}
\begin{proof}
(1) Let $k=|a_0\cap a_1|$. Then by Lemma \ref{lem initial}, $a_0\cap a_1=a_0[k]=a_1[k]$. So
\begin{align*}
a_0[I]\cap a_1[J]=& a_0[I]\cap (a_0\cap a_1)\cap a_1[J]\\
= &a_0[I\cap k]\cap a_1[J\cap k]\\
= &a_0[I\cap k]\cap a_0[J\cap k]=a_0[I\cap J\cap k].
\end{align*}

(2) Fix $i\in I$. Choose $j<n$ such that 
\begin{enumerate}[(i)]
\item $a_0(i)\in a_j[J_j]$.
\end{enumerate}
Now, $a_0(i)\in a_0\cap a_j$. By Lemma \ref{lem initial}, $a_0(i)=a_j(i)$. Together with (i), $i\in J_j$.
\end{proof}

We will also need the following characterizations of $\mathfrak{p}$ and $\mathfrak{t}$.
\begin{thm}\label{thm BRMS}
\begin{enumerate}
\item (\cite{Bell}) {\rm MA}$_{\kappa}(\sigma$-centered) is equivalent to $\mathfrak{p}>\kappa$.
\item (\cite{Rothberger}) $\mathfrak{p}>\omega_1$ iff $\mathfrak{t}>\omega_1$.
\item (\cite{MS}) $\mathfrak{p}=\mathfrak{t}$.
\end{enumerate}
\end{thm}

\section{Constructing colorings and posets   from a tower}

In this section, we prove that $\mathfrak{t}>\omega_1$ follows from $\ms{K}_3$, as well as several other assumptions. Throughout this section, we fix the following notation.

\textbf{Notation.} For $I\in [\omega]^\omega$ and $n<\omega$, denote 
\[I_n=[I(n), I(n+1)).\]




Several colorings will be introduced in this section. We first introduce the following notation for colorings.
\begin{defn}
\begin{enumerate}[(i)]
\item For $2<N\leq \omega$, a coloring $\pi: [\omega_1]^{<N}\ra 2$ is \emph{downward closed} if 
\[(a\subseteq b\wedge\pi(b)=0)\ra\pi(a)=0.\]
 For downward closed $\pi$, $\mc{H}_0^\pi$ is the collection of finite 0-homogeneous subsets ordered by reverse inclusion.
 \item For $n\geq 2$, $i<2$ and a coloring $\pi: [\omega_1]^n\ra 2$, 
$\mc{H}_i^\pi$ is the collection of finite $i$-homogeneous subsets ordered by reverse inclusion.
  A finite set $a\in [\omega_1]^{<\omega}$ is \emph{$i$-homogeneous} if either $(|a|\geq n\wedge [a]^n\subseteq \pi^{-1}\{i\})$ or $a\subseteq b$ for some  $b\in \pi^{-1}\{i\}$.
\end{enumerate}
\end{defn}

\textbf{Remark.} Colorings of forms $\pi: [\omega_1]^{<\omega}\ra 2$ and $\pi: [\omega_1]^n\ra 2$ are frequently used in the literature. For a coloring $\pi: [\omega_1]^n\ra 2$, it is also common to define 0-homogeneity so that elements in $[\omega_1]^{<n}$ are all 0-homogeneous. 
Here, for coloring on $n$-ary subsets of $\omega_1$, we require, 
\[\exists b\in \pi^{-1}\{i\} ~ a\subseteq b,\]
 for $a\in [\omega_1]^{<n}$ to be $i$-homogeneous.
For colorings defined in a constructive way, these two requirements usually do not make any difference (see, e.g., the coloring in Theorem \ref{thm K3 t} below).
 We want to point out that our additional requirement is  natural, especially for   colorings   closely related to other posets (see, e.g., Section 4).  Consider a poset $(\omega_1, \prec)$ with property K$_3$. We define $\pi: [\omega_1]^3\ra 2$ by $\pi(a)=0$ iff $a$ has a common lower bound in $(\omega_1, \prec)$. Then an uncountable 0-homogeneous subset of $\pi$ is a 3-linked subset of $(\omega_1, \prec)$. Without the additional requirement, $\mc{H}_0^\pi$ is in general not ccc even if $(\omega_1, \prec)$ has property K$_3$. For example, $\{a_\alpha: \alpha<\omega_1\}$ is an uncountable antichain if each $a_\alpha$ is a pair of incompatible conditions.
 
In above definition, colorings in form of (ii) can be viewed as   colorings in form of (i). For example, suppose $\pi: [\omega_1]^n\ra 2$ is a coloring for $n\geq 2$. Then define $\pi': [\omega_1]^{<n+1}\ra 2$ by
\[\pi'(a)=0\text{ iff } \pi(b)=0 \text{ for some } b\supseteq a.\]
Then $\pi'$ is downward closed and $\mc{H}^\pi_0=\mc{H}^{\pi'}_0$.

 \subsection{$\ms{K}_3$ implies $\mathfrak{t}>\omega_1$}
 
 In this subsection, we  prove that $\ms{K}_3$ implies $\mathfrak{t}>\omega_1$ by the coloring method. The proof illustrates the structure we shall use to fill towers.
 
 
 For a tower $T=\{t_\alpha: \alpha<\omega_1\}$, an equivalent way to fill $T$ is to find    $\Gamma\in [\omega_1]^{\omega_1}$ such that $\bigcap_{\alpha\in \Gamma} t_\alpha$ is infinite. The difficulty is to guarantee that the intersection is infinite.  
 
 First recall by \cite{Todorcevic86}, $\ms{K}_2$ implies $\mathfrak{b}>\omega_1$.
Replacing $T$ by  a sub-tower,  we may assume that there exists $I\in [\omega]^\omega$ such that $t_\alpha\cap I_n\neq\emptyset$ for all $\alpha$ and all $n$. 

We will investigate  the following type of structures:
\[\langle (\{t_\alpha\cap I_n: \alpha\in \Gamma\}, \subseteq): n\in A\rangle\]
where $A\subseteq \omega$ and $\Gamma\subseteq \omega_1$.
And our goal is:
 \begin{itemize}
\item Find $A\in [\omega]^\omega$ and $\Gamma\in [\omega_1]^{\omega_1}$ such that for every $n\in A$, 
\[(\{t_\alpha\cap I_n: \alpha\in \Gamma\}, \subseteq)\text{ is linearly ordered}.\]
 \end{itemize}
 If above goal is achieved, then for $n\in A$, $\bigcap_{\alpha\in \Gamma} t_\alpha\cap I_n$ is the $\subseteq$-least element of $\{t_\alpha\cap I_n: \alpha\in \Gamma\}$ which is non-empty. Consequently, $\bigcap_{\alpha\in \Gamma} t_\alpha$ is infinite.

\begin{thm}\label{thm K3 t}
$\ms{K}_3$ implies $\mathfrak{t}>\omega_1$.
\end{thm}
\begin{proof}
Fix a tower $T=\{t_\alpha: \alpha<\omega_1\}$.
By \cite{Todorcevic86}(see also Proposition \ref{prop sK3toK4 b}), $\mathfrak{b}>\omega_1$. Replacing $T$ by a sub-tower if necessary,  find $I\in [\omega]^\omega$ such that 
\begin{enumerate}
\item $t_\alpha\cap I_n\neq\emptyset$ for all $\alpha$ and all $n$.
\end{enumerate}

Define a coloring $\pi: [\omega_1]^3\ra 2$ by for $\alpha<\beta<\gamma$, $\pi(\alpha, \beta, \gamma)=0$ iff
\[e(\alpha, \beta)=e(\alpha, \gamma)\ra (t_\beta\cap I_{e(\alpha, \beta)}\subseteq t_\gamma \vee t_\gamma\cap I_{e(\alpha, \beta)}\subseteq t_\beta).\]
Before proceeding to the proof, we would like to point out that the positions of $\beta$ and $\gamma$ are symmetric. So the order between $\beta$ and $\gamma$ is not important and we can simply say that $\alpha<\min\{\beta, \gamma\}$.\medskip

\textbf{Claim 1.} $\mc{H}^\pi_0$ is ccc.
\begin{proof}[Proof of Claim 1.]
Fix $\{p_\alpha\in \mc{H}^\pi_0: \alpha<\omega_1\}$. Find $\Sigma\in [\omega_1]^{\omega_1}$ such that for some $n, N<\omega$,
\begin{enumerate}\setcounter{enumi}{1}
\item $\{p_\alpha: \alpha\in \Sigma\}$ forms a $\Delta$-system with root $\overline{p}$;
\item  for $\alpha\in \Sigma$, $|p_\alpha|=n$ and  for $\xi<\zeta$ in $p_\alpha$,
\begin{itemize}
\item $e(\xi, \zeta)<N$;
\item $t_\zeta\setminus N\subseteq t_\xi$;
\end{itemize}
\item  for $\alpha<\beta$ in $\Sigma$,
\begin{itemize}
\item $(p_\alpha, e)$ is isomorphic to $(p_\beta, e)$;\footnote{See Definition \ref{defn iso}.}
\item  for $i<n$, $t_{p_\alpha(i)}\cap I(N)=t_{p_\beta(i)}\cap I(N)$.
\end{itemize}
\end{enumerate}

We may assume $n\geq 3$. 
Now find $\xi<\zeta$ in $\Sigma$ such that 
\begin{enumerate}\setcounter{enumi}{4}
\item $\max(p_\xi)<\min(p_\zeta\setminus \overline{p})$ and $e(\alpha, \beta)>N$ for $\alpha\in p_\xi\setminus \overline{p}$ and $\beta\in p_\zeta\setminus \overline{p}$.
\end{enumerate}
We will show that $p_\xi\cup p_\zeta\in \mc{H}^\pi_0$. Fix $\alpha<\beta<\gamma$ in $p_\xi\cup p_\zeta$. It suffices to show that $\pi(\alpha, \beta, \gamma)=0$. The non-trivial case is
\begin{enumerate}\setcounter{enumi}{5}
\item $e(\alpha,\beta)=e(\alpha, \gamma)$.
\end{enumerate}
We discuss by cases.\medskip

\textbf{Case 1.} $\alpha\in \overline{p}$.\medskip

The non-trivial case is $\beta\in p_\xi\setminus \overline{p}$ and $\gamma\in p_\zeta\setminus \overline{p}$. Note by (3), $e(\alpha, \beta)<N$.

Let $i<n$ and $\gamma^*\in p_\xi\setminus \overline{p}$ be such that $\gamma=p_\zeta(i)$ and $\gamma^*=p_\xi(i)$. 

If $\beta=\gamma^*$, then by (4) and the fact $e(\alpha, \beta)<N$, 
\[t_\beta\cap I_{e(\alpha, \beta)}=t_{\gamma^*}\cap  I_{e(\alpha, \beta)}=t_\gamma\cap  I_{e(\alpha, \beta)}.\] Hence, $\pi(\alpha, \beta, \gamma)=0$.

Now suppose $\beta\neq \gamma^*$. By (6) and (4),
\[e(\alpha,\beta)=e(\alpha, \gamma)=e(\alpha, \gamma^*).\]
Recall that $\{\alpha, \beta, \gamma^*\}\subseteq p_\xi\in \mc{H}^\pi_0$. So $\pi(\{\alpha, \beta, \gamma^*\})=0$. Together with the fact $\alpha<\min \{\beta, \gamma^*\}$,
\[t_\beta\cap I_{e(\alpha, \beta)}\subseteq t_{\gamma^*}\vee t_{\gamma^*} \cap I_{e(\alpha, \beta)}\subseteq t_\beta.\]
By (4) and the fact $e(\alpha, \beta)<N$, $t_\gamma\cap I_{e(\alpha, \beta)}=t_{\gamma^*}\cap I_{e(\alpha, \beta)}$. This shows
\[t_\beta\cap I_{e(\alpha, \beta)}\subseteq t_{\gamma}\vee t_{\gamma} \cap I_{e(\alpha, \beta)}\subseteq t_\beta.\]
Hence, $\pi(\alpha, \beta, \gamma)=0$.\medskip

\textbf{Case 2.} $\alpha\in p_\xi\setminus \overline{p}$.\medskip

The non-trivial case is $\gamma\in p_\zeta$. By (6) and (5), 
\[e(\alpha, \beta)=e(\alpha, \gamma)>N.\]
 Together with (3), $\beta\in p_\zeta$. Now by (3) and the fact $e(\alpha, \beta)>N$, $t_\gamma\cap I_{e(\alpha, \beta)}\subseteq t_\beta$. So  $\pi(\alpha, \beta, \gamma)=0$.\medskip

\textbf{Case 3.} $\alpha\in p_\zeta\setminus \overline{p}$.

Trivial.\medskip

So in any case, $\pi(\alpha, \beta, \gamma)=0$. This shows $p_\xi\cup p_\zeta\in \mc{H}^\pi_0$ and finishes the proof of the claim.
\end{proof}
It is easy to see that $\mc{H}^\pi_0$ is   uncountable. Apply $\ms{K}_3$ to find an uncountable 0-homogeneous subset $\Gamma'\in [\omega_1]^{\omega_1}$.

First find $\delta<\omega_1$ such that $\Gamma'\cap \delta$ is infinite. Note by coherence of $e$,  for every $\gamma\in \Gamma'\setminus \delta$, there exists $F_\gamma\in [\delta]^{<\omega}$ such that $e_\gamma$ agrees with $e_\delta$ on $\delta\setminus F_\gamma$. Then
find $F\in [\delta]^{<\omega}$ and $\Gamma\in [\Gamma'\setminus \delta]^{\omega_1}$ such that
\begin{enumerate}\setcounter{enumi}{6}
\item for every $\gamma\in \Gamma$ and every $\alpha\in \delta\setminus F$, $e(\alpha, \delta)=e(\alpha, \gamma)$.
\end{enumerate}
Let
\[A=\{e(\alpha, \delta): \alpha\in \Gamma'\cap \delta\setminus F\}.\]
We check the following property.\medskip 

\textbf{Claim 2.} For every $n\in A$, $(\{t_\beta\cap I_n: \beta\in \Gamma\}, \subseteq)$ is linearly ordered.
\begin{proof}[Proof of Claim 2.]
Fix $n\in A$. By definition of $A$, find $\alpha\in \Gamma'\cap \delta\setminus F$ such that 
\[e(\alpha, \delta)=n.\]
 Arbitrarily choose $\beta<\gamma$ in $\Gamma$.
By (7), 
\[e(\alpha, \beta)=e(\alpha, \delta)=e(\alpha, \gamma).\]
Together with the fact $\pi(\alpha, \beta, \gamma)=0$,
\[t_\beta\cap I_{e(\alpha, \beta)}\subseteq t_\gamma \vee t_\gamma\cap I_{e(\alpha, \beta)}\subseteq t_\beta.\]
Now the claim follows from  $e(\alpha, \beta)=e(\alpha, \delta)=n$.
\end{proof}
By Claim 2, for every $n\in A$, $\bigcap_{\gamma\in \Gamma} (t_\gamma\cap I_n)$ is the $\subseteq$-least element of $\{t_ \gamma\cap I_n: \gamma\in \Gamma\}$. Together with (1), 
$\bigcap_{\gamma\in \Gamma} t_\gamma\cap I_n\neq \emptyset$.
So $\bigcap_{\gamma\in \Gamma} t_\gamma$ is infinite and fills the tower $T$. Since $T$ is an arbitrary tower, $\mathfrak{t}>\omega_1$.
\end{proof}

\subsection{{\rm P($\sigma$-$n$-linked$\ra$K$_{n+1}$)} implies $\mathfrak{t}>\omega_1$}

Throughout this subsection, we fix a natural number $n\geq 2$ and assume {\rm P($\sigma$-$n$-linked$\ra$K$_{n+1}$)}.
We will present the forcing method.

By \cite{Todorcevic86} and \cite{Todorcevic91},  {\rm P($\sigma$-$n$-linked$\ra$K$_{n+1}$)} implies $\mathfrak{b}>\omega_1$. For completeness, we sketch a proof below.
\begin{prop}[\cite{Todorcevic86}, \cite{Todorcevic91}]\label{prop sK3toK4 b}
{\rm P($\sigma$-$n$-linked$\ra$K$_{n+1}$)} implies $\mathfrak{b}>\omega_1$.
\end{prop}
\begin{proof}[Sketch Proof.]
Fix a $<^*$-increasing family $\mc{B}=\{f_\alpha\in \omega^\omega: \alpha<\omega_1\}$ such that each $f_\alpha$ is strictly increasing. Say $F\in [\mc{B}]^{<\omega}$ is $n$-splitting if 
\[\text{for every $s\in \omega^{<\omega}$, $|\{g(|s|): g\in F$ and }s\subset g\}|\leq n.\]

Let $\mathbb{P}=\{F\in [\mc{B}]^{<\omega}: F$ is $n$-splitting$\}$ ordered by reverse inclusion. 

It is straightforward to check that $\mathbb{P}$ is $\sigma$-$n$-linked. By {\rm P($\sigma$-$n$-linked$\ra$K$_{n+1}$)}, there is $X\in [\mc{B}]^{\omega_1}$ such that $[X]^{n+1}\subseteq \mathbb{P}$. Then $X$ is $n$-splitting and hence has a $<^*$-upper bound. This upper bound dominates $\mc{B}$.
\end{proof}

Inspired from the coloring proof of Theorem \ref{thm K3 t} and its generalizations, we will use properties of the structure:
\[\langle (\{t_\alpha\cap I_n: \alpha\in \Gamma\}, \subseteq): n\in A\rangle\]
where $A\subseteq \omega$ and $\Gamma\subseteq \omega_1$. First, we recall a concept.

For a finite partial order $(P,\prec)$, the \emph{width} of $(P,\prec)$ is the maximal size of a subset consisting of pairwise incomparable elements, i.e.,
\[\max(\{|F|: F\subseteq P\text{ and } \forall a\neq b\in F ~ (a\not\prec b \wedge b\not\prec a)\}.\]

For a tower $T=\{t_\alpha: \alpha<\omega_1\}$ and $I\in [\omega]^\omega$, our goal is:

\begin{itemize}
\item Find    $\Gamma\in [\omega_1]^{\omega_1}$ such that 
\[(\{t_\alpha\cap I_m: \alpha\in \Gamma\}, \subseteq)\text{ has width $\leq n$ for all }m<\omega.\]
\end{itemize}

For this, we first introduce the $\sigma$-$n$-linked poset.
\begin{defn}\label{defn tower poset}
For a tower $T=\{t_\alpha: \alpha<\omega_1\}$ and $I\in [\omega]^\omega$,  $\mathbb{P}_{T, I}$ is the poset consisting of $F\in [\omega_1]^{<\omega}$ such that for every $m<\omega$, $(\{t_\alpha\cap I_m: \alpha\in F\}, \subseteq)$ has width $\leq n$. The order is reverse inclusion.
\end{defn}
We will need $\mathfrak{b}>\omega_1$ to guarantee   that the potential pseudointersection    is infinite.
\begin{prop}\label{prop sK3toK4 t}
{\rm P($\sigma$-$n$-linked$\ra$K$_{n+1}$)} implies $\mathfrak{t}>\omega_1$. In particular, {\rm P($\sigma$-$n$-linked$\ra$K$_{n+1}$)} implies {\rm MA$_{\omega_1}(\sigma$-centered)}.
\end{prop}
\begin{proof}
Fix a tower $T=\{t_\alpha: \alpha<\omega_1\}$. By Proposition \ref{prop sK3toK4 b}, $\mathfrak{b}>\omega_1$. Replacing $T$ by a sub-tower if necessary, we find  $I\in [\omega]^\omega$ such that
\begin{enumerate}
\item for every $\alpha<\omega_1$ and every $m<\omega$, $t_\alpha\cap I_m\neq\emptyset$.
\end{enumerate}
Denote for $m<\omega$ and $F\subseteq \omega_1$,
\[\mc{I}^F_m=\{t_\alpha\cap I_m: \alpha\in F\}.\]
\medskip

\textbf{Claim.} $\mathbb{P}_{T, I}$, defined in Definition \ref{defn tower poset}, is $\sigma$-$n$-linked. 

\begin{proof}[Proof of Claim.]
Fix a partition $\mathbb{P}_{T, I}=\bigcup_{m<\omega} P_m$ such that for every $m$, there are $N, M<\omega$ satisfying
\begin{enumerate}\setcounter{enumi}{1}
\item for $F\in P_m$,  $|F|=N$ and $t_{F(j)}\setminus M\subseteq t_{F(i)}$ whenever $i<j<N$;
\item for $E, F\in P_m$ and $i<N$, $t_{E(i)}\cap I(M+1)=t_{F(i)}\cap I(M+1)$.
\end{enumerate}

It suffices to show that each $P_m$ is $n$-linked.  Fix $m<\omega$ and $F_0,...,F_{n-1}\in P_m$. Let $F=\bigcup_{i<n} F_i$. To show $F\in \mathbb{P}_{T, I}$, we will check that for every $k<\omega$, $(\mc{I}^F_k, \subseteq)$ has width $\leq n$. Let $N, M$ witness (2)-(3).

For $k\leq M$, by (3), $(\mc{I}^F_k, \subseteq)=(\mc{I}^{F_0}_k, \subseteq)$ and hence has width $\leq n$.

For $k>M$, by (2), $(\mc{I}^{F_i}_k, \subseteq)$ is linear for each $i<n$. So $\mc{I}^F_k=\bigcup_{i<n} \mc{I}^{F_i}_k$  has width $\leq n$.

So $F\in \mathbb{P}_{T, I}$ and hence is a lower bound of $F_0,..., F_{n-1}$. This shows that $\mathbb{P}_{T, I}$ is $\sigma$-$n$-linked and finishes the proof of the claim.
\end{proof}

By {\rm P($\sigma$-$n$-linked$\ra$K$_{n+1}$)}, fix $X\in [\omega_1]^{\omega_1}$ such that $[X]^{n+1}\subseteq \mathbb{P}_{T, I}$.
Note that if  $(\mc{I}^X_k, \subseteq)$ has width $> n$ for some $k$, then this fact is witnessed by an $(n+1)$-element subset of $X$. So,
\begin{enumerate}\setcounter{enumi}{3}
\item for every $k<\omega$, $(\mc{I}^X_k, \subseteq)$ has width $\leq n$.
\end{enumerate}
Consequently, for all $k<\omega$, $|\min_{\subseteq}(\mc{I}^X_k)|\leq n$ where $\min_{\subseteq}(\mc{I}^X_k)$ is the collection of $\subseteq$-minimal elements of $\mc{I}^X_k$.

So there exists $m\leq n$ with the following property.
\begin{enumerate}\setcounter{enumi}{4}
\item For some $Y\in [X]^{\omega_1}$, $A\in [\omega]^\omega$ and $\langle \mc{I}'_k\in [\mc{I}^X_k]^m: k\in A\rangle$, 
\[\text{ for every } \alpha\in Y, ~ \{k\in A: \forall a\in \mc{I}'_k ~ a\not\subseteq t_\alpha\}\text{ is finite}.\]
\end{enumerate}
To see this, 
take $m\leq n$ such that $A=\{k<\omega: |\min_{\subseteq}(\mc{I}^X_k)|=m\}$ is infinite. Then $X, A$ and $\langle \min_{\subseteq}(\mc{I}^X_k): k\in A\rangle$ witness (5) for $m$. 

Let $n^*$   be the least $m$ satisfying (5) witnessed by
$Y$, $A$ and $\langle \mc{I}'_k: k\in A\rangle$. We claim that 
\begin{enumerate}\setcounter{enumi}{5}
\item for every $\alpha\in Y$, $\bigcup_{k\in A} \bigcup \mc{I}'_k\subseteq^* t_\alpha$.
\end{enumerate}
Suppose otherwise. There are  $\alpha\in Y$, $B\in [A]^\omega$ and $\{a_k\in\mc{I}'_k: k\in B\}$ such that 
\[a_k\not\subseteq t_\alpha\text{ for all }k\in B.\] 
Now fix $\beta\in Y\setminus \alpha$. By (5) and $t_\beta\subseteq^* t_\alpha$, for large enough  $k\in B$,
\[\text{there exists $a\in \mc{I}'_k$ with }a\subseteq t_\beta\cap I_k\subseteq t_\alpha.\] 
Note $a\neq a_k$ and hence $a\in \mc{I}'_k\setminus \{a_k\}$. This shows that $Y\setminus \alpha$, $B$ and $\langle \mc{I}'_k\setminus \{a_k\}: k\in B\rangle$ witness (5) for $n^{*}-1$. But this contradicts minimality of $n^*$.

This contradiction shows that (6) holds. Together with (1), $\bigcup_{k\in A} \bigcup \mc{I}'_k$ is infinite and hence fills the tower $T$.

Since $T$ is an arbitrary tower, $\mathfrak{t}>\omega_1$. The in particular part follows from Theorem \ref{thm BRMS}.
\end{proof}

We also present a coloring method, but leave the details to interested reader. 

For $\mc{B}$ as in Proposition \ref{prop sK3toK4 b}, define $c: [\omega_1]^{n+1}\ra 2$ by $c(\alpha_0,...,\alpha_n)=0$ iff
\[e(\alpha_0, \alpha_1)>e(\alpha_1, \alpha_2)>\cdots>e(\alpha_{n-1}, \alpha_n)\ra f_{\alpha_n}(e(\alpha_{n-1}, \alpha_n))\leq e(\alpha_0, \alpha_1).\]
Then $\mc{H}^c_0$ is $\sigma$-$n$-linked and an uncountable 0-homogeneous subset of $c$ will induce an upper bound of $\mc{B}$.

For a tower $T=\{t_\alpha: \alpha<\omega_1\}$ and $I\in [\omega]^\omega$ such that $t_\alpha\cap I_m\neq\emptyset$ for all $\alpha$ and $m$,
define a coloring $\pi: [\omega_1]^{n+1}\ra 2$ by $\pi(\alpha_0,...,\alpha_n)=0$ iff
\begin{align*}
e(\alpha_0, \alpha_1)=e(\alpha_0, \alpha_2)=\cdot\cdot\cdot&=e(\alpha_0, \alpha_n)  \\
&\ra t_{\alpha_i}\cap I_{e(\alpha_0, \alpha_1)}\subseteq t_{\alpha_j} \text{ for some } 0<i\neq j\leq n.
\end{align*}
Then $\mc{H}_0^\pi$ is $\sigma$-$n$-linked and an uncountable 0-homogeneous subset of $\pi$ induces $A\in [\omega]^\omega$ and $\Gamma\in [\omega_1]^{\omega_1}$ such that
\begin{itemize}
\item for every $m\in A$, $(\{t_\alpha\cap I_m: \alpha\in \Gamma\}, \subseteq)$ has width $\leq n-1$.
\end{itemize}
Then above property will induce an infinite set filling the tower $T$.

The above forcing proof induces an uncountable $X$ such that $(\{t_\alpha\cap I_m: \alpha\in X\}, \subseteq)$ has width $\leq n$ for all $m$. And the coloring method induces an uncountable $\Gamma$ such that $(\{t_\alpha\cap I_m: \alpha\in \Gamma\}, \subseteq)$ has width $\leq n-1$ for infinitely many $m$. The latter has a smaller width because in the coloring method, one of the $n+1$ inputs is only used to compute the interval $I_m$.

\subsection{{\rm P$_{\omega_1}$(precaliber $\omega_1\ra \sigma$-linked)} implies $\mathfrak{b}>\omega_1$}

As in the proof of Proposition \ref{prop sK3toK4 t}, we need first to prove $\mathfrak{b}>\omega_1$ before proving $\mathfrak{t}>\omega_1$. 

A \emph{C-sequence}  (or a \emph{ladder system}) is a sequence $\langle C_\alpha: \alpha<\omega_1\rangle$ such that $C_{\alpha+1}=\{\alpha\}$ and for limit ordinal $\alpha$, $C_\alpha$ is a cofinal subset of $\alpha$ of order type $\omega$.

The following well-known fact will be used to justify the coloring property.
\begin{fact}\label{fact1}
If $\Gamma\in [\omega_1]^{\omega_1}$, then for all but countably many $\alpha\in \Gamma$, for every $n<\omega$,   $\{\beta\in \Gamma\setminus (\alpha+1): e(\alpha, \beta)>n\}$ is uncountable.
\end{fact}
\begin{proof}
Suppose otherwise. For uncountably many $\alpha\in \Gamma$, there is $n_\alpha<\omega$ such that $\{\beta\in \Gamma\setminus (\alpha+1): e(\alpha, \beta)>n_\alpha\}$ is at most countable.
Then there are $n<\omega$ and $X\in [\Gamma]^{\omega_1}$ such that for $\alpha<\beta$ in $X$, $e(\alpha, \beta)\leq n$.  This clearly contradicts (coh1).
\end{proof}

Now we are ready to prove the result of this subsection.
\begin{thm}\label{thm ptosK2 b}
{\rm P$_{\omega_1}$(precaliber $\omega_1\ra \sigma$-linked)} implies $\mathfrak{b}>\omega_1$
\end{thm}
 \begin{proof}
 Fix a $<^*$-increasing family $\mc{B}=\{f_\alpha\in \omega^\omega: \alpha<\omega_1\}$ such that each $f_\alpha$ is strictly increasing. Fix a C-sequence $\langle C_\alpha: \alpha<\omega_1\rangle$.
 
 Define a coloring $\pi: [\omega_1]^{<\omega}\ra 2$ by $\pi(a)=0$ iff for all $\alpha<\beta<\gamma $  in $a$,
 \[ (e(\alpha, \beta)\leq e(\beta, \gamma) \wedge C_\gamma\cap (\alpha, \beta]\neq \emptyset)\ra f_\gamma(e(\alpha, \beta))\leq e(\beta, \gamma). \]
 Recall that $\mc{H}^\pi_0$ is $\pi^{-1}\{0\}$ ordered by reverse inclusion.\medskip
 
 \textbf{Claim.}   $\mc{H}^\pi_0$ has precaliber $\omega_1$. 
 \begin{proof}[Proof of Claim.]
 Fix an uncountable $\mc{A}\subseteq\mc{H}^\pi_0$. Find $\mc{A}'\in [\mc{A}]^{\omega_1}$ such that
  \begin{enumerate}
  \item $\mc{A}'$ forms a $\Delta$-system with root $\overline{p}$.
    \end{enumerate}
  Find  a stationary set $\Sigma\subseteq \omega_1$, $\{p_\alpha\in \mc{A}': \alpha\in \Sigma\}$ and $N<M<\omega$ such that
 \begin{enumerate}\setcounter{enumi}{1}
 \item for all $\alpha<\beta$ in $\Sigma $,  $\max(\overline{p})<\alpha$ and $\max(p_\alpha)<\beta<\min (p_\beta\setminus \overline{p})$;
 \item   for all $\alpha\in \Sigma$, $e[[p_\alpha]^2]\subseteq N$ and $\max(\{f_\xi(N): \xi\in p_\alpha\})<M$.
  \end{enumerate}
  
  Note for a given $\beta\in \Sigma$, $\bigcup_{\eta\in p_\beta\setminus \overline{p}} e^{-1}_\eta[M]$ is  finite. So by the Pressing Down Lemma,  find a stationary subset $\Sigma'\subseteq\Sigma$ such that
   \begin{enumerate}\setcounter{enumi}{3}
 \item for all $\alpha<\beta$ in $\Sigma'$, for all $\xi\in p_\alpha\setminus \overline{p}$ and all $\eta\in p_\beta\setminus \overline{p}$, $e(\xi, \eta)\geq M$.
 \end{enumerate}
 
Note by (2),  $\bigcup_{\eta\in p_\beta\setminus \overline{p}}  (C_\eta\cap \beta)$ is  finite for every $\beta\in \Sigma$. So by the Pressing Down Lemma again,  find a stationary subset $\Gamma\subseteq \Sigma'$ such that
 \begin{enumerate}\setcounter{enumi}{4}
 \item for some $\overline{\alpha}<\min(\Gamma)$, for all $\alpha\in\Gamma$ and  all $\xi\in p_\alpha\setminus \overline{p}$, $C_\xi\cap \alpha\subseteq \overline{\alpha}$.
 \end{enumerate}
 
We check that $\{p_\alpha: \alpha\in \Gamma\}$ is centered. 
Fix $\alpha<\beta<\gamma$ in $\bigcup_{\xi\in \Gamma} p_\xi$. It suffices to prove that $\pi(\alpha, \beta, \gamma)=0$. So assume 
\begin{enumerate}\setcounter{enumi}{5}
\item $e(\alpha, \beta)\leq e(\beta, \gamma) \text{ and } C_\gamma\cap (\alpha, \beta]\neq \emptyset$.
\end{enumerate}\medskip

\textbf{Case 1.} $\alpha\in \overline{p}$. \medskip

By (3), 
\[e(\alpha, \beta)<N.\]

If $e(\beta, \gamma)\geq N$, then by (3), $\beta$ and $\gamma$ are not contained in a single $p_\xi$. Then by (4),
$e(\beta, \gamma)\geq M$. Together with (3) and the fact that $f_\gamma$ is strictly increasing, 
\[f_\gamma(e(\alpha, \beta))<f_\gamma(N)<M\leq e(\beta, \gamma)\text{ and so } \pi(\alpha, \beta, \gamma)=0.\]

If $e(\beta, \gamma)<N$, then by (4), $\beta$ and $\gamma$ are from a single $p_\xi$. Then $\{\alpha, \beta, \gamma\}\subseteq p_\xi$. Hence $\pi(\alpha, \beta, \gamma)=0$.\medskip

\textbf{Case 2.} $\alpha\in p_{\alpha'}\setminus \overline{p}$ for some $\alpha'\in \Gamma$.\medskip

 Assume $\gamma\in p_{\gamma'}\setminus \overline{p}$ for some $\gamma'\in \Gamma$. 

If $\alpha'<\gamma'$, then by (5), $C_\gamma\cap \gamma'\subseteq \overline{\alpha}< \alpha'<\alpha$. Together with (6),  we conclude that $\beta\geq\gamma'$. So $\gamma'\leq \beta<\gamma$ and $\gamma\in p_{\gamma'}\setminus \overline{p}$.  Then by (2),  $\beta\in p_{\gamma'}$. But then by (3)-(4), 
\[e(\alpha, \beta)\geq M>N> e(\beta, \gamma).\]
 This contradicts (6). 

This contradiction shows that $\alpha'=\gamma'$. So $\{\alpha, \beta, \gamma\}\subseteq p_{\alpha'}$ and hence $\pi(\alpha, \beta, \gamma)=0$.\medskip

So in any case, $\pi(\alpha, \beta, \gamma)=0$. This shows that $\{p_\alpha: \alpha\in \Gamma\}$ is centered and hence $\mc{H}^\pi_0$ has precaliber $\omega_1$.
\end{proof}

By {\rm P$_{\omega_1}$(precaliber $\omega_1\ra \sigma$-linked)}, $\mc{H}^\pi_0$ is $\sigma$-linked. Let
\[\mc{H}^\pi_0=\bigcup_{n<\omega} P_n \text{ be a partition of $\mc{H}^\pi_0$ into countably many linked subsets}.\]

Note for every $\alpha<\omega_1$, there exists $n<\omega$ such that $\{\gamma>\alpha: \{\alpha, \gamma\}\in P_n\}$  is stationary. So
we find $n<\omega$ such that
\[X=\{\alpha<\omega_1: \{\gamma>\alpha: \{\alpha, \gamma\}\in P_n\}\text{ is stationary$\}$ is uncountable}.\]
By Fact \ref{fact1}, choose $\alpha\in X$ such that 
\begin{enumerate}\setcounter{enumi}{6}
\item for every $m$,   $\{\beta\in X: e(\alpha, \beta)>m\}$ is uncountable.
\end{enumerate}
 By definition of $X$, choose $i^*<\omega$ and $\eta>\alpha$ such that
\begin{enumerate}\setcounter{enumi}{7}
\item $Y=\{\gamma>\eta: \eta\in C_\gamma, e(\alpha, \gamma)=i^* \text{ and }  \{\alpha, \gamma\}\in P_n\}$ is stationary.
\end{enumerate}
Now by (7), find $A\in [\omega\setminus (i^*+1)]^\omega$ and an increasing sequence $\langle\beta_k: k\in A\rangle$ such that
\[\text{for every } k\in A, ~ \beta_k\in X\setminus (\eta+1)\text{ and } e(\alpha, \beta_k)=k.\]
Let $\delta=\sup \{\beta_k: k\in A\}$. 

We are now ready to find an upper bound of $\{f_\gamma: \gamma\in Y\setminus \delta\}$. Arbitrarily choose $k\in A$ and $\gamma\in Y\setminus \delta$. By (coh2) and (8), 
\[k=e(\alpha, \beta_k)\leq \max\{e(\alpha, \gamma), e(\beta_k, \gamma)\}=\max\{i^*, e(\beta_k, \gamma)\}.\]
Together with the fact $\min (A)>i^*$ and (8), we conclude that 
\begin{enumerate}\setcounter{enumi}{8}
\item $k=e(\alpha, \beta_k)\leq e(\beta_k, \gamma) \text{ and } \eta\in C_\gamma\cap (\alpha, \beta_k]$.
\end{enumerate}
Recall that $\{\alpha, \gamma\}\in P_n$ and $\beta_k\in X$. Together with the definition of $X$ and the fact that $P_n$ is linked, $\pi(\alpha, \beta_k, \gamma)=0$. Then by  (9), $f_\gamma(k)\leq e(\beta_k, \gamma)$.

Since $k\in A$ and $\gamma\in Y\setminus \delta$ are arbitrary, 
\begin{enumerate}\setcounter{enumi}{9}
\item $f_\gamma(k)\leq e(\beta_k, \gamma)$ for all $k\in A$ and $\gamma\in Y\setminus \delta$.
\end{enumerate}
By coherence of $e$, for every $\gamma\in Y\setminus \delta$, $e_\delta=^*e_\gamma$ and hence
\[f_\gamma(k)\leq e(\beta_k, \delta) \text{ for all but finitely many }k\in A.\]
Now define $f: \omega\ra \omega$ by $f(j)=e(\beta_{A(j)}, \delta)$. Arbitrarily choose $\gamma\in Y\setminus \delta$.  Then
\[f_\gamma(j)< f_\gamma(A(j))\leq e(\beta_{A(j)}, \delta)=f(j)\text{ for all but finitely many }j<\omega.\]
This shows that $f$ dominates $\{f_\gamma: \gamma\in Y\setminus \delta\}$ and hence $\mc{B}$.
 \end{proof}
 
 \subsection{{\rm P$_{\omega_1}$(precaliber $\omega_1\ra \sigma$-3-linked)} implies $\mathfrak{t}>\omega_1$}
 
 Throughout this subsection, we use $Lim$ to denote the set of countable limit ordinals.
 
\begin{thm}\label{thm ptosK3 t}
{\rm P$_{\omega_1}$(precaliber $\omega_1\ra \sigma$-3-linked)} implies $\mathfrak{t}>\omega_1$
\end{thm}
 \begin{proof}
 Fix a C-sequence $\langle C_\alpha: \alpha<\omega_1\rangle$. Let $T=\{t_\alpha: \alpha<\omega_1\}$ be a tower.
 
By Theorem \ref{thm ptosK2 b} and replacing $T$ by a sub-tower, we may assume that there exists $I\in [\omega]^\omega$ such that
 \begin{enumerate}
 \item for every $\alpha<\omega_1$ and every $n<\omega$, $t_\alpha\cap I_n\neq \emptyset$.
 \end{enumerate}
 Define a coloring $\pi: [Lim]^{<\omega}\ra 2$ by $\pi(a)=0$ iff for all $\xi<\alpha<\beta<\gamma$ in $a$,
 \begin{align*}
 (e(\alpha, \beta)=e(\alpha, \gamma)<e(\xi, \alpha) \wedge C_\beta&\cap (\xi, \alpha]\neq \emptyset \wedge C_\gamma\cap (\xi, \alpha]\neq \emptyset)  \\
 &\ra (t_\beta\cap I_{e(\xi, \alpha)}\subseteq t_\gamma \vee t_\gamma\cap I_{e(\xi, \alpha)}\subseteq t_\beta).
 \end{align*}
 Note that the positions of $\beta$ and $\gamma$ are symmetric. So the order between $\beta$ and $\gamma$ is not important and we can simply say that $\xi<\alpha<\min \{\beta, \gamma\}$.\medskip

\textbf{Claim.} $\mc{H}_0^\pi$ has precaliber $\omega_1$.  
\begin{proof}[Proof of Claim.]
 Fix an uncountable $\mc{A}\subseteq\mc{H}^\pi_0$. 
  Find  a stationary set $\Sigma\subseteq \omega_1$, $\{p_\alpha\in \mc{A}: \alpha\in \Sigma\}$ and $n, N<\omega$ such that
 \begin{enumerate}\setcounter{enumi}{1}
 \item $\{p_\alpha: \alpha\in \Sigma\}$ forms a $\Delta$-system with root $\overline{p}$;
 \item for all $\alpha<\beta$ in $\Sigma $,  $\max(\overline{p})<\alpha$ and $\max(p_\alpha)<\beta<\min (p_\beta\setminus \overline{p})$;
 \item  for all $\alpha<\beta$  in $\Sigma $, 
 \begin{itemize}
 \item $|p_\alpha|=n$, $e[[p_\alpha]^2]\subseteq N$;
  \item for $i<j<n$, $t_{p_\alpha(j)}\setminus N\subseteq t_{p_\alpha(i)}$;
 \item  $t_{p_\alpha(i)}\cap I(N+1)=t_{p_\beta(i)}\cap I(N+1)$ whenever $i<n$;
 \end{itemize}
  \end{enumerate}
  Note that for every $\alpha\in \Sigma$, $\bigcup_{\xi\in p_\alpha\setminus \overline{p}} (e^{-1}_\xi[N]\cup (C_\xi\cap \alpha))$ is finite. By the Pressing Down Lemma, we find a stationary $\Gamma\subseteq \Sigma$ such that
   \begin{enumerate}\setcounter{enumi}{4}
 \item for all $\alpha<\beta$ in $\Gamma$, 
 \begin{itemize}
 \item $(p_\alpha, e)$ is isomorphic to $(p_\beta, e)$;
 \item for all $\xi\in p_\alpha\setminus \overline{p}$ and all $\eta\in p_\beta\setminus \overline{p}$, $e(\xi, \eta)\geq N$;
 \end{itemize}
 \item for some $\overline{\alpha}<\min(\Gamma)$, for all $\alpha<\beta$ in $\Gamma$, 
 \begin{itemize}
 \item for all $\xi\in p_\alpha\setminus \overline{p}$, $C_\xi\cap \alpha\subseteq \overline{\alpha}$;
  \item for $i<n$, $C_{p_\alpha(i)}\cap \overline{\alpha}=C_{p_\beta(i)}\cap \overline{\alpha}$.
  \end{itemize}
 \end{enumerate}
 We check that $\{p_\alpha: \alpha\in \Gamma\}$ is centered. To see this, fix $\xi<\alpha<\beta<\gamma$ in $\bigcup_{\zeta\in \Gamma} p_\zeta$. It suffices to prove that $\pi(\xi, \alpha, \beta, \gamma)=0$. So assume 
\begin{enumerate}\setcounter{enumi}{6}
\item $e(\alpha, \beta)=e(\alpha, \gamma)<e(\xi, \alpha) \wedge C_\beta\cap (\xi, \alpha]\neq \emptyset \wedge C_\gamma\cap (\xi, \alpha]\neq \emptyset$.
\end{enumerate}
We discuss by cases.\medskip

\textbf{Case 1.} $\alpha\in \overline{p}$.\medskip

Then $\xi\in \overline{p}$. The only non-trivial case is that $\beta\in p_{\beta'}\setminus \overline{p}$, $\gamma\in p_{\gamma'}\setminus \overline{p}$ and $\beta'<\gamma'$. Let $i<n$ be such that $\gamma=p_{\gamma'(i)}$. Denote $\gamma^*=p_{\beta'}(i)$. 

If $\gamma^*=\beta$, then by (4), $e(\xi, \alpha)<N$ and 
\[t_\beta\cap I_{e(\xi, \alpha)}=t_{p_{\beta'}(i)}\cap I_{e(\xi, \alpha)}=t_{p_{\gamma'}(i)}\cap I_{e(\xi, \alpha)}=t_\gamma\cap  I_{e(\xi, \alpha)}.\]
 Hence $\pi(\xi, \alpha, \beta, \gamma)=0$.

Now suppose $\gamma^*\neq \beta$. Note by (6), 
\[C_\gamma\cap \gamma'=C_\gamma\cap \overline{\alpha}=C_{\gamma^*}\cap \overline{\alpha}=C_{\gamma^*}\cap \beta'.\]
Since $\alpha<\beta'<\gamma'$, 
\[C_{\gamma^*}\cap (\xi, \alpha]=C_{\gamma}\cap (\xi, \alpha]\text{ and is non-empty by (7)}.\]
 By (5) and (7), 
 \[e(\alpha, \gamma^*)=e(\alpha, \gamma)=e(\alpha, \beta)<e(\xi, \alpha).\]
 In summary,
 \begin{enumerate}\setcounter{enumi}{7}
\item $e(\alpha, \beta)=e(\alpha, \gamma^*)<e(\xi, \alpha) \wedge C_\beta\cap (\xi, \alpha]\neq \emptyset \wedge C_{\gamma^*}\cap (\xi, \alpha]\neq \emptyset$.
\end{enumerate}
 Note that $\pi(\{\xi, \alpha, \beta, \gamma^*\})=0$ and $\xi<\alpha<\min\{\beta, \gamma^*\}$. By (8),
 \[t_\beta\cap I_{e(\xi, \alpha)}\subseteq t_{\gamma^*} \vee t_{\gamma^*}\cap I_{e(\xi, \alpha)}\subseteq t_\beta.\]
Together with (4) and the fact $e(\xi, \alpha)<N$, 
 \[t_\beta\cap I_{e(\xi, \alpha)}\subseteq t_{\gamma} \vee t_{\gamma}\cap I_{e(\xi, \alpha)}\subseteq t_\beta.\]
 Hence $\pi(\xi, \alpha, \beta, \gamma)=0$.\medskip
 
 \textbf{Case 2.} $\xi\in \overline{p}$ and $\alpha\in p_{\alpha'}\setminus \overline{p}$ for some $\alpha'\in \Gamma$.\medskip
 
 By (4), $e(\xi, \alpha)<N$ and then by (5) and (7), $\beta, \gamma$ are in $p_{\alpha'}$. Hence $\{\xi, \alpha, \beta, \gamma\}\subseteq p_{\alpha'}$ and $\pi(\xi, \alpha, \beta, \gamma)=0$.\medskip
 
 \textbf{Case 3.} $\xi\in p_{\xi'}\setminus \overline{p}$ for some $\xi'\in \Gamma$.\medskip
 
 Assume $\alpha\in p_{\alpha'}\setminus \overline{p}$. If $\xi'=\alpha'$, then the argument in Case 2 shows that $\beta, \gamma$ are also in $p_{\xi'}$. Hence $\pi(\xi, \alpha, \beta, \gamma)=0$.
 
 Now assume $\xi'<\alpha'$. Let $\gamma'\in \Gamma$ be such that $\gamma\in p_{\gamma'}\setminus \overline{p}$.   Then by (6), 
 \[C_\gamma\cap [\overline{\alpha}, \gamma')=\emptyset\]
  and by (7),
  \[C_\gamma\cap (\xi, \alpha]\neq \emptyset.\]
  Note by (3) and the choice of $\overline{\alpha}$, $\overline{\alpha}<\xi<\gamma'$. 
  We conclude that $\gamma'\leq \alpha$ and hence by (3), $\alpha'=\gamma'$. So $\{\alpha, \beta, \gamma\}\subseteq p_{\alpha'}$. 
  
  By (5), $e(\xi, \alpha)\geq N$ and then by (4), $t_\gamma\cap I_{e(\xi, \alpha)}\subseteq t_\beta$. This shows $\pi(\xi, \alpha, \beta, \gamma)=0$.\medskip
  
  So in any case, $\pi(\xi, \alpha, \beta, \gamma)=0$. This shows that $\{p_\eta: \eta\in \Gamma\}$ is centered and hence $\mc{H}_0^\pi$ has precaliber $\omega_1$. 
  \end{proof}
  
 By  {\rm P$_{\omega_1}$(precaliber $\omega_1\ra \sigma$-3-linked)}, $\mc{H}^\pi_0$ is $\sigma$-3-linked. Let 
 \[\mc{H}^\pi_0=\bigcup_{n<\omega} H_n \text{ be a partition of $\mc{H}^\pi_0$ into 3-linked subsets}.\]
 Fix a countable elementary submodel $\mc{M}\prec H({\omega_2})$    containing everything relevant.  First choose $\eta\in \omega_1\setminus \mc{M}$ such that
 \[\{\beta\in Lim: \eta\in C_\beta\} \text{ is stationary.}\]
Fix $\alpha\in Lim\setminus \eta$. Then choose natural numbers $n, N$ such that
 \begin{enumerate}\setcounter{enumi}{8}
 \item $X=\{\beta\in Lim\setminus (\alpha+1): \eta\in C_\beta, ~\{\alpha, \beta\} \in H_n \text{ and } e(\alpha, \beta)=N\}$ is stationary.
 \end{enumerate}
 We will find $A\in [\omega]^\omega$ such that for every $m\in A$, $(\{t_\beta\cap I_m: \beta\in X\}, \subseteq)$ is a linear order.
 
 Let 
 $Y=\{\xi\in Lim: \{\xi, \beta'\}\in H_n \text{ for some }\beta'>\xi\}$. 
 Then $Y\in \mc{M}$ and by (9), $\alpha\in Y\setminus \mc{M}$. By elementarity, $Y$ is uncountable. Let
 \[Y'=Y\cap \mc{M}\setminus e_\alpha^{-1}[N+1].\]
 Then $Y'$  is infinite and for every $ \xi\in Y'$,
 \[ e(\xi, \alpha)>N \text{ and } \xi<\mc{M}\cap \omega_1\leq \eta\leq \alpha.\]
Then by (coh1),
  \begin{enumerate}\setcounter{enumi}{9}
 \item $A=\{e(\xi, \alpha): \xi\in Y'\} \text{ is infinite and} \min(A)>N$.
 \end{enumerate}
To check that $A$ is as desired, fix $m\in A$ and $\beta<\gamma$ in $X$. By (10), there exists $\xi\in Y'$ such that $e(\xi, \alpha)=m$. Note that 
 \begin{itemize}
 \item $\xi\in Y'\subseteq Y$ and hence $\{\xi, \beta'\}\in H_n$ for some $\beta'$;
 \item $\beta,\gamma\in X$ and hence $\{\alpha, \beta\}\in H_n$ and $\{\alpha, \gamma\}\in H_n$.
 \end{itemize} 
 Since $H_n$ is 3-linked, $\pi(\xi, \alpha, \beta, \gamma)=0$. Recall that
  \[e(\alpha, \beta)=N=e(\alpha, \gamma)<m=e(\xi, \alpha)\text{ and }\eta\in C_\beta\cap C_\gamma\cap (\xi, \alpha].\]
   By definition of $\pi$, 
 \[t_\beta\cap I_{m}\subseteq t_\gamma \vee t_\gamma\cap I_{m}\subseteq t_\beta.\]
 This justifies that for every $m\in A$, $(\{t_\beta\cap I_m: \beta\in X\}, \subseteq)$ is a linear order.
 
 So for every $m\in A$, $\bigcap_{\beta\in X} (t_\beta\cap I_m)$ is the $\subseteq$-least element of $\{t_\beta\cap I_m: \beta\in X\}$ and  is non-empty by (1).
 Then $\bigcap_{\beta\in X} t_\beta$ is infinite.
 
So $\bigcap_{\beta\in X} t_\beta$ 
  fills the tower $T$. This finishes the proof of the theorem.
 \end{proof}
Now Theorem \ref{thm ptosc} follows from the above theorem.
\begin{proof}[Proof of Theorem \ref{thm ptosc}.]
By Theorem \ref{thm ptosK3 t}, $\mathfrak{t}>\omega_1$. Then by Theorem \ref{thm BRMS},  MA$_{\omega_1}(\sigma$-centered) holds. Together with {\rm P$_{\omega_1}$(precaliber $\omega_1\ra \sigma$-centered)}, MA$_{\omega_1}$(precaliber $\omega_1$) holds.
\end{proof}


\section{Coding finatary coloring by $(n+1)$-ary coloring}

In this section, we   prove that  if every $\sigma$-$n$-linked poset has property K$_{n+1}$, then every $\sigma$-$n$-linked poset has precaliber $\omega_1$. 
The proof actually indicates a method of finding a $\sigma$-$n$-linked coloring on $[\omega_1]^{\leq n+1}$ (or $[\omega_1]^{n+1}$) that codes sufficient 0-homogeneous information for a $\sigma$-$n$-linked coloring on $[\omega_1]^{<\omega}$.


Say $\ms{C}\subseteq [\omega_1]^{<\omega}$ is  \emph{$\sigma$-$n$-linked} if there is a partition $\ms{C}=\bigcup_{m<\omega} C_m$ such that $\bigcup F\in \ms{C}$ for all $m$ and all $F\in [C_m]^n$. We show the following stronger property which has its own interest. 
\begin{prop}\label{prop code fin by 4}
 Assume {\rm MA$_{\omega_1}$($\sigma$-centered)} and $2\leq n<\omega$. If $\ms{C}\subseteq [\omega_1]^{<\omega}$ is downward closed and $\sigma$-$n$-linked,  then there is a coloring $\pi: [\omega_1]^{<n+2}\ra 2$ with the following properties.
 \begin{enumerate}[(i)]
 \item $\mc{H}^\pi_0$ is $\sigma$-$n$-linked.
 \item $\ms{C}\cap [\omega_1]^{<n+1}\subseteq \mc{H}^\pi_0\subseteq \ms{C}$.
  \end{enumerate}
\end{prop}
Recall   $\mc{H}^\pi_0=\{F\in [\omega_1]^{<\omega}: [F]^{<n+2}\subseteq \pi^{-1}\{0\}\}$. Note that conclusion (ii) of above proposition guarantees that $\pi^{-1}\{0\}$ is downward closed.

Fix a partition witnessing that $\ms{C}$ is $\sigma$-$n$-linked.
\begin{itemize}
\item $\tau: \ms{C}\ra \omega $ is a partition such that  every $\tau^{-1}\{m\}$ is $n$-linked. 
\end{itemize}
We then have the following property.
\begin{enumerate}[(I)]
\item $\ms{C}\subseteq [\omega_1]^{<\omega}$ is downward closed and for $F_0,...,F_{n-1}$ in $\ms{C}$, if $\tau(F_0)=\cdots=\tau(F_{n-1})$, then $\bigcup_{i<n} F_i\in \ms{C}$.
\end{enumerate}
 
We will use the following natural forcing to force the desired coloring $\pi$ and its $\sigma$-$n$-linked partition.

$\mc{P}$ is the poset consisting of $p=(\pi_p, \sigma_p)$ where $\pi_p: [D_p]^{<n+2}\ra 2$ for some $D_p\in [\omega_1]^{<\omega}$, $\sigma_p: \mc{H}^{\pi_p}_0\ra \omega$ such that 
\begin{enumerate}[(a)]
\item $\ms{C}\cap [D_p]^{\leq n}\subseteq \mc{H}^{\pi_p}_0\subseteq \ms{C}$;
\item for every $m<\omega$, $\sigma_p^{-1}\{m\}$ is an $n$-linked subset of $\mc{H}^{\pi_p}_0$;
\item for every $m<\omega$,  $\sigma^{-1}_p\{m\}\subseteq \tau^{-1}\{k\}$ for some $k$.
\end{enumerate}
For $p, q\in \mc{P}$, $p\leq q$ iff $\pi_p\supseteq \pi_q$ and $\sigma_p\supseteq \sigma_q$.\medskip

To simplify our notation, we denote, for $p\in \mc{P}$,
\begin{enumerate}[(a)]\setcounter{enumi}{3}
\item $\mc{H}_0^p=\mc{H}^{\pi_p}_0$.
\end{enumerate}

To apply MA$_{\omega_1}$($\sigma$-centered), we  will show that $\mc{P}$ is $\sigma$-centered.
\begin{lem}\label{lem 4.1}
$\mc{P}$ is $\sigma$-centered.
\end{lem}

Before proving above lemma, we first show that Proposition \ref{prop code fin by 4} follows from above lemma.

\begin{proof}[Proof of Proposition \ref{prop code fin by 4} from Lemma \ref{lem 4.1}.]
Let
$\mc{D}_\alpha=\{p\in \mc{P}: \alpha\in D_p\}$ for $\alpha<\omega_1$. We claim that each $\mc{D}_\alpha$ is dense. To see this, fix $p\in \mc{P}$ and $\alpha<\omega_1$. 
Say $\alpha\notin D_p$. 

Define $D_q=D_p\cup\{\alpha\}$. Then define $\pi_q: [D_q]^{<n+2}\ra 2$ such that
\[\pi_q^{-1}\{0\}=\pi_p^{-1}\{0\}\cup (\ms{C}\cap [D_q]^{\leq n}).\]
Note $\mc{H}_0^{q}=\mc{H}_0^{p}\cup (\ms{C}\cap [D_q]^{\leq n}\setminus [D_p]^{\leq n})$.
Then define $\sigma_q: \mc{H}_0^q\ra \omega$ such that
\begin{itemize}
\item $\sigma_q\up{\mc{H}_0^p}=\sigma_p$;
\item $\sigma_q\up (\ms{C}\cap [D_q]^{\leq n}\setminus [D_p]^{\leq n})$ is one-to-one  and has range disjoint from $rang(\sigma_p)$.
\end{itemize}
It is straightforward to check that $q\in \mc{P}$ and $q\leq p$. So $q\in \mc{D}_\alpha$ and  $\mc{D}_\alpha$ is dense.

By Lemma \ref{lem 4.1} and MA$_{\omega_1}$($\sigma$-centered), there is a filter $G$ meeting all $\mc{D}_\alpha$'s.
Now let
\[\pi=\bigcup \{\pi_p: p\in G\}\text{ and }\sigma=\bigcup \{\sigma_p: p\in G\}.\]
Then $\pi: [\omega_1]^{<n+2}\ra 2$ is a total function and by (a)-(b) in definition of $\mc{P}$, 
\begin{enumerate}
\item $\ms{C}\cap [\omega_1]^{\leq n}\subseteq \mc{H}^\pi_0\subseteq \ms{C}$;
\item $\sigma: \mc{H}^\pi_0\ra \omega$ is a partition of $\mc{H}^\pi_0$ into countably many $n$-linked subsets.
\end{enumerate}
So $\pi$ is as desired.
\end{proof}

To prove Lemma \ref{lem 4.1}, we fix a partition $\mc{P}=\bigcup_{m<\omega} Q_m$ such that for every $m<\omega$,
\begin{enumerate}[(I)]\setcounter{enumi}{1}
\item for $p, q$ in $Q_m$, $(D_p, e, \pi_p, \sigma_p, \tau)$ is isomorphic to $(D_q, e, \pi_q, \sigma_q, \tau)$.\footnote{Recall the definition of isomorphism  in Definition \ref{defn iso}.}
\end{enumerate}
The mapping $e$ is used, by Lemma \ref{lem initial}, to guarantee that $D_p\cap D_q$ is an initial segment of both $D_p$ and $D_q$.

It suffices to prove that each $Q_m$ is centered. Fix $m<\omega$ and $p_0, ..., p_{k-1}$ in $Q_m$. We will define a condition $q\in \mc{P}$ that is a lower bound of all $p_i$'s.

First let 
\[D_q=\bigcup_{i<k} D_{p_i}.\]
Then define $\pi_q$ by cases. Fix  $a\in [D_q]^{<n+2}$.
\medskip

\textbf{Case 1.} $|a|<n+1$. Then define $\pi_q(a)=0$ iff $a\in \ms{C}$.\medskip

\textbf{Case 2.} For some $s\in k^{\leq n}$ and $\langle a_i\in \mc{H}_0^{p_{s(i)}}: i< |s|\rangle$,  $|a|=n+1$,
\[a\subseteq \bigcup_{i<|s|} a_i\text{ and $\sigma_{p_{s(i)}}(a_i)=\sigma_{p_{s(j)}}(a_{j})$ for }i, j<|s|.\]
 Then define $\pi_q(a)=0$. \medskip

\textbf{Case 3.} $|a|=n+1$ and Case 2 does not occur. Then define $\pi_q(a)=1$.\medskip

Note that Case 2 is necessary for $q$ to satisfy (b). And our $\pi_q$ is defined to minimize $\mc{H}_0^{q}$.

Finally define $\sigma_q: \mc{H}_0^{q}\ra \omega$ to satisfy the following requirements.
\begin{enumerate}[(I)]\setcounter{enumi}{2}
\item If for some $i<k$, $a\in \mc{H}_0^{q}\cap \mc{H}_0^{p_i}$, then $\sigma_q(a)=\sigma_{p_i}(a)$.
\item $\sigma_q\up (\mc{H}_0^{q}\setminus \bigcup_{i<k} \mc{H}_0^{p_i})$ is one-to-one    and has range disjoint from $rang(\sigma_{p_0})$.
\end{enumerate}

We will show that $q=(\pi_q, \sigma_q)$ is in $\mc{P}$ and   a lower bound of all $p_i$'s. First, we check that $q$ is well-defined. $\pi_q$ is clearly well-defined and $\sigma_q$ is well-defined by the following lemma.

\begin{lem}\label{lem 4.1.1}
For $i<j<k$ and $a\in \mc{H}_0^{p_i}\cap \mc{H}_0^{p_j}$, $\sigma_{p_i}(a)=\sigma_{p_j}(a)$.
\end{lem}
\begin{proof}
Recall by (II) and Lemma \ref{lem initial}, $D_{p_i}\cap D_{p_j}$ is an initial segment of both $D_{p_i}$ and $D_{p_j}$.
Together with the fact $a\subseteq D_{p_i}\cap D_{p_j}$, $a=D_{p_i}[I]=D_{p_j}[I]$ for some $I\subseteq |D_{p_i}\cap D_{p_j}|$. Then by (II), 
$\sigma_{p_i}(a)=\sigma_{p_i}(D_{p_i}[I])=\sigma_{p_j}(D_{p_j}[I])=\sigma_{p_j}(a)$.
\end{proof}

Now we check that $q\in \mc{P}$. We will use the following fact to prove  (a) in the definition of $\mc{P}$.    
\begin{lem}\label{lem initial3}
For every $b\in \mc{H}_0^{q}$, there exists   $F\in [k]^{\leq n}$ such that
$b\subseteq \bigcup_{i\in F} D_{p_i}$.
\end{lem}
\begin{proof}
Inductively choose  $\langle \beta_j, s(j): j<l\rangle$  for some $l\leq k$  such that
\begin{enumerate}
\item $\beta_0>\cdot\cdot\cdot>\beta_{l-1}$ are in $b$,  $s(0),...,s(l-1)< k$ are  pairwise distinct and  
\begin{itemize}
\item   for $j<l$, $\beta_{j}=\max (b\setminus \bigcup_{i< j} D_{p_{s(i)}})$ and $\beta_{j}\in D_{p_{s(j)}}$, 
\item $b\subseteq\bigcup_{i<l} D_{p_{s(i)}}$.
\end{itemize}
\end{enumerate}
We will show $l\leq n$. First note that 
\begin{enumerate}\setcounter{enumi}{1}
\item  $\beta_{j'}\notin D_{p_{s(j)}}$ for $j'<j<l$.
\end{enumerate}
To see this,  recall that $D_{p_{s(j')}}\cap D_{p_{s(j)}}$ is an initial segment of $D_{p_{s(j')}}$ and $D_{p_{s(j)}}$.  If $\beta_{j'}\in D_{p_{s(j)}}$, then $\beta_{j'}\in D_{p_{s(j)}}\cap D_{p_{s(j')}}$ and hence
\[\beta_j\in D_{p_{s(j)}}\cap \beta_{j'}\subseteq D_{p_{s(j)}}\cap D_{p_{s(j')}}\subseteq  D_{p_{s(j')}}.\]
But this contradicts the definition of $\beta_j$ in (1). This contradiction shows that (2) holds.

By (1),   $\beta_{j'}\notin D_{p_{s(j)}}$ for $j'>j$. We conclude that
\begin{enumerate}\setcounter{enumi}{2}
\item $D_{p_{s(j)}}\cap \{\beta_i: i<l\}=\{\beta_j\}$.
\end{enumerate}
Recall  that $D_{p_i}\cap D_{p_{s(j)}}$ is an initial segment of $D_{p_i}$ and $D_{p_{s(j)}}$. Consequently, no $D_{p_i}$ can contain both $\beta_{j}$ and $\beta_{j'}$ for $j<j'<l$, since otherwise $D_{p_{s(j)}}$ contains both $\beta_j$ and $\beta_{j'}$   too.
\begin{enumerate}\setcounter{enumi}{3}
\item For $i<k$, $|D_{p_{i}}\cap \{\beta_j: j<l\}|\leq 1$.
\end{enumerate}

We claim that $l\leq n$. Suppose otherwise, $l>n$. Then by (4), Case 2 in the definition of $\pi_q$ cannot occur for $\{\beta_j: j\leq n\}$ and hence $\pi_q(\{\beta_j: j\leq n\})=1$. This contradicts the fact $b\in \mc{H}_0^{q}$. 

Now by (1), $F=\{s(i): i<l\}$ is as desired.
\end{proof}

We now show that $q$ satisfies (a) in definition of $\mc{P}$.
\begin{lem}\label{lem a of P}
$\ms{C}\cap [D_q]^{\leq n}\subseteq \mc{H}_0^q\subseteq \ms{C}$.
\end{lem}
\begin{proof}
$\ms{C}\cap [D_q]^{\leq n}\subseteq \mc{H}_0^q$ follows from our definition of $\pi_q$. Now we show $\mc{H}_0^q\subseteq \ms{C}$. Fix $b\in \mc{H}_0^q$. By definition of $\pi_q$, assume $|b|>n$. By Lemma \ref{lem initial3},   assume that
\begin{enumerate}
\item $b\subseteq \bigcup_{i\in F} D_{p_i}$ for some $F\in [k]^{\leq n}$.
\end{enumerate}
Let $\langle I_i\subseteq |D_{p_i}|: i\in F\rangle$ and $I$ be such that
\[b\cap D_{p_i}=D_{p_i}[I_i] \text{ and } I=\bigcup_{i\in F} I_i.\]

\textbf{Claim.} $D_{p_0}[I]\in \mc{H}_0^{p_0}$.\medskip

\begin{proof}[Proof of Claim.]
 Arbitrarily choose $a\in [D_{p_0}[I]]^{<n+2}$. For $i\in F$, let $I'_i\subseteq I_i$ be such that
\[a\cap D_{p_0}[I_i]=D_{p_0}[I'_i].\]
 And for $i\in F$, let 
\[I''_i=I'_i\setminus \bigcup_{j\in F\cap i} I'_j.\]
 Then $I''_i$'s are pairwise disjoint and $a=\bigcup_{i\in F} D_{p_0}[I''_i]$. Let
\begin{enumerate}\setcounter{enumi}{1}
\item $a'=\bigcup_{i\in F} D_{p_i}[I''_i]$.
\end{enumerate}
By Lemma \ref{lem initial2},   $D_{p_i}[I''_i]$'s are pairwise disjoint. So $|a'|=|\bigcup_{i\in F} I''_i|=|a|<n+2$.

Find $a''\in [b]^{n+1}$ such that
\begin{enumerate}\setcounter{enumi}{2}
\item $a'\subseteq a''$.
\end{enumerate}
Note  $\pi_q(a'')=0$.  By  definition of $\pi_q$, choose $s\in k^{\leq n}$ and $\langle a_j\in \mc{H}_0^{p_{s(j)}}: j<|s|\rangle$ such that 
\begin{enumerate}\setcounter{enumi}{3}
\item $a''\subseteq \bigcup_{j<|s|} a_j\text{ and $\sigma_{p_{s(j)}}(a_j)=\sigma_{p_{s(l)}}(a_{l})$ for }j, l<|s|$.
\end{enumerate}
For $j<|s|$, let 
\begin{enumerate}\setcounter{enumi}{4}
\item $J_j\subseteq |D_{p_{s(j)}}|\text{ be such that } a_j=D_{p_{s(j)}}[J_j]$.
\end{enumerate}
Together with (II),
\[\sigma_{p_{0}}(D_{p_0}[J_j])=\sigma_{p_{0}}(D_{p_0}[J_l])\text{ for }j, l<|s|.\]
By (b) and the fact $|s|\leq n$, 
\begin{enumerate}\setcounter{enumi}{5}
\item $D_{p_0}[\bigcup_{j<|s|} J_j]\in \mc{H}_0^{p_0}$.
\end{enumerate}
Recall by (2)-(5),
\[\bigcup_{i\in F} D_{p_i}[I''_i]=a'\subseteq a''\subseteq \bigcup_{j<|s|} a_j= \bigcup_{j<|s|} D_{p_{s(j)}}[J_j]\] 
 Then by Lemma \ref{lem initial2}, $I''_i\subseteq \bigcup_{j<|s|} J_j$ and hence  
 \[\bigcup_{i\in F} I''_i\subseteq \bigcup_{j<|s|} J_j.\]
Together with (6) and $a=\bigcup_{i\in F} D_{p_0}[I''_i]$, we conclude that $\pi_{p_0}(a)=0$.

Since $a$ is arbitrary in $[D_{p_0}[I]]^{<n+2}$, $D_{p_0}[I]\in \mc{H}_0^{p_0}$. 
\end{proof}

Now by (a) for $p_0$ and (II),   $D_{p_i}[I]\in \ms{C}=dom(\tau)$ for $i\in F$ and 
\[\tau(D_{p_i}[I])=\tau(D_{p_{j}}[I]) \text{ for $i, j$ in } F.\] 
By (I) and the fact $|F|\leq n$,  $\bigcup_{i\in F} D_{p_i}[I]\in \ms{C}$.
So $b\subseteq \bigcup_{i\in F} D_{p_i}[I]$ is in $\ms{C}$.
\end{proof}

Then we check (b)-(c) for $q$.
\begin{lem}\label{lem q in P}
$q\in \mc{P}$.
\end{lem}
\begin{proof}
We first check (c) for $q$. Fix $l\in rang(\sigma_q)$. 

If $l\notin rang(\sigma_{p_0})$, then by (IV), $\sigma^{-1}_q\{l\}$ is a singleton and hence (c) is satisfied.

Assume $l\in rang(\sigma_{p_0})$. By (III) and Lemma \ref{lem 4.1.1}, $\sigma^{-1}_q\{l\}=\bigcup_{i<k} \sigma^{-1}_{p_i}\{l\}$. By (II) and (c) for $p_i$, $\tau$ is constant on $\bigcup_{i<k} \sigma^{-1}_{p_i}\{l\}$. So (c) is satisfied for $q$.

We then check (b) for $q$. Fix $l\in rang(\sigma_q)$. 
If $l\notin rang(\sigma_{p_0})$, then by (IV), $\sigma_q^{-1}\{l\}$ is a singleton and hence $n$-linked.

Now assume $l\in rang(\sigma_{p_0})$. Choose $\langle a_i\in \sigma_q^{-1}\{l\}: i<n\rangle$. By (III) and (IV), for each $i<n$, there exists $s(i)<k$ such that $a_i\in \mc{H}_0^{p_{s(i)}}$ and $\sigma_{p_{s(i)}}(a_i)=\sigma_q(a_i)=l$.

Note by (c) for $q$, $\tau(a_0)=\cdots=\tau(a_{n-1})$. Then by (I), 
\begin{enumerate}
\item $\bigcup_{i<n} a_i\in \ms{C}$.
\end{enumerate}
Fix  $a\in [\bigcup_{i<n} a_i]^{<n+2}$. If $|a|<n+1$, then $\pi(a)=0$ by definition of $\pi_q$ and (1).
Now assume $|a|=n+1$. Then $\pi_q(a)=0$ witnessed by $s$ and $\langle a_i\in \mc{H}_0^{p_{s(i)}}: i<n\rangle$. So $\bigcup_{i<n} a_i\in \mc{H}_0^{q}$ and hence $\sigma_q^{-1}\{l\}$ is $n$-linked.\medskip

Now (b)-(c) are all satisfied for $q$. Together with Lemma \ref{lem a of P}, $q\in \mc{P}$.
\end{proof}

Then we check that  $q$ is a common lower bound of $p_0,...,p_{k-1}$. 

\begin{lem}\label{lem lower bound}
For every $i<k$, $q\leq p_i$.
\end{lem}
\begin{proof}
Fix $i<k$. First note by  (III) and Lemma \ref{lem 4.1.1}, $\sigma_q\supseteq \sigma_{p_i}$. To show $\pi_q\supseteq \pi_{p_i}$, note by (a),
   it suffices to show $\pi_q\up{[D_{p_i}]^{n+1}}=\pi_{p_i}$.

 Fix $a\in [D_{p_i}]^{n+1}$. Assume $a=D_{p_i}[I]$ for some $I$. It suffices to prove that $\pi_{p_i}(a)=0$ iff $\pi_{q}(a)=0$.

 If $\pi_{p_i}(a)=0$, then by definition of $\pi_q$ Case 2, $\pi_q(a)=0$.
 
Now suppose $\pi_q(a)=0$. Then $a=D_{p_i}[I]\in \mc{H}^q_0$.
By the Claim of Lemma \ref{lem a of P}, $D_{p_0}[I]\in \mc{H}^{p_0}_0$ and hence $D_{p_i}[I]\in \mc{H}^{p_i}_0$. So $\pi_{p_i}(a)=0$.
\end{proof}

\begin{proof}[Proof of Lemma \ref{lem 4.1}.]
For the partition $\mc{P}=\bigcup_{m<\omega} Q_m$ satisfying (II), every $Q_m$ is centered by Lemma \ref{lem q in P} and Lemma \ref{lem lower bound}. So $\mc{P}$ is $\sigma$-centered.
\end{proof}

We will need the following property which follows from Proposition \ref{prop code fin by 4}.
\begin{thm}\label{thm sK3toK4 sK3top}
Assume {\rm MA$_{\omega_1}$($\sigma$-centered)}. For $n\geq 2$, if every $\sigma$-$n$-linked poset has property {\rm K}$_{n+1}$, then every $\sigma$-$n$-linked poset has precaliber $\omega_1$. 
\end{thm}
\begin{proof}
Fix a $\sigma$-$n$-linked poset $\mathbb{P}$ of size $\omega_1$ and an uncountable subset $X\subseteq \mathbb{P}$. Denote $\ms{C}=\{F\in [\mathbb{P}]^{<\omega}: F$ is centered$\}$. It is clear that $\ms{C}$ is downward closed and $\sigma$-$n$-linked.

Let $\pi: [\mathbb{P}]^{<n+2}\ra 2$ be a coloring guaranteed by Proposition \ref{prop code fin by 4}. Then $\mc{H}^\pi_0$ has property K$_{n+1}$.

Note that $[X]^1=\{\{x\}: x\in X\}\subseteq \mc{H}^\pi_0$. Find $Y\in [X]^{\omega_1}$ such that $[Y]^1$ is $(n+1)$-linked in $\mc{H}^\pi_0$. 
Then $Y$ is 0-homogeneous over $\pi$. By Proposition \ref{prop code fin by 4} (ii), 
\[[Y]^{<\omega}\subseteq \mc{H}^\pi_0\subseteq \ms{C}.\]
Then $Y$ is an uncountable  centered subset of $\mathbb{P}$. So $\mathbb{P}$ has precaliber $\omega_1$. 
 \end{proof}

We have the following different forms of consequences of Proposition \ref{prop code fin by 4}.
 \begin{cor}
 Assume {\rm MA$_{\omega_1}$($\sigma$-centered)} and $2\leq n<\omega$. If $\mathbb{P}$ is a $\sigma$-$n$-linked poset of size $\omega_1$, then there is a coloring $\pi: [\mathbb{P}]^{<n+2}\ra 2$ with the following properties.
 \begin{enumerate}[(i)]
 \item $\mc{H}^\pi_0$ is $\sigma$-$n$-linked.
 \item $\{F\in [\mathbb{P}]^{<n+1}: F$ is centered$\}\subseteq \mc{H}^\pi_0\subseteq \{F\in [\mathbb{P}]^{<\omega}: F$ is centered$\}$.
  \end{enumerate}
  \end{cor}

  \begin{cor}\label{cor coloring fin by 4}
 Assume {\rm MA$_{\omega_1}$($\sigma$-centered)} and $2\leq n<\omega$. If $f: [\omega_1]^{<\omega}\ra 2$ is a downward closed coloring with $\mc{H}^f_0$ $\sigma$-$n$-linked, then there is a coloring $\pi: [\omega_1]^{<n+2}\ra 2$ with the following properties.
 \begin{enumerate}[(i)]
 \item $\mc{H}^\pi_0$ is $\sigma$-$n$-linked.
 \item $\mc{H}^f_0\cap [\omega_1]^{<n+1}\subseteq \mc{H}^\pi_0\subseteq \mc{H}^f_0$.
  \end{enumerate}
  \end{cor}


The statement that $\ms{K}_3$ implies MA$_{\omega_1}$ is clearly a corollary of Theorem \ref{thm K3toK4}. But at this point, we already have a proof for this statement.

\begin{proof}[Proof of Theorem \ref{thm K3toMA}.]
Firstly by Theorem \ref{thm K3 t} and Theorem \ref{thm BRMS}, MA$_{\omega_1}$($\sigma$-centered) holds.

Secondly by Theorem \ref{thm sK3toK4 sK3top}, every $\sigma$-linked poset has precaliber $\omega_1$.

Then by \cite[Corollary 2.6]{TV}, every ccc poset of size $\omega_1$ is $\sigma$-linked. So every ccc poset has precaliber $\omega_1$. 

Finally by \cite[Theorem 3.4]{TV}, {\rm MA}$_{\omega_1}$ holds.
\end{proof}

\section{Coding a poset by an uncountable homogeneous set}
In this section, we generalize the coding in \cite{TV} that codes a ccc poset by 0-homogeneous subset of a coloring.  New structures will be used to get stronger properties on posets.

Throughout this section, we fix $n\geq 2$ and  a sequence of distinct reals $\{r_\alpha\in 2^\omega: \alpha<\omega_1\}$. For $x\neq y$ in $2^\omega$,
\[\Delta(x,y)=\min\{m: x(m)\neq y(m)\}.\]
The following coding $c$   defined in \cite{Todorcevic87} and \cite{TV}  will be used.
\begin{defn}
For $\alpha<\beta<\omega_1$, $\Delta(e_\alpha, e_\beta)=\min(\{\xi<\alpha: e(\xi, \alpha)\neq e(\xi, \beta)\}\cup \{\alpha\})$.
\end{defn}
Define $c': [\omega_1]^2\ra \omega_1$ by
\[c'(\alpha, \beta)=\min(\{\xi\in (\alpha, \beta]: e(\xi, \beta)\leq e(\Delta(e_\alpha, e_\beta), \beta)\}).\]
We add a boundary value $e(\beta, \beta)=0$ so that $c'(\alpha, \beta)$ is always defined. Let
\[f: \omega_1\ra \omega\]
be a partition of $\omega_1$ into countably many stationary sets. Then define $c: [\omega_1]^2\ra \omega_1$ by 
\begin{equation*}
c(\alpha, \beta)=
\begin{cases}
e_\beta^{-1}(f(c'(\alpha, \beta))) & \text{ if } f(c'(\alpha, \beta))\in rang(e_\beta)\\
\text{undefined} & \text{otherwise.}
\end{cases}
\end{equation*}
The following property of $c$ is needed (see \cite[Lemma 2.3]{TV}).
\begin{lem}[\cite{Todorcevic87},\cite{TV}]\label{lem TV}
For every $X\in [\omega_1]^{\omega_1}$, there exists $\delta<\omega_1$ such that for any $\xi<\omega_1$, there are $\alpha\in X\cap \delta$ and $\beta\in X\setminus \delta$ with $c(\alpha, \beta)=\xi$.
\end{lem}
The following slight generalization of above lemma will be used in     Subsection 5.2.

\begin{lem}[\cite{Todorcevic87},\cite{TV}]\label{lem 52}
For every $X\in [\omega_1]^{\omega_1}$, there exists $\delta<\omega_1$ such that for any $\xi<\omega_1$ and $m<\omega$, there are $\alpha_0<\alpha_1<\cdot\cdot\cdot<\alpha_m$ in $X\cap \delta$ and $\alpha_{m+1}\in X\setminus \delta$ with 
\[e(\alpha_0, \alpha_1)>e(\alpha_1, \alpha_2)>\cdot\cdot\cdot>e(\alpha_{m}, \alpha_{m+1})\text{ and }c(\alpha_m, \alpha_{m+1})=\xi.\]
\end{lem}
\begin{proof}
Choose $\eta<\omega_1$ such that $X\cap \eta$ has order type $\geq \omega^2$. Then let $\delta<\omega_1$ satisfy the conclusion of Lemma \ref{lem TV} for $X\setminus \eta$. To show that this $\delta$ works, fix $\xi<\omega_1$ and $m<\omega$.

By Lemma \ref{lem TV}, find $\alpha_m\in (X\setminus \eta)\cap \delta$ and $\alpha_{m+1}\in X\setminus \delta$ such that $c(\alpha_m, \alpha_{m+1})=\xi$. Then find $\eta_0<\eta_1<\cdot\cdot\cdot<\eta_{m}<\eta$ such that for every $j<m$,  $X\cap [\eta_j, \eta_{j+1})$ is infinite. Then by (coh1), find the desired $\alpha_j\in X\cap [\eta_j, \eta_{j+1})$ by reverse induction on $j$.
\end{proof}

\subsection{{\rm P($\sigma$-$n$-linked$\ra $precaliber $\omega_1$)} implies {\rm P$_{\omega_1}$($\sigma$-$n$-linked$\ra \sigma$-centered)}}
Throughout this subsection, we assume {\rm P($\sigma$-$n$-linked$\ra $precaliber $\omega_1$)}. And we will show that {\rm P$_{\omega_1}$($\sigma$-$n$-linked$\ra \sigma$-centered)}. Then  {\rm MA$_{\omega_1}(\sigma$-$n$-linked)} holds by Theorem \ref{thm sK3toK4 sK3top} and Proposition \ref{prop sK3toK4 t}.

Fix a $\sigma$-$n$-linked poset of size $\omega_1$
\[\mathbb{P}=\{p_\alpha: \alpha<\omega_1\} \text{ and  a partition }\mathbb{P}=\bigcup_{m<\omega} P_m\]
 of $\mathbb{P}$ into countably many $n$-linked subsets.  Assume in addition that 
 \[ P_m\cap P_k=\emptyset \text{ for } m\neq k.\]
 Let
\[\ms{C}=\{F\in [\mathbb{P}]^{<\omega}: F\text{ is centered$\}$ and } \tau: \ms{C}\ra \omega\]
be a partition such that every $F\in \ms{C}$ has a common lower bound in $P_{\tau(F)}$.

We then have the following property.
\begin{enumerate}[(I)]
\item $\ms{C}$ is downward closed and for $F_0,...,F_{n-1}$ in $\ms{C}$, if $\tau(F_0)=\cdots=\tau(F_{n-1})$, then $\bigcup_{i<n} F_i\in \ms{C}$.
\end{enumerate}

A strong property, $\sigma$-$n$-linked, will be attached to the coloring we shall define. So we need an extra structure of $e$.

For $F\in [\omega_1]^{<\omega}$  and $m<\omega$, say $F$ is \emph{$m$-closed} if for every $\beta\in F$, $e_\beta^{-1}[m+1]\subseteq F$. 
 Let
 $cl(F, m)$ be the smallest  $m$-closed set containing $F$.
  
  Note for $F\in [\omega_1]^{<\omega}$  and $m<\omega$, $cl(F, m)$ is finite. To see this, let $T$ be the collection of $s\in \omega_1^{<\omega}$ such that $s(0)\in F$ and $s(i+1)\in e_{s(i)}^{-1}[m+1]$ for $i<|s|-1$.  Then $T$ is a finite branching tree without infinite branches since every branch is a decreasing sequence of ordinals. So $T$ is finite. Then $cl(F, m)=\{s(i): s\in T, i<|s|\}$ is finite.

For $F\in [\omega_1]^{<\omega}$, $s\in 2^{<\omega}$ and $m<\omega$, let
\[P(F, s, m)=\{p_\gamma: \exists \alpha, \beta \in F ~ (c(\alpha,  \beta)=\gamma\wedge r_\beta\up{e(\alpha, \beta)}=s \wedge p_\gamma\in P_m)\}.\]
Define a coloring $\pi: [\omega_1]^{<\omega}\ra 2$ by $\pi(F)=0$ iff
\[\forall s\in 2^{<\omega} \forall m<\omega ~ P(F, s, m)\in \ms{C}.\]

\begin{lem}
$\mc{H}^\pi_0$ is $\sigma$-$n$-linked.
\end{lem}
\begin{proof}
Fix a partition $\mc{H}^\pi_0=\bigcup_{N<\omega} H_N$ such that the following properties hold: For every $N$, there exists $M<\omega$ such that for $E, F\in H_N$,
\begin{enumerate}
\item $\max(e[[E]^2]\cup \Delta[[\{r_\alpha: \alpha\in E\}]^2])=M$;
\item for $i<|E|$, $r_{E(i)}\up M= r_{F(i)}\up M$;
\item $\langle \tau(P(E, s, m)): s\in 2^{<\omega}, m<\omega\rangle
=\langle \tau(P(F, s, m)): s\in 2^{<\omega}, m<\omega\rangle$;
\item $(cl(E, M), \chi_E, e)$ is isomorphic to $(cl(F, M), \chi_F, e)$.
\end{enumerate}
In (4), $\chi_E$ is the characteristic function of $E$ and is used to guarantee that the positions of $E$ in $cl(E, M)$ are the same as positions of $F$ in $cl(F, M)$. 

Note also that for all but finitely many $s$ and $m$, $P(E, s, m)=\emptyset$. So the collection of sequences in (3), $\{\langle \tau(P(E, s, m)): s\in 2^{<\omega}, m<\omega\rangle: E\in \mc{H}^\pi_0\}$, is countable.

Now it suffices to show that each $H_N$ is $n$-linked. Fix $N<\omega$ and $F_0,...,F_{n-1}$ in $H_N$. Let \[M=\max(e[[F_0]^2]\cup \Delta[[\{r_\alpha: \alpha\in F_0\}]^2])\text{ and } F=\bigcup_{i<n} F_i.\]
 Fix $s\in 2^{<\omega}$ and $m<\omega$.\medskip

\textbf{Case 1.}  $|s|\leq M$.\medskip

 We will show that $P(F, s, m)=\bigcup_{i<n} P(F_i, s, m)$. Clearly, $P(F, s, m)\supseteq \bigcup_{i<n} P(F_i, s, m)$. To show $P(F, s, m)\subseteq \bigcup_{i<n} P(F_i, s, m)$, fix $p_\gamma\in P(F,s,m)$ witnessed by $\alpha<\beta$ in $F$.

 Fix $i, j<n$ such that $\alpha\in F_i$ and $\beta\in F_j$. Since $e(\alpha, \beta)=|s|\leq M$, $\alpha\in cl(F_j, M)$. In particular, 
 \[\alpha\in cl(F_i, M)\cap cl(F_j, M).\]
  By (4) and Lemma \ref{lem initial}, $cl(F_i, M)\cap cl(F_j, M)$ is an initial segment of $cl(F_i, M)$ and $cl(F_j, M)$. Together with (4) and the fact $\alpha\in F_i$, we conclude that $\alpha\in F_j$.
 
 So $p_\gamma\in P(F_j, s, m)$. This shows that $P(F, s, m)=\bigcup_{i<n} P(F_i, s, m)$. Together with (3) and (I), $P(F, s, m)\in \ms{C}$.\medskip

\textbf{Case 2.}  $|s|>M$.\medskip

 We will show that $|P(F, s, m)|\leq n$.  Suppose otherwise, $p_{\gamma_0},...,p_{\gamma_n}$ are in $P(F, s, m)$ witnessed by $(\alpha_0, \beta_0),...,(\alpha_n, \beta_n)$ respectively.

Recall that $r_{\beta_0}\up |s|=\cdot\cdot\cdot=r_{\beta_n}\up |s|=s$. By definition of $M$ and the fact $|s|>M$, $|\{\beta_i: i\leq n\}\cap F_j|\leq 1$ for every $j<n$. By the Pigeonhole Priciple,  $\beta_i=\beta_{i'}$ for some $i<i'$. Then $e(\alpha_i, \beta_i)=|s|=e(\alpha_{i'}, \beta_{i'})=e(\alpha_{i'}, \beta_i)$. This contradicts the fact that $e_{\beta_i}$ is one-to-one.

This contradiction shows that $|P(F, s, m)|\leq n$. By definition of $P(F, s, m)$, $P(F, s, m)\subseteq P_m$. So $P(F, s, m)\in \ms{C}$.\medskip

In any case, $P(F, s, m)\in \ms{C}$. So $F\in \mc{H}^\pi_0$. This shows that $H_N$ is $n$-linked and $\mc{H}^\pi_0$ is $\sigma$-$n$-linked.
\end{proof}

Now by {\rm P($\sigma$-$n$-linked$\ra $precaliber $\omega_1$)}, $\pi$ has an uncountable 0-homogeneous subset $X$. By Lemma \ref{lem TV},  $c[[X]^2]=\omega_1$. For  $s\in 2^{<\omega}$ and $m<\omega$, let
\[\ms{C}_{ s, m}=\bigcup\{P(F, s, m): F\in [X]^{<\omega}\}.\]
Then $\mathbb{P}=\bigcup \{\ms{C}_{s, m}:  s\in 2^{<\omega}, m<\omega\}$ and $[\ms{C}_{s, m}]^{<\omega}\subseteq \ms{C}$ for all $s$ and $m$. So $\mathbb{P}$ is $\sigma$-centered.

Above argument actually show the following property.
\begin{lem}\label{lem sK3top sK3tos}
For every $\sigma$-$n$-linked poset $\mathbb{P}$   of size $\omega_1$,   there is a downward closed coloring $\pi: [\omega_1]^{<\omega}\ra 2$ with the following properties.
\begin{enumerate}[(i)]
\item $[\omega_1]^2\subseteq \mc{H}^\pi_0$ and $\mc{H}^\pi_0$ is $\sigma$-$n$-linked;
\item If $\pi$ has an uncountable 0-homogeneous subset, then $\mathbb{P}$ is $\sigma$-centered.
\end{enumerate}
\end{lem}

Now we are ready to prove Theorem \ref{thm sK3toK4}.
\begin{proof}[Proof of Theorem \ref{thm sK3toK4}.]
By Theorem \ref{thm sK3toK4 sK3top}  and Lemma \ref{lem sK3top sK3tos},   P$_{\omega_1}(\sigma$-$n$-linked$\ra\sigma$-centered) holds. Together with Proposition \ref{prop sK3toK4 t}, MA$_{\omega_1}(\sigma$-$n$-linked) holds.
\end{proof}

Combine Corollary \ref{cor coloring fin by 4} with above argument to reduce arity of colorings.
\begin{cor}
 Assume {\rm MA$_{\omega_1}$($\sigma$-centered)} and $2\leq n<\omega$. For every $\sigma$-$n$-linked poset $\mathbb{P}$   of size $\omega_1$,   there is a coloring $\pi: [\omega_1]^{<n+2}\ra 2$ with the following properties.
 \begin{enumerate}[(i)]
 \item $\mc{H}^\pi_0$ is $\sigma$-$n$-linked and $ [\omega_1]^{2}\subseteq \mc{H}^\pi_0$.
 \item If $\pi$ has an uncountable 0-homogeneous subset, then $\mathbb{P}$ is $\sigma$-centered.
  \end{enumerate}
\end{cor}
\begin{proof}
First fix a coloring $\pi': [\omega_1]^{<\omega}\ra 2$ witnessing Lemma \ref{lem sK3top sK3tos}. So
\begin{enumerate}
\item $ [\omega_1]^{2}\subseteq \mc{H}^{\pi'}_0$  and $\mc{H}^{\pi'}_0$ is $\sigma$-$n$-linked;
\item if $\pi'$ has an uncountable 0-homogeneous subset, then $\mathbb{P}$ is $\sigma$-centered.
\end{enumerate}
Then   find $\pi: [\omega_1]^{<n+2}\ra 2$ witnessing Corollary \ref{cor coloring fin by 4} for $\pi'$.

Then (i) holds. Just note that $[\omega_1]^2\subseteq \mc{H}^{\pi'}_0$. And (ii) follows from Corollary \ref{cor coloring fin by 4} (ii) and (2).
\end{proof}

\subsection{{\rm P(K$_n\ra$K$_{n+1}$)} implies {\rm MA$_{\omega_1}$(K$_n$)}}

Throughout this subsection, we  assume {\rm P(K$_n\ra$K$_{n+1}$)}. We will show P$_{\omega_1}$(K$_n\ra\sigma$-$n$-linked). This, together with Theorem \ref{thm sK3toK4}, shows MA$_{\omega_1}$(K$_n$).


We view every poset as 1-linked and hence $\sigma$-1-linked. We then prove for $1\leq i<n$,
\begin{itemize}
\item {\rm P(K$_n\ra$K$_{n+1}$)} implies {\rm P$_{\omega_1}$((K$_n\wedge \sigma$-$i$-linked)$\ra\sigma$-$(i+1)$-linked)}.
\end{itemize}

For the rest of this subsection, we fix $1\leq i^*<n$ and a  $\sigma$-$i^*$-linked  poset 
\[\mathbb{P}=\{p_\alpha: \alpha<\omega_1\}\]
 of size $\omega_1$ that has property K$_n$. We fix the following notations.

\begin{enumerate}[(a)] 
\item $\mathbb{P}=\bigcup_{m<\omega} P_m$
is a partition of $\mathbb{P}$ into countably many $i^*$-linked subsets such that
\[P_m\cap P_k=\emptyset \text{ for } m\neq k.\]
\item $\ms{C}=\{F\in [\mathbb{P}]^{<\omega}: F$ is centered$\}$.
\item $\tau: \ms{C}\ra \mathbb{P}$ 
is a map sending $F\in \ms{C}$ to a common lower bound in $\mathbb{P}$. 
\end{enumerate}

Before defining the coloring that induces a $\sigma$-$(i^*+1)$-linked partition, we first prove the following fact about $e$.
\begin{lem}\label{lem 51}
For $m<\omega$ and $\alpha_0<\cdot\cdot\cdot<\alpha_m<\omega_1$, if $e(\alpha_0, \alpha_1)>e(\alpha_1, \alpha_2)>\cdot\cdot\cdot>e(\alpha_{m-1}, \alpha_m)$, then $e(\alpha_0, \alpha_1)=e(\alpha_0, \alpha_m)$. Consequently, for $i< m$, $e(\alpha_i, \alpha_{i+1})=e(\alpha_i, \alpha_{i+2})=\cdots=e(\alpha_i,\alpha_m)$.
\end{lem} 
\begin{proof}
We prove by induction on $m$.
The case $m=2$ follows from (coh2).

Suppose the lemma holds for $m=k$ and we prove for $m=k+1$. By induction hypothesis, $e(\alpha_0, \alpha_k)=e(\alpha_0, \alpha_1)>e(\alpha_k, \alpha_{k+1})$. By (coh2), $e(\alpha_0, \alpha_k)=e(\alpha_0, \alpha_{k+1})$. Hence, $e(\alpha_0, \alpha_1)=e(\alpha_0, \alpha_{k+1})$.
\end{proof}

For $F\in [\omega_1]^{<\omega}$, $s\in 2^{<\omega}$, $a\in \omega^{n-i^*}$ and $m<\omega$, 
\begin{enumerate}[(a)]\setcounter{enumi}{3}
\item $P(F, s, a, m)$ is the collection of $p_\beta$ such that for some $\alpha_0<\cdots<\alpha_{n-i^*}$ in $F$,
\begin{enumerate}[(d1)]
\item $\beta=c(\alpha_{n-i^*-1}, \alpha_{n-i^*})$;
\item $e(\alpha_0, \alpha_1)>e(\alpha_1, \alpha_2)>\cdots>e(\alpha_{n-i^*-1}, \alpha_{n-i^*})$;
\item $s=r_{\alpha_{n-i^*}}\up e(\alpha_0, \alpha_1)$;
\item for $j<n-i^*$, $a(j)=e(\alpha_j, \alpha_{j+1})$;
\item $p_\beta\in P_m$.
\end{enumerate}
\end{enumerate}
Note that for $F\subseteq E$, $P(F,s,a,m)\subseteq P(E,s,a,m)$.
The following coloring will be used to induce a $\sigma$-$(i^*+1)$-linked partition of $\mathbb{P}$.
\begin{enumerate}[(a)]\setcounter{enumi}{4}
\item Define $\pi: [\omega_1]^{<\omega}\ra 2$ by $\pi(F)=0$ iff
\[\forall s\in 2^{<\omega} \forall a\in \omega^{n-i^*} \forall m<\omega ~ P(F,s,a,m)\in \ms{C}.\]
\end{enumerate}
Note that  $\pi^{-1}\{0\}$ is downward closed.

\begin{lem}
$\mc{H}^\pi_0$ has property {\rm K}$_n$.
\end{lem}
\begin{proof}
Fix $\{F_\alpha\in \mc{H}^\pi_0: \alpha<\omega_1\}$. First find $\Gamma\in [\omega_1]^{\omega_1}$, $M, N<\omega$, $\overline{F}\in [\omega_1]^{<\omega}$ and $\mc{A}\in [2^{<\omega}\times \omega^{n-i^*}\times \omega]^{<\omega}$ such that
\begin{enumerate}
\item $\{F_\alpha: \alpha\in \Gamma\}$ forms a $\Delta$-system with root $\overline{F}$ and each $F_\alpha$ has size $M$;
\item for $\alpha\in \Gamma$, $\mc{A}=\{(s,a, m)\in 2^{<\omega}\times \omega^{n-i^*}\times \omega: P(F_\alpha, s, a, m)\neq\emptyset\}$;
\item for $\alpha\in \Gamma$ and $\xi<\eta$ in $F_\alpha$, $e(\xi, \eta)<N$ and $\Delta(r_\xi, r_\eta)<N$;
\end{enumerate}

Recall that by (a), $P_m$'s are pairwise disjoint. So for every $\alpha$, $\{(s,a, m)\in 2^{<\omega}\times \omega^{n-i^*}\times \omega: P(F_\alpha, s, a, m)\neq\emptyset\}$ is finite. Hence there is a finite $\mc{A}$ satisfying (2).
Going to an uncountable subset of $\Gamma$ if necessary, we may assume that
\begin{enumerate}\setcounter{enumi}{3}
\item for $\alpha<\beta$ in $\Gamma$, $\max (\overline{F})< \min (F_\alpha\setminus \overline{F})$ and $\max (F_\alpha)<\min (F_\beta\setminus \overline{F})$;
\item for $\alpha<\beta$ in $\Gamma$,
\begin{itemize}
\item $e(\xi, \eta)>N$ for $\xi\in F_\alpha\setminus \overline{F}$ and $\eta\in F_\beta\setminus \overline{F}$;
\item $(F_\alpha, e)$ is isomorphic to $(F_\beta, e)$;
\item  $r_{F_\alpha(j)}\up N=r_{F_\beta(j)}\up N$ for $j<M$.
\end{itemize}
\end{enumerate}

By the fact that $\mathbb{P}$ has property K$_n$, a $\frac{M(M-1)}{2}$ step induction induces $\Gamma'\in [\Gamma]^{\omega_1}$ such that
\begin{enumerate}\setcounter{enumi}{5}
\item for $j<k<M$, $X_{j,k}=\{p_\beta: \exists \alpha\in \Gamma' ~ c(F_\alpha(j), F_\alpha(k))=\beta\}$ is $n$-linked.
\end{enumerate}
Then a $|\mc{A}|$ step induction induces $\Gamma''\in [\Gamma']^{\omega_1}$ such that
\begin{enumerate}\setcounter{enumi}{6}
\item for $(s,a,m)\in \mc{A}$, $Y_{s, a,m}=\{\tau(P(F_\alpha, s,a,m)): \alpha\in \Gamma''\}$ is $n$-linked.
\end{enumerate}
It suffices to prove that $\{F_\alpha: \alpha\in \Gamma''\}$ is $n$-linked. For this, fix $x\in [\Gamma'']^n$ and let
\[F=\bigcup_{\xi\in x} F_\xi.\]
We will show $\pi(F)=0$. For this, 
fix $(s,a, m)\in 2^{<\omega}\times \omega^{n-i^*}\times \omega$ such that $P(F, s, a, m)\neq\emptyset$.\medskip

\textbf{Case 1.} $a(0)<N$.

Choose $p_\beta\in P(F, s,a, m)$ witnessed by $\alpha_0<\cdot\cdot\cdot<\alpha_{n-i^*}$ in $F$.  By (d2), for all $j<n-i^*$, $e(\alpha_j, \alpha_{j+1})\leq e(\alpha_0, \alpha_1)=a(0)<N$. Then by (4) and (5), for some $\xi\in x$, $\{\alpha_j: j\leq n-i^*\}\subseteq F_\xi$. So $p_\beta\in P(F_\xi, s, a, m)$.

Above argument shows that 
\[P(F, s, a, m)=\bigcup_{\xi\in x} P(F_\xi, s, a, m).\]
 By (2) and the fact $P(F, s, a, m)\neq\emptyset$, $P(F_\xi, s, a, m)\neq\emptyset$ for all $\xi\in F$. So  $(s,a, m)\in \mc{A}$ by (2).  Then by (7) and the fact $|x|=n$, $\bigcup_{\xi\in x} P(F_\xi, s, a, m)$  is centered.  Hence $P(F, s, a, m)\in \ms{C}$.\medskip

\textbf{Case 2.} $a(n-i^*-1)<N\leq a(0)$.

First note by (d3) and (d4), $|s|=a(0)\geq N$. Let $k^*$ be the unique $k<m$, if exists, such that
\[r_{F_{x(0)}(k)}\up N=s\up N.\]
By (5), if $k^*$ exists, then
\begin{enumerate}\setcounter{enumi}{7}
\item for every $\alpha\in F$, if $r_\alpha\up a(0)=s$,  then $\alpha\in \{F_\xi(k^*): \xi\in x\}$.
\end{enumerate}

Let $j^*$ be the unique $j<k^*$, if exists, such that 
\[e(F_{x(0)}(j), F_{x(0)}(k^*))=a(n-i^*-1).\]
Note   by (coh1), there is at most one such $j$.  And by (5), 
\begin{enumerate}\setcounter{enumi}{8}
\item $e(F_\xi(j^*), F_\xi(k^*))=a(n-i^*-1)$ for all $\xi\in x$.
\end{enumerate}

Now arbitrarily choose $p_\beta\in P(F, s,a, m)$ witnessed by $\alpha_0<\cdot\cdot\cdot<\alpha_{n-i^*}$ in $F$.  By (d4), 
\[e(\alpha_{n-i^*-1}, \alpha_{n-i^*})=a(n-i^*-1)<N\leq a(0)=e(\alpha_0, \alpha_1).\]
Then 
By (3), for some $\xi<\eta$ in $x$, $\alpha_0\in F_\xi\setminus \overline{F}$ and $\alpha_1\in F_\eta\setminus \overline{F}$. And by (5), for some $\delta\geq \eta$ in $x$, $\{\alpha_{n-i^*-1}, \alpha_{n-i^*}\}\subseteq F_\delta$.  Since $\alpha_{n-i^*-1}\geq \alpha_0>\max(\overline{F})$, $\{\alpha_{n-i^*-1}, \alpha_{n-i^*}\}\subseteq F_\delta\setminus \overline{F}$.

By (d3) and (d4), $s=r_{\alpha_{n-i^*}}\up a(0)$.
So by (5), $k^*$ exists and by (8), 
\[\alpha_{n-i^*}=F_\delta(k^*).\]
 Then   $j^*$  exists and 
 \[\alpha_{n-i^*-1}=F_\delta(j^*).\]
  Together with (d1), 
  \[\beta=c(F_\delta(j^*), F_\delta(k^*)). \]

This shows 
\[P(F, s, a, m)\subseteq \{p_\beta: \exists \xi \in x ~ \beta=c(F_\xi(j^*), F_\xi(k^*))\}\subseteq X_{j^*, k^*}.\]
By (6) and the fact $|x|=n$, $\{p_\beta: \exists \xi \in x ~ \beta=c(F_\xi(j^*), F_\xi(k^*))\}$ is centered. Hence, $P(F, s, a, m)\in \ms{C}$.\medskip

\textbf{Case 3.} $a(n-i^*-1)\geq N$.

Note $|s|=a(0)\geq N$. Let $k^*$ be the unique $k<m$, if exists, such that 
\[r_{F_{x(0)}(k)}\up N=s\up N.\]

Choose $p_\beta\in P(F, s,a,m)$ witnessed by $\alpha_0<\cdot\cdot\cdot<\alpha_{n-i^*}$ in $F$. 

By (d2), $e(\alpha_j, \alpha_{j+1})\geq e(\alpha_{n-i^*-1}, \alpha_{n-i^*})\geq N$ for all $j<n-i^*$. So by (3), different $\alpha_j$'s belong to different $F_\xi\setminus \overline{F}$'s. In particular, 
\[\alpha_{n-i^*}\in F_{x(l)}\setminus \overline{F}\text{ for some }n-i^*\leq l<n.\]

By (d3), $r_{\alpha_{n-i^*}}\supseteq s$. Then by (5), $k^*$ exists  and
$\alpha_{n-i^*}=F_{x(l)}(k^*)$.

In summary,
\begin{enumerate}\setcounter{enumi}{9}
\item $\alpha_{n-i^*}\in \{F_{x(l)}(k^*): n-i^*\leq l<n\}$.
\end{enumerate}

By (d2) and Lemma \ref{lem 51},   
$e(\alpha_j, \alpha_{n-i^*})=a(j)\text{ for }j<n-i^*$.
 So by the fact that $e_{\alpha_{n-i^*}}$ is one-to-one, 
\begin{enumerate}\setcounter{enumi}{10}
\item  $\alpha_j=e_{\alpha_{n-i^*}}^{-1}(a(j))\text{ for }j<n-i^*$.
\end{enumerate} 

(10) shows  that $\alpha_{n-i^*}$ has at most $i^*$ choices and (11) shows that $\alpha_0,...,\alpha_{n-i^*-1}$ are determined by $\alpha_{n-i^*}$ and $a$. So $|P(F, s, a, m)|\leq i^*$. Together with the fact that $P(F, s,a,m)\subseteq P_m$ and $P_m$ is $i^*$-linked, $P(F, s, a, m)\in \ms{C}$.\medskip

So in any case, $P(F, s, a, m)\in \ms{C}$. This shows that $F\in \mc{H}^\pi_0$ and hence is a lower bound of $\{F_\xi: \xi\in x\}$. Then $\{F_\alpha: \alpha\in \Gamma''\}$ is $n$-linked and $\mc{H}^\pi_0$ has property K$_n$.
\end{proof}

Applying {\rm P(K$_n\ra$K$_{n+1}$)} to $\mc{H}^\pi_0$, we find $X\in [\omega_1]^{\omega_1}$ such that $[X]^{n+1}\subseteq \mc{H}^\pi_0$.  Let $\delta<\omega_1$ satisfy the conclusion of Lemma \ref{lem 52}. For $x\in [X\cap \delta]^{n-i^*}$,  $s\in 2^{<\omega}$, $a\in \omega^{n-i^*}$ and $m<\omega$, let
\[\ms{C}_{x, s, a, m}=\bigcup \{P(x\cup \{\alpha\}, s, a,m): \alpha\in X\setminus \delta\}.\] 
Then by (d) and Lemma \ref{lem 52}, 
\[\mathbb{P}=\bigcup \{\ms{C}_{x, s, a, m}: x\in [X\cap \delta]^{n-i^*}, s\in 2^{<\omega}, a\in \omega^{n-i^*}, m<\omega\}.\]
By   $[X]^{n+1}\subseteq \mc{H}^\pi_0$, each $\ms{C}_{x, s, a, m}$ is $(i^*+1)$-linked. To see this, choose $p_{\beta_0},...,p_{\beta_{i^*}}$ in $\ms{C}_{x, s, a, m}$ witnessed by $\alpha_0,...,\alpha_{i^*}$ in $X\setminus \delta$ respectively.  Let $F=x\cup \{\alpha_j: j\leq i^*\}$. Then $F\in [X]^{n+1}$ and hence $\pi(F)=0$. So $\{p_{\beta_j}: j\leq i^*\}\subseteq P(F,s,a,m)$ is centered.

This shows that $\mathbb{P}$ is $\sigma$-$(i^*+1)$-linked.

\begin{lem}\label{lem K3toK4}
Suppose $n\geq 2$ and {\rm P(K$_n\ra$K$_{n+1}$)}. Then the following statements hold.
\begin{enumerate}
\item {\rm P$_{\omega_1}$(K$_n\ra\sigma$-linked)}.
\item For $2\leq i<n$,  {\rm P$_{\omega_1}$((K$_n\wedge \sigma$-$i$-linked)$\ra\sigma$-$(i+1)$-linked)}.
\end{enumerate}
\end{lem}
\begin{proof}
 In above argument,  take $i^*=1$ for (1) and  take $i^*=i$ for (2).
\end{proof}

The following corollary follows from a finite induction.
\begin{cor}\label{cor K3toK4 K3tosK3}
{\rm P(K$_n\ra$K$_{n+1}$)} implies {\rm P$_{\omega_1}$(K$_n\ra\sigma$-$n$-linked)}.
\end{cor}

\begin{proof}[Proof of Theorem \ref{thm K3toK4}.]
Assume {\rm P(K$_n\ra$K$_{n+1}$)}. Then clearly {\rm P($\sigma$-$n$-linked$\ra$K$_{n+1}$)} holds. By Theorem \ref{thm sK3toK4}, MA$_{\omega_1}(\sigma$-$n$-linked) holds.  Together with Corollary \ref{cor K3toK4 K3tosK3}, MA$_{\omega_1}$(K$_n$) holds.
\end{proof}

\section{Distinguishing {\rm P$_{\omega_1}$(K$_n\ra \sigma$-$n$-linked)} and {\rm MA$_{\omega_1}$(K$_n$)}}

In Section 3, it is proved that {\rm P$_{\omega_1}$(K$_n\ra \sigma$-$n$-linked)} (or {\rm P$_{\omega_1}$(precaliber $\omega_1\ra \sigma$-$n$-linked)}) implies combinatorial properties on cardinal invariants.
However, in this section, we will prove that strengthening K$_n$ to $\sigma$-$n$-linked is not as strong as enlarging the coding arity. More precisely, {\rm P$_{\omega_1}$(K$_n\ra \sigma$-$n$-linked)} does not imply P(K$_n\ra$K$_{n+1}$).

We first make the following observation.
\begin{fact}\label{fact p K3}
Suppose $n\geq 2$ and $\pi: [\omega_1]^{n}\ra 2$ has no uncountable 0-homogeneous subset. If $\mc{P}$ has property {\rm K$_{n}$}, then 
\[\Vdash_\mc{P} \pi\text{ has no uncountable 0-homogeneous subset}.\]
 In particular,  forcing with posets having precaliber $\omega_1$ does not add uncountable 0-homogeneous subset of $\pi$.
\end{fact}

The following coloring and posets will be needed in adding a $\sigma$-$n$-linked partition to a poset with property K$_n$.
\begin{defn}\label{defn hpn}
For a poset $\mathbb{P}$ and $n\geq 2$, $\pi_{\mathbb{P}}^n: [\mathbb{P}]^n\ra 2$ is defined by for $F\in [\mathbb{P}]^n$,
\[\pi^n_\mathbb{P}(F)=0 \text{ iff } F \text{ is centered}.\]
Use $\mc{H}_{\mathbb{P}, n}$ to denote $\mc{H}_0^{\pi^n_\mathbb{P}}$. $(\mc{H}_{\mathbb{P},n})^\omega$ is the countable product of $\mc{H}_{\mathbb{P}, n}$ with finite support.
\end{defn}
In other words, $(\pi^n_\mathbb{P})^{-1}\{0\}$ is the collection of centered (or $n$-linked) $n$-element subsets. A simple density argument shows that for a ccc poset $\mathbb{P}$, $(\mc{H}_{\mathbb{P},n})^\omega$ (not necessarily ccc) adds a $\sigma$-$n$-linked partition to $\mathbb{P}$.
\begin{lem}\label{lem 6.1}
Suppose $n\geq 2$ and $\mathbb{P}$ is a ccc poset. Then $\Vdash_{(\mc{H}_{\mathbb{P},n})^\omega} \mathbb{P}$ is $\sigma$-$n$-linked.
\end{lem}

Although forcing with $\mathbb{P}^\omega$ also adds a $\sigma$-$n$-linked partition of $\mathbb{P}$, we cannot use this type of posets to separate {\rm P$_{\omega_1}$(K$_n\ra \sigma$-$n$-linked)} and {\rm MA$_{\omega_1}$(K$_n$)}. In fact, this strategy would induce  {\rm MA$_{\omega_1}$(K$_n$)}. On the other hand, the following lemma indicates the different forcing properties  that $\mathbb{P}^\omega$ and $(\mc{H}_{\mathbb{P},n})^\omega$ have.

\begin{lem}\label{lem 6.2}
Suppose $n\geq 2$ and $\mathbb{P}$ is a poset with property {\rm K$_n$}. Then $(\mc{H}_{\mathbb{P},n})^\omega$ has precaliber $\omega_1$.
\end{lem}
\begin{proof}
We will show that $\mc{H}_{\mathbb{P},n}$ has precaliber $\omega_1$. Then the lemma follows from the well-known fact that precaliber $\omega_1$ is closed under product with finite support (see \cite{Barnett}).

Fix $\{F_\alpha\in \mc{H}_{\mathbb{P},n}: \alpha<\omega_1\}$. Find $\Gamma\in [\omega_1]^{\omega_1}$ and $M<\omega$ such that
\[|F_\alpha|=M\text{ for every  }\alpha\in \Gamma.\]
We may assume that $M\geq n$. Enumerate each $F_\alpha$ as $\{p_{\alpha, i}: i<M\}$ and for each $I\in [M]^n$, choose a common lower bound $p_{\alpha, I}$ of $\{p_{\alpha, i}: i\in I\}$. Then a ${M \choose n}$ step induction induces $\Gamma'\in [\Gamma]^{\omega_1}$ such that
\[\text{for each $I\in [M]^n$,  $\{p_{\alpha, I}: \alpha\in \Gamma'\}$ is $n$-linked.}\]
Then $\bigcup_{\alpha\in \Gamma'} F_\alpha$ is an $n$-linked subset of $\mathbb{P}$. Hence $\{F_\alpha: \alpha\in \Gamma'\}$ is a centered subset of $\mc{H}_{\mathbb{P},n}$.
\end{proof}

Now we are ready to distinguish {\rm P$_{\omega_1}$(K$_n\ra \sigma$-$n$-linked)} and {\rm MA$_{\omega_1}$(K$_n$)}.
\begin{thm}\label{thm K3sK3}
It is relatively consistent with {\rm ZFC} that
\begin{enumerate}
\item for $n\geq 2$, {\rm P$_{\omega_1}$(K$_n\ra \sigma$-$n$-linked)} holds and
\item for  $n\geq 2$, {\rm MA$_{\omega_1}$(K$_n$)} fails.
\end{enumerate}
\end{thm}
\begin{proof}
Start from a model of {\rm GCH}. By Proposition \ref{prop sK3toK4 b}, for each $n\geq 2$, fix a coloring
\begin{enumerate}
\item $\pi_n: [\omega_1]^{n+1}\ra 2$ such that $\mc{H}^{\pi_n}_0$ is $\sigma$-$n$-linked and $\pi$ has no uncountable 0-homogeneous (or $(n+1)$-linked) subset.
\end{enumerate}
Fix a partition 
\[\{\alpha<\omega_3: cf(\alpha)=\omega_2\}=\bigcup \{ S_{\alpha,\beta,n}: \alpha<\omega_3, \beta<\omega_2, 2\leq n<\omega\}\]
such that each $S_{\alpha, \beta, n}$ is stationary.
Then iteratively force with finite support $\langle \mc{P}_\alpha, \dot{\mc{Q}}_\beta: \beta<\omega_3, \alpha\leq \omega_3\rangle$ such that at step $\gamma$,
\begin{itemize}
\item  if $\gamma\in S_{\alpha, \beta,n}$, $\dot{\mathbb{P}}$ be the $\beta$th ccc poset   in $(H(\omega_2))^{V^{\mc{P}_\alpha}}$ and $\dot{\mathbb{P}}$ has property K$_n$ in $V^{\mc{P}\gamma}$,  $\dot{\mc{Q}}_\gamma$ is $(\mc{H}_{\dot{\mathbb{P}}, n})^\omega$,
\item otherwise, $\dot{\mc{Q}}_\gamma$ is the trivial forcing.
\end{itemize}

By Lemma \ref{lem 6.2}, $\mc{P}_{\omega_3}$ has precaliber $\omega_1$. So by Fact \ref{fact p K3}, in $V^{\mc{P}_{\omega_3}}$, no $\pi_n$ has an uncountable 0-homogeneous subset. In particular, {\rm MA$_{\omega_1}$(K$_n$)} fails for $n>2$.

On the other hand, every poset $\mathbb{P}$ of size $\omega_1$ that has property K$_n$ in $V^{\mc{P}_{\omega_3}}$, has  property K$_n$ in $V^{\mc{P}_\gamma}$ for club relative to $\{\alpha<\omega_3: cf(\alpha)=\omega_2\}$ many $\gamma$.  So by Lemma \ref{lem 6.1}, such $\mathbb{P}$ is $\sigma$-$n$-linked. So {\rm P$_{\omega_1}$(K$_n\ra \sigma$-$n$-linked)} holds for $n\geq 2$.
\end{proof}

The use of $\mc{H}_{\mathbb{P},n}$ also indicates that strengthening K$_n$ in {\rm P$_{\omega_1}$(K$_n\ra \sigma$-$n$-linked)} does not change the property. This is justified by the following lemma.
\begin{lem}\label{lem 6.3}
Suppose $n\geq 2$ and $\mathbb{P}$ is a ccc poset. Then $\mathbb{P}$ is $\sigma$-$n$-linked iff $\mc{H}_{\mathbb{P},n}$ is $\sigma$-$n$-linked.
\end{lem} 
\begin{proof}
First assume that $\mathbb{P}$ is $\sigma$-$n$-linked witnessed by 
\[\mathbb{P}=\bigcup_{m<\omega} P_m.\]
Without loss of generality, assume $\mathbb{P}=(\kappa, <_\mathbb{P})$ for some ordinal $\kappa$.

For each $m<\omega$, let $H_m$ be the collcetion of all $F\in \mc{H}_{\mathbb{P},n}\cap [\mathbb{P}]^{\leq n}$ that   has a common lower bound in $P_m$. Then
\[\mc{H}_{\mathbb{P},n}\cap [\mathbb{P}]^{\leq n}=\bigcup_{m<\omega} H_m\]
and each $H_m$ is $n$-linked.

For $M>n$ and $a: [M]^n\ra \omega$, let $H_{M, a}$ be the collection of all $F\in \mc{H}_{\mathbb{P},n}\cap [\mathbb{P}]^{M}$ such that $F[I]$ has a common lower bound in $P_{a(I)}$ for all $I\in [M]^n$. Then
\[\mc{H}_{\mathbb{P},n}\cap [\mathbb{P}]^{> n}=\bigcup\{ H_{M, a}: M>n, a: [M]^n\ra \omega\}\]
and each $H_{M, a}$ is $n$-linked.

So $\mc{H}_{\mathbb{P},n}$ is $\sigma$-$n$-linked.

Now assume that $\mc{H}_{\mathbb{P},n}$ is  $\sigma$-$n$-linked witnessed by 
\[\mc{H}_{\mathbb{P},n}=\bigcup_{m<\omega} Q_m.\]
For each $m<\omega$, let $P_m=\bigcup Q_m$. Then each $P_m$ is $n$-linked and
\[\bigcup  \mc{H}_{\mathbb{P},n}=\bigcup_{m<\omega} P_m.\]
Since $\mathbb{P}$ is ccc, $\mathbb{P}\setminus \bigcup  \mc{H}_{\mathbb{P},n}$ is at most countable. So $\mathbb{P}$ is $\sigma$-$n$-linked.
\end{proof}

Now the following conclusion follows from Lemma \ref{lem 6.2} and Lemma \ref{lem 6.3}.
\begin{prop}\label{prop K3sK3}
For $n\geq 2$, {\rm P$_{\omega_1}$(K$_n\ra \sigma$-$n$-linked)} is equivalent to {\rm P$_{\omega_1}$(precaliber $\omega_1\ra \sigma$-$n$-linked)}. In particular, {\rm P$_{\omega_1}$(K$_{n+1}\ra \sigma$-$(n+1)$-linked)} implies {\rm P$_{\omega_1}$(K$_n\ra \sigma$-$n$-linked)}.
\end{prop}

A straightforward corollary of above proposition is that
\begin{itemize}
\item MA$_{\omega_1}$(precaliber $\omega_1$) implies {\rm P$_{\omega_1}$(K$_n\ra \sigma$-$n$-linked)} for all $n\geq 2$.
\end{itemize}

Now Theorem \ref{thm K3sK3} is a corollary of above proposition and the well-known fact that MA$_{\omega_1}$(precaliber $\omega_1$) does not imply any MA$_{\omega_1}$(K$_n$).

A natural question is whether the implications in Proposition \ref{prop K3sK3} are strict. In other words,  does {\rm P$_{\omega_1}$(K$_n\ra \sigma$-$n$-linked)} imply {\rm P$_{\omega_1}$(K$_{n+1}\ra \sigma$-$(n+1)$-linked)}? We will show that the implication is strict.

\begin{thm}\label{thm K3tosK3 n+1}
Suppose $n\geq 2$. It is consistent that {\rm P$_{\omega_1}$(K$_n\ra \sigma$-$n$-linked)} holds and {\rm P$_{\omega_1}$(K$_{n+1}\ra \sigma$-$(n+1)$-linked)} fails.
\end{thm}

Theorem \ref{thm K3tosK3 n+1} is proved via an iteration with controlled damage generalizing the form in  \cite{Peng25}.
So before proving Theorem \ref{thm K3tosK3 n+1}, we first introduce the general damage control structure in the next section.

\section{Damage control structure}

In \cite{Peng25}, we introduced an iterated forcing with  minimal damage to a strong coloring.  In this section, we  will express this idea of controlling  damage to strong colorings in an explicit way. 

We  will take out the structure that controls damage from the construction in \cite{Peng25} and introduce its general form.  More precisely, the structure that minimizes damage is the collection $(\mc{C}, \mc{T}, \mathbf{E})$ together with the properties they satisfy. The $\pi, \mc{R}$ in $\varphi_0(\pi, \mc{C}, \mc{R}, \mc{T}, \mathbf{E})$ of \cite{Peng25} are specific objects and properties to fulfill the corresponding purpose.  Moreover, although all the applications we use are associated with strong colorings, we  would like to point out that the associated objects may vary for different purposes.
 
  The general damage control structure is almost the same as the specific form  in \cite{Peng25} except for the following major difference.
\begin{itemize}
\item  In the general form, elements in the collection $\mcc$ have roots and are of form $(a, \msa)$.
\item In the specific form,  elements in the collection $\mcc$ have no root and are of form $\msa$.
\end{itemize}
To see the reason behind this difference, we first recall some notations from \cite{Peng25}.
    \begin{defn}
  \begin{enumerate}
    \item Suppose   $a$, $b$ are both finite sets of ordinals but neither is an ordinal. 
    Say $a<b$ if $\max a<\min b$.
    \item A set $\ms{A}\subset [Ord]^{<\omega}$ is \emph{non-overlapping} if for every $a\neq b$ in $\ms{A}$, either $a< b$ or $ b< a$.
    \item For an uncountable non-overlapping family $\msa\subset [\omega_1]^n$, use $N_\msa$ to denote this $n$, i.e., $\msa\subset [\omega_1]^{N_\msa}$.
    \item For an uncountable non-overlapping family $\msa\subset [\omega_1]^{N_\msa}$ and $i<N_\msa$, denote $\msa_i=\{a(i): a\in \msa\}$.
  \end{enumerate}
\end{defn}

The general structure  to be used to control damage is the collection $(\mc{C}, \mc{T}, \mathbf{E})$ together with the property $\varphi_1(\mc{C}, \mc{T}, \mathbf{E})$ in Definition \ref{cte} below.

 As mentioned above, elements of $\mc{C}$ in the general form will be different from that in the specific form.  This difference comes from the arity of the coloring: the colorings in the specific form are binary and the colorings in the general form are arbitrarily $n$-ary or finitary.  This phenomenon is common in analyzing corresponding colorings. For example, let $\pi: [\omega_1]^n\ra 2$ be a coloring for some $n\geq 2$. In the process of proving ccc of $\mc{H}^\pi_0$, we fix an uncountable family $\{p_\alpha\in \mc{H}^\pi_0: \alpha<\omega_1\}$. Then we find an uncountable $\Gamma\subseteq \omega_1$ such that $\{p_\alpha: \alpha\in \Gamma\}$ forms a $\Delta$-system with root $\overline{p}$. When $n=2$, $p_\alpha$ is compatible with $p_\beta$ iff $p_\alpha\setminus \overline{p}$ is compatible with $p_\beta\setminus \overline{p}$. So the root $\overline{p}$ is not important in the case $n=2$ and   is usually omitted. But for $n\geq 3$, this   is not true  and we cannot omit the root.

So we may view $(a, \msa)$ in the structure $\mc{C}$ as a $\Delta$-system $\{a\cup b: b\in \msa\}$ with 
 root $a$. 
To simplify our expression, we will require $\mcc$ to contain $(\emptyset, [\omega_1]^1)$ where $[\omega_1]^1=\{\{\alpha\}: \alpha<\omega_1\}$.

 \begin{defn}\label{c}
    Let $\varphi_1(\mc{C})$ be the assertion that the following statements hold.
      \begin{enumerate}[{\rm (C1)}]
  \item $(\emptyset, [\omega_1]^1)\in \mcc$ and every element in $\mc{C}$ has form  $(a, \msa)$ and satisfies the following properties.
  \begin{itemize}
  \item $a\in [\omega_1]^{<\omega}$;
  \item $\msa\subset [\omega_1]^{n}$ is uncountable non-overlapping for some $1\leq n<\omega$;
  \item $\max(a)<\min(\bigcup \msa)$ when $a\neq \emptyset$.
  \end{itemize}
  For above $\msa$, $N_\msa$ exists and equals $n$.
  \item For $(a,\msa), (b,\msb)$ in $\mc{C}$, if $\msb_i\subseteq \msa_j$ for some $i<N_\msb$ and $ j<N_\msa$, then
  \begin{itemize}
  \item $a\subseteq b$;
  \item  there are $k$ and $\{I_l\in [N_\msb]^{N_\msa}: l<k\}$ such that   
  $$[c]^{N_\msa}\cap \msa=\{c[I_l]: l<k\}\text{ and }c(j')\notin \bigcup \msa$$
   whenever  $c\in \msb$ and $j'\in N_\msb\setminus \bigcup_{l<k} I_l$.  
   \end{itemize}
  \end{enumerate}
        \end{defn}
        
        Note that $I_l$'s above are pairwise disjoint.
      Applying (C2) to $\msb_i\subseteq \msa_j$ and $\msa_j\subseteq \msb_i$ respectively,  we  observe the following property of $\mc{C}$.
                
                \begin{lem}\label{lem 7.a}
Assume $\varphi_1(\mc{C})$. For $(a, \msa), (b, \msb)$ in $\mc{C}$, if $\msb_i= \msa_j$ for some $i<N_\msb$ and $ j<N_\msa$, then $(a, \msa)=(b, \msb)$.
\end{lem}

    $\mc{T}$ and $\mathbf{E}$ are induced from $\mc{C}$. 
           \begin{defn}\label{ct}
    Let $\varphi_1(\mc{C}, \mc{T})$ be the assertion that $\varphi_1(\mc{C})$ together with the following statements hold.
      \begin{enumerate}[{\rm (T1)}]
        \item $(\mc{T}, \supset)$ is a downward tree of height $\leq \omega$.
  \item $\mc{T}=\{X: X=\msa_i \text{ for some } (a, \msa)\in \mc{C} \text{ and } i<N_\msa\}$. Note that the 0th level is
  $\mc{T}_0=\{\omega_1\}$.
  \item For every incomparable $A, B$ in $\mc{T}$, $|A\cap B|\leq \omega$.
  \item For every $A\supset B$ in $\mc{T}$ and every $\alpha<\omega_1$, $A(\alpha)<B(\alpha)$.
    \end{enumerate}
        \end{defn}

 \begin{defn}\label{cte}
Let $\varphi_1(\mc{C}, \mc{T}, \mathbf{E})$ be the assertion that $\varphi_1(\mc{C}, \mc{T})$ together with the following statement hold.
  \begin{enumerate}[{\rm (E1)}]
  \item $\mathbf{E}$ is an equivalence relation on $[\omega_1]^{\omega_1}$ defined by $A\mathbf{E}B$  iff  for some   $(a, \msa)\in \mc{C}$, $i, j<N_\msa$ and $\Gamma\in [\omega_1]^{\omega_1}$, $A=\msa_i[\Gamma]$ and $B=\msa_j[\Gamma]$.
  \end{enumerate}
  \end{defn}
  
  Note that $A\mathbf{E} A$ is witnessed by $(\emptyset, [\omega_1]^1)$, 0, 0 and $A$.
  
\textbf{Remark.} The relation $\mathbf{E}$ defined above is an equivalence relation if $\varphi_1(\mc{C}, \mc{T})$ holds, by    Lemma \ref{lem 7.h} below (see also \cite[Lemma 5]{Peng25}).\medskip

We will use the following notation to simplify our expression.
\begin{defn}\label{defn agamma}
Suppose $\msa\subseteq [\omega_1]^{<\omega}$ is an uncountable non-overlapping family. For $\alpha<\omega_1$, $\msa(\alpha)$ is the $\alpha$th element, under $<$, of $\msa$. For $\Gamma\subseteq \omega_1$, $\msa[\Gamma]=\{\msa(\alpha): \alpha\in \Gamma\}$.
\end{defn}

In general, $\mc{C}\setminus \{(\emptyset, [\omega_1]^1)\}$ is used to collect generically added   families in $ [\omega_1]^{<\omega}\times ([\omega_1]^{<\omega})^{\omega_1}$.\medskip

In order to have a clearer description of the structure, we define a relation for pairs of elements with property (C2). We then use a function to record the $I_l$'s.
\begin{defn}\label{defn fab}
Suppose $\msa\subseteq [\omega_1]^{N_\msa}, \msb\subseteq [\omega_1]^{N_\msb}$ are uncountable non-overlapping families.
Say $\msa\preceq \msb$ if there are $k$ and $ \{I_l\in [N_\msb]^{N_\msa}: l<k\}$   such that  
\begin{align*}
\{c[I_l]: l<k\}\subseteq \msa\text{ and }  c(j)\notin \bigcup \msa 
 \text{ whenever $c\in \msb$ and }  j\in N_\msb\setminus \bigcup_{l<k} I_l.
\end{align*} 
In this case, denote $\mc{F}(\msa, \msb)=\{I_l: l<k\}$. 

For $a, b$   in $[\omega_1]^{<\omega}$,   say $(a, \msa)\preceq (b, \msb) $ if $a\subseteq b$ and $\msa\preceq \msb$.

Say $(a, \msa)\prec (b, \msb)$ if $(a, \msa)\preceq (b, \msb)$ and $(a, \msa)\neq (b, \msb)$.
\end{defn}
In our notation, different types of sets, $(a, \msa)$ and $\msb$, are not $\preceq$-comparable. Note $\preceq$ is not a partial order, but we shall see in Lemma \ref{lem 7.d} below that $(\mcc, \preceq)$ is a partial order if $\varphi_1(\mc{C}, \mc{T})$ holds.

(C2) induces another equivalent formulation of $\preceq$ on $\mcc$.

\begin{lem}\label{lem 7.b}
Assume $\varphi_1(\mc{C}, \mc{T})$. For $(a, \msa),  (b, \msb)$ in $\mc{C}$, 
\begin{enumerate}[(i)]
\item $(a, \msa)\preceq (b, \msb)\text{ iff $\msb_{i}\subseteq \msa_{j}$ for some $i<N_\msb$ and }j<N_\msa$;
\item $(a, \msa)\prec (b, \msb) \text{ iff } \msb_i\subset \msa_j \text{ for some $i<N_\msb$ and }j<N_\msa$.
\end{enumerate}
\end{lem}

We use $\mc{T}$ to collect  $\msa_i$'s for $(a, \msa)\in\mc{C}$ and $i<N_\msa$. (C2) and (T3) add the following restrictions (a)-(c) on the collection $\mc{C}$.  (T1) and (T3) are used to guarantee that no alternative other than (a)-(c) below can occur.

\textbf{Remark.}
Assume $\varphi_1(\mc{C}, \mc{T})$. For   $(a, \msa)\neq  (b, \msb)$ in $\mc{C}$, by (C2) and (T3), exactly one of the following three alternatives occurs.
\begin{enumerate}[(a)]
\item $\msa$ and $\msb$ do not affect each other: $|(\bigcup \msa)\cap (\bigcup \msb)|\leq \omega$.
\item $(a, \msa)\preceq (b, \msb)$ and $(b, \msb)$ inherits structure from $(a, \msa)$: 
\begin{itemize}
\item $a\subseteq b$;
\item  for $I\in \mc{F}(\msa, \msb)$, $\{c[I]: c\in\msb\}\subseteq \msa$   is a part of $\msb$ inheriting structure from $\msa$;
\item $\{c[N_\msb\setminus \bigcup \mc{F}(\msa, \msb)]: c\in \msb\} $ is the part not relevant to $\msa$, moreover, 
\[(\bigcup \{c[N_\msb\setminus \bigcup \mc{F}(\msa, \msb)]: c\in \msb\})\cap(\bigcup \msa)=\emptyset.\]
\end{itemize}
\item $(b, \msb)\preceq (a, \msa)$ and $(a, \msa)$ inherits structure from $(b, \msb)$: 
\begin{itemize}
\item $b\subseteq a$;
\item  for $I\in \mc{F}(\msb, \msa)$, $\{c[I]: c\in\msa\}\subseteq \msb$   is a part of $\msa$ inheriting structure from $\msb$;
\item $\{c[N_\msa\setminus \bigcup \mc{F}(\msb, \msa)]: c\in \msa\} $ is the part not relevant to $\msb$, moreover, 
\[(\bigcup \{c[N_\msa\setminus \bigcup \mc{F}(\msb, \msa)]: c\in \msa\})\cap(\bigcup \msb)=\emptyset.\]
\end{itemize}
\end{enumerate}

In above condition (b), $\msb$ is locally big: every $b'\in \msb$ contains several copies, $\{b'[I]: I\in \mc{F}(\msa, \msb)\}$, of elements in $\msa$,  and is globally small: $\msb_{I(i)}\subset \msa_i$ whenever $i<N_\msa$ and $I\in \mc{F}(\msa, \msb)$.

We now turn to investigating properties of  $\preceq$.

\begin{lem}\label{lem 7.c}
Assume $\varphi_1(\mc{C}, \mc{T})$. For $(a, \msa)\preceq  (b, \msb)$ in $\mc{C}$,   the following statements hold.
\begin{enumerate}[(i)]
\item For $i<N_\msa$ and $I\in \mc{F}(\msa, \msb)$, $\msa_i\supseteq \msb_{I(i)}$. The equality holds only when    $(a, \msa)= (b, \msb)$.
\item $\msa_i\cap \msb_j=\emptyset$ for all $i<N_\msa$ and $j\in N_\msb\setminus \{I(i): I\in \mc{F}(\msa, \msb)\}$.
\item For $I\in \mc{F}(\msa, \msb)$, there exists $\Gamma\in [\omega_1]^{\omega_1}$ such that $\msa[\Gamma]=\{b[I]: b\in \msb\}$. In particular, 
 for every $i<N_\msa$, $\msa_i[\Gamma]=\msb_{I(i)}$.
  \item For $(a', \msa')\preceq (b, \msb)$ in $\mcc$, $I\in \mc{F}(\msa, \msb)$ and $I'\in \mc{F}(\msa', \msb)$, one of the following 3 alternatives occurs.
  \begin{itemize}
  \item $I\cap I'=\emptyset$.
  \item $(a, \msa)\preceq (a', \msa')$ and $I\subseteq I'$.
  \item $(a', \msa')\preceq (a, \msa)$ and $I'\subseteq I$.
  \end{itemize} 
\end{enumerate}
\end{lem} 
\begin{proof}
(i) and (ii) are trivial. To see (iii), fix $I\in \mc{F}(\msa, \msb)$. Note that $\{b'[I]: b'\in \msb\}\subseteq \msa$ and are both well ordered by $<$. So $\Gamma$ is the collection of $\alpha<\omega_1$ such that the $\alpha$th element (under $<$) of $\msa$ is in $\{b'[I]: b'\in \msb\}$.

(iv). Suppose $I\cap I'\neq\emptyset$. Assume $j=I(i)=I'(i')$. Then $\msa_{i}\cap \msa'_{i'}\supseteq \msb_j$. By (T3), $\msa_{i}\subseteq \msa'_{i'}$ or $\msa'_{i'}\subseteq \msa_{i}$. 

By symmetry, assume $\msa_{i}\subseteq \msa'_{i'}$.
 If $\msa_{i}= \msa'_{i'}$, then by Lemma \ref{lem 7.a}, $(a, \msa)=(a', \msa')$. By (C2) and the fact $I\cap I'\neq\emptyset$, $I=I'$.

Now assume $\msa_{i}\subset \msa'_{i'}$. Then $(a', \msa')\prec (a, \msa)$. Since $j=I(i)=I'(i')$, $\msb_j\subseteq\msa_{i}\subseteq \msa'_{i'}$. Choose $c\in \msb$ and let $d=c[I]$. Then 
\[c[I]\in \msa, c[I']\in \msa'\text{ and  }c(j)=d(i)\in \msa_{i}\subseteq \msa'_{i'}.\]
 By (C2), 
\[\text{ for some $J\in \mc{F}(\msa', \msa)$, } d[J]\in \msa'\text{ and }i\in J.\]
  By $\{c[I'],d[J]\}\subseteq \msa'$ and the fact $c(j)=d(i)\in c[I']\cap d[J]$, $c[I']=d[J]$. 
  
  Recall that $d=c[I]$. So $I'=I[J]$ and hence $I'\subseteq I$.
\end{proof}

\begin{lem}\label{lem 7.d}
Assume $\varphi_1(\mc{C}, \mc{T})$.   $\preceq$ is a  partial order on $\mc{C}$.
\end{lem} 
\begin{proof}
To see antisymmetry, fix $(a, \msa)\preceq (b, \msb)\preceq (a, \ms{A})$. Then $(a, \msa)= (b, \msb)$, since otherwise by Lemma \ref{lem 7.c} (i), $\msa_0\supsetneqq \msb_i\supsetneqq \msa_j$ for some $i<N_\msb, j<N_\msa$. 

To see transitivity, fix $(a, \msa)\preceq (b, \msb)\preceq (c, \ms{C})$. Then by Lemma \ref{lem 7.c} (i),  $\msa_0\supseteq \msb_i\supseteq \ms{C}_j$ for some $i<N_\msb$ and $j<N_\ms{C}$. Then by Lemma \ref{lem 7.b}, $(a, \msa)\preceq (c, \ms{C})$.
\end{proof}

We will   use the following sub-collection to analyze the structure $(\mcc, \prec)$.

\begin{defn}
For $X\in [\omega_1]^{\omega_1}$, $\mcc(X)=\{(a, \msa)\in \mcc: X\subseteq \msa_i\text{ for some } i<N_\msa\}$.
\end{defn}


We have the following properties on $\prec$.
\begin{lem}\label{lem 7.e}
Assume $\varphi_1(\mc{C}, \mc{T})$. 
For $X\in [\omega_1]^{\omega_1}$,  $(\mcc(X), \prec)$
 is a well-ordering. In particular,
$(\mc{C}, \prec)$ is well founded.
\end{lem}
\begin{proof}
For every $(a,\msa)\in \mcc(X)$, fix $f(a, \msa)<N_\msa$ such that $X\subseteq \msa_{f(a, \msa)}$. Then $\mc{T}'=\{\msa_{f(a, \msa)}: (a, \msa)\in \mcc(X)\}$ is a subset of $\mct$. 

By (T3), elements in $\mc{T}'$ are pairwise $\supset$-comparable.
By (T1), $(\mc{T}', \supset)$ is a well-ordering. Then by Lemma \ref{lem 7.b}, $(\mcc(X), \prec)$ is a well-ordering.

For the in particular part, note that for $(a, \msa)\in \mcc$, 
\[\{(b, \msb)\in \mcc: (b, \msb)\preceq (a, \msa)\}=\bigcup_{i<N_\msa} \mcc(\msa_i)\]
 is a finite union of well orders and hence has no infinite decreasing sequence.
 \end{proof}

 $(\mc{C}, \prec)$ is in general not a tree order, since $(\{(b, \msb)\in \mcc: (b, \msb)\preceq (a, \msa)\}, \prec)$ is in general not a linear order. 

To have finite $|a|$ and $N_\msa$ for $(a, \msa)\in \mc{C}$, we want the $\preceq$-chain to have length $\leq \omega$. In fact, (T4) guarantees a stronger property (see also \cite[Lemma 7]{Peng25}).

\begin{lem}\label{lem 7.f}
Assume $\varphi_1(\mc{C}, \mc{T})$. For a $\prec$-chain $\{ (a_n, \msa^n)\in \mc{C}: n<\omega\}$ in $\mcc$, $\bigcap_{n<\omega} (\bigcup\msa^n)=\emptyset$.
\end{lem} 
\begin{proof}
Suppose otherwise, $\xi\in \bigcap_{n<\omega} (\bigcup\msa^n)$. Then for every $n<\omega$, there exists unique $i_n<N_{\msa^n}$ such that $\xi\in \msa^n_{i_n}$. 

By Lemma \ref{lem 7.e}, the chain $(\{(a_n, \msa^n): n<\omega\}, \prec)$ is well-founded and hence a well-ordering. Arguing with a subsequence, we may assume $(a_0, \msa^0)\prec (a_1, \msa^1)\prec\cdots$.
Then by Lemma \ref{lem 7.c} (i)-(ii), 
\[\msa^0_{i_0}\supset \msa^1_{i_1}\supset \cdots\supset \msa^n_{i_n}\supset\cdots.\] 
For $n<\omega$, let $\alpha_n$ be such that $\xi=\msa^n_{i_n}(\alpha_n)$. Then by (T4) and the fact that $\msa^n_{i_n}$'s are in $\mct$, 
\[\msa^{n+1}_{i_{n+1}}(\alpha_{n+1})=\xi=\msa^n_{i_n}(\alpha_n)<\msa^{n+1}_{i_{n+1}}(\alpha_n).\]
 So $\alpha_0>\alpha_1>\cdots$. A contradiction.
\end{proof}

The following property follows from Lemma \ref{lem 7.e}, Lemma \ref{lem 7.f} and the fact that $\{(b, \msb)\in \mcc: (b, \msb)\preceq (a, \msa)\}=\bigcup_{i<N_\msa} \mcc(\msa_i)$ for  $(a, \msa)\in \mcc$.

\begin{lem}\label{lem 7.g}
Assume $\varphi_1(\mc{C}, \mc{T})$.  For every $X\in [\omega_1]^{\omega_1}$, $\mcc(X)$ is finite. In particular, for every $(a,\msa)\in \mcc$, $\{(b, \msb)\in \mcc: (b, \msb)\preceq (a, \msa)\}$ is finite.
\end{lem}

Now we see that $\mathbf{E}$ is an equivalence relation.

\begin{lem}\label{lem 7.h}
Assume $\varphi_1(\mc{C}, \mc{T})$. 
\begin{enumerate}[(i)]
\item The $\mathbf{E}$ defined by {\rm (E1)} is an equivalence relation. 
\item For $X\in [\omega_1]^{\omega_1}$, $[X]_\mathbf{E}=\{\msa_i[\Gamma]: i<N_\msa\}$ where for some $(a, \msa)\in \mcc$ and $j$, $(a, \msa)$ is the $\preceq$-maximal element of $\mcc(X)$ and $X=\msa_j[\Gamma]$. In particular, $[X]_\mathbf{E}$ is finite.
\item If $A\mathbf{E}B$, then ${\max}_\preceq(\mcc(A))={\max}_\preceq(\mcc(B))$ where ${\max}_\preceq(\mcc')$ is the collection of $\preceq$-maximal elements of $\mcc'$ for $\mcc'\subseteq \mcc$.
\end{enumerate}
\end{lem}
\begin{proof}
(i). Only transitivity needs a verification. Assume $A\mathbf{E}B\mathbf{E}C$. Choose $(a, \msa), (b, \msb)$ in $\mcc$ and $\Gamma, \Sigma$ such that $A=\msa_i[\Gamma]$, $B=\msa_j[\Gamma]=\msb_{i'}[\Sigma]$ and $C=\msb_{j'}[\Sigma]$. Then by Lemma \ref{lem 7.e}, $(a,\msa), (b, \msb)$ are in $\mcc(B)$ and hence is $\prec$-comparable. 

We will find $k$ such that either $A=\msb_k[\Sigma]$ if $(a, \msa)\preceq (b, \msb)$ or $C=\msa_k[\Gamma]$ if $(b, \msb)\preceq (a, \msa)$.
By symmetry, assume $(a, \msa)\preceq (b, \msb)$. Then by Lemma \ref{lem 7.c} (i)-(ii) and the fact $\msa_j[\Gamma]=\msb_{i'}[\Sigma]$,  $i'=I(j)$ for some $I\in \mc{F}(\msa, \msb)$. 

Now, $\msa_j[\Gamma]=\msb_{I(j)}[\Sigma]$ and hence $\msa[\Gamma]=\{b[I]: b\in \msb[\Sigma]\}$. Consequently, $A=\msa_i[\Gamma]=\msb_{I(i)}[\Sigma]$. Hence, $A\mathbf{E}C$.

(ii) and (iii) follow from above argument. For example, for (ii), if  $(b, \msb)$ is the $\preceq$-maximal element of $\mcc(B)$, then $[B]_\bfe=\{\msb_k[\Sigma]: k<N_\msb\}$.
\end{proof}

For Lemma \ref{lem 7.h} (iii), we would like to point out that   in general, $\mcc(A)\neq \mcc(B)$ for $A\mathbf{E}B$.

Now we turn to extending the structure $(\mcc, \mct, \mathbf{E})$. We will only add elements to $\mcc$ and then $\mct, \mathbf{E}$ are induced from $\mcc$ by (T2) and (E1).

\begin{defn}
Assume $\varphi_1(\mc{C}, \mc{T}, \bfe)$.  Say $(a, \msa)$ is a \emph{$\mcc$-candidate}  (or a \emph{candidate for $\mcc$}) if 
\begin{enumerate}
\item  $\varphi_1(\mcc\cup\{(a, \msa)\}, \mct', \bfe')$ holds where $\mct', \bfe'$ are induced from $\mcc\cup\{(a, \msa)\}$;
\item  $(a, \msa)\notin \mcc$ and is a $\preceq$-maximal element of $\mcc\cup\{(a, \msa)\}$.
\end{enumerate}
\end{defn}
\textbf{Remark.} It is straightforward to check that if $(a, \msa)$ is a $\mcc$-candidate, then $(a, \msa')$ is  a $\mcc$-candidate for all $\msa'\in [\msa]^{\omega_1}$.


\begin{lem}\label{lem 7.i}
Assume $\varphi_1(\mc{C}, \mc{T},\bfe)$. Suppose $a\in [\omega_1]^{<\omega}$, $\msa\subseteq [\omega_1]^{N_\msa}$ is uncountable non-overlapping and $\max(a)<\min(\bigcup \msa)$. Then $(a, \msa)$ is a $\mcc$-candidate iff the following statements hold.
\begin{enumerate}[(a)]
\item $\{\msa_i: i<N_\msa\}$ is $\bfe$-invariant.
\item For every $(b, \msb)\in \mcc$, $j<N_\msb$ and $i<N_\msa$, either ($b\subseteq a$ and $\msa_i\subset \msb_j$) or $|\msa_i\cap\msb_j|\leq \omega$.
\item For every $i<N_\msa$ and every $A\in \mct$ with $\msa_i\subseteq A$, $A(\alpha)<\msa_i(\alpha)$ for all $\alpha<\omega_1$.
\item For every $i, j<N_\msa$ and every $A\in \mct$ with $ \msa_j\subseteq A$, either $\msa_i\subseteq A$ or $\msa_i\cap A=\emptyset$.
\end{enumerate}
\end{lem}
\begin{proof}
First assume $(a, \msa)$ is a $\mcc$-candidate. Then $\varphi_1(\mcc\cup\{(a, \msa)\}, \mct', \bfe')$ holds. To see (a), note that $\{\msa_i: i<N_\msa\}$ is $\bfe'$-invariant by Lemma \ref{lem 7.h} (ii). So (a) follows from the fact $\bfe\subseteq \bfe'$. To see (b), suppose $|\msa_i\cap \msb_j|=\omega_1$. By (T3), $\msa_i$ is comparable with $\msb_j$.  Since $(a, \msa)$ is a $\mcc$-candidate, $\msa_i\subset \msb_j$. Now $b\subseteq a$ follows from (C2).

(c) follows from (T4) of $\varphi_1(\mcc\cup\{(a, \msa)\}, \mct')$. To see (d), fix $(b, \msb)\in \mcc$ and $i'$ be such that $A=\msb_{i'}$. Then $(b, \msb)\prec (a, \msa)$ are in $\mcc\cup \{(a, \msa)\}$. Now (d) follows from Lemma \ref{lem 7.c} (i)-(ii) under the assumption $\varphi_1(\mcc\cup\{(a, \msa)\}, \mct')$.\medskip

Now assume (a)-(d) hold and we show $\varphi_1(\mcc\cup\{(a, \msa)\}, \mct', \bfe')$. (C1)   is trivial. To see (C2), fix $(b, \msb)\in\mcc$ with $\msa_i\subseteq \msb_j$. By (b), $b\subseteq a$.  Let $J=\{l<N_\msa: \msa_{l}\subseteq \msb_j\}$.  Fix $l\in J$ and let $\Gamma$ be such that $\msa_l=\msb_j[\Gamma]$. Let 
\[I_l=\{i'<N_\msa: \msa_{i'}=\msb_{j'}[\Gamma]\text{ for some }j'<N_\msb\}.\]
By (a),  $l\in I_l$ and $\{a'[I_l]: a'\in \msa\}\subseteq \msb$. 

Then by (d), for $j'\in N_\msa\setminus \bigcup_{l\in J} I_l$, $\msa_{j'}\cap (\bigcup \msb)=\emptyset$. This shows (C2).

 (T2) is trivial. (T1) and (T3)   follow from (b). (T4)  follows from (c).
 
 This shows that $\varphi_1(\mcc\cup\{(a, \msa)\}, \mct', \bfe')$ holds. By (b), $(a, \msa)\notin \mcc$ and is a  $\preceq$-maximal element in $\mcc\cup\{(a, \msa)\}$.
\end{proof}

 A new $\mcc$-candidate can be formed by combining finitely many $\mcc$-candidates. The following lemma provides a verification that the combination, omitting a countable part,  is indeed a $\mcc$-candidate.
\begin{lem}\label{lem 7.l}
Assume $\varphi_1(\mc{C}, \mc{T},\bfe)$, $n<\omega$ and $\langle (a_i, \msa^i): i<n\rangle$ is a sequence of $\mcc$-candidates. Suppose $b\in [\omega_1]^{<\omega}$, an uncountable non-overlapping $\msb\subseteq [\omega_1]^{N_\msb}$    and $\langle I_i\subseteq N_\msb: i<n\rangle$ satisfy the following conditions.
\begin{enumerate}[(a)]
\item $\bigcup_{i<n} a_i\subseteq b$ and $N_\msb=\bigcup_{i<n} I_i$.
\item For every $c\in \msb$ and $i<n$, $c[I_i]\in \msa^i$.
\end{enumerate}
Then for some $\alpha<\omega_1$, $(b, \msb\cap [\omega_1\setminus \alpha]^{<\omega})$ is a $\mcc$-candidate.
\end{lem}
\begin{proof}
We first show that conditions (a)-(c) in Lemma \ref{lem 7.i} hold for $(b, \msb)$. First note by (a) and (b), 
\begin{enumerate}
\item for every $k<N_\msb$, there are $i<n$ and $j<N_{\msa^i}$ such that $k\in I_i$ and $\msb_k\subseteq\msa^i_j$.
\end{enumerate}
In particular,  Lemma \ref{lem 7.i} (b)-(c) hold.

To see  Lemma \ref{lem 7.i} (a), fix $k<N_\msb$. Let $i<n$ and $ j<N_{\msa^i}$ witness (1). Since $(a_i, \msa^i)$ is a $\mcc$-candidate,  ${\max}_\preceq(\mcc(\msb_k))={\max}_\preceq(\mcc(\msa^i_j))$. Together with Lemma \ref{lem 7.h} (ii), $[\msb_k]_\bfe\subseteq \{\msb_l: l\in I_i\}$. So $\{\msb_k: k<N_\msb\}$ is $\bfe$-invariant.

Now  Lemma \ref{lem 7.i} (d) follows from Lemma \ref{lem 7.i} (b) by omitting a countable part of $\msb$. This finishes the proof of the lemma.
\end{proof}

We introduce the structure inside a $\mcc$-candidate or an element of $\mcc$ (see also \cite[Definition 14]{Peng25}). The structure will be used when we apply the damage control structure to a specific problem, e.g., a coloring with pre-described property.
 \begin{defn}
 Assume $\varphi_1(\mcc, \mct, \bfe)$ and $(a, \msa)$ is either a $\mcc$-candidate or an element of $\mcc$. 
 \begin{enumerate}[(i)]
 \item For $n<\omega$, $\mc{C}_n(a, \msa)$ is the collection of all $(b, \msb, \langle I_i: i<n\rangle)$ such that 
 \[(b, \msb)\in \mcc, ~ (b, \msb)\prec (a, \msa) \text{ and each } I_i\in \mc{F}(\msb, \msa).\]
$\mc{C}(a, \msa)=\bigcup_{n<\omega} \mc{C}_n(a, \msa)$. 
  
  \item Let $\preceq_{a, \msa}$ be the partial order defined on $\mc{C}(a, \msa)$ by 
  \begin{align*}
  (b, \msb, \langle I_i: i<m\rangle)\preceq_{a, \msa} (c, \ms{C}, \langle J_i: i<n\rangle) \text{ iff }\\
  m=n, ~ (b, \msb)\preceq (c, \ms{C}) \text{ and } \prod_{i<m} I_i\subseteq \prod_{i<n} J_i.
  \end{align*}
  \item $\prec_{a, \msa}$ is the strict part of $\preceq_{a, \msa}$, i.e., $x\prec_{a, \msa} y \text{ if } x\preceq_{a, \msa} y \wedge x\neq y$.
\item For $n<\omega$, $\mc{C}_{n, \max}(a, \msa)$ is the collection of $\preceq_{a, \msa}$-maximal elements of $\mc{C}_n(a, \msa)$ and $\mc{C}_{\max}(a, \msa)=\bigcup_{n<\omega} \mc{C}_{n, \max}(a, \msa)$.
  \end{enumerate}
 \end{defn}
 In particular, $\mc{C}_0(a, \msa)$ is the collection of $(b, \msb)\prec (a, \msa)$ in $\mcc$. Note by Lemma \ref{lem 7.g}, $\mc{C}_0(a, \msa)$ is finite. So $\mc{C}_n(a, \msa)$ is finite for every $n<\omega$.
 
 \begin{lem}\label{lem 7.po}
 Suppose $\varphi_1(\mcc, \mct, \bfe)$ and $(a, \msa)$ is either a $\mcc$-candidate or an element of $\mcc$.
Suppose $0<n<\omega$ and $  (b, \msb, \langle I_i: i<n\rangle)\neq (c, \ms{C}, \langle J_i: i<n\rangle)$ are in $\mc{C}_n(a, \msa)$. Then one of the following 3 alternatives occurs.
\begin{itemize}
\item $\prod_{i<n} I_i\cap \prod_{i<n} J_i=\emptyset$.
\item $(b, \msb)\prec (c, \ms{C})$ and $\prod_{i<n} I_i\subseteq \prod_{i<n} J_i$.
\item $(c, \ms{C})\prec (b, \msb)$ and $\prod_{i<n} J_i\subseteq \prod_{i<n} I_i$.
\end{itemize}
In particular, if $  (b, \msb, \langle I_i: i<n\rangle)$ and $(c, \ms{C}, \langle J_i: i<n\rangle)$ are both in $\mc{C}_{n, \max}(a, \msa)$, then $\prod_{i<n} I_i\cap \prod_{i<n} J_i=\emptyset$.
 \end{lem}
 \begin{proof}
 Assume $\prod_{i<n} I_i\cap \prod_{i<n} J_i\neq \emptyset$. First note that $(b, \msb)\neq (c, \ms{C})$. Otherwise, for each $i<n$, $I_i=J_i$ since $I_i\cap J_i\neq\emptyset$. This contradicts $  (b, \msb, \langle I_i: i<n\rangle)\neq (c, \ms{C}, \langle J_i: i<n\rangle)$.
 
Applying Lemma \ref{lem 7.c} (iv) under assumption $\varphi_1(\mcc\cup \{(a, \msa)\}, \mct')$, we conclude that either $(b, \msb)\prec (c, \ms{C})$ or $(c, \ms{C})\prec (b, \msb)$.
By symmetry, assume $(b, \msb)\prec (c, \ms{C})$. Now by Lemma \ref{lem 7.c} (iv) and the fact $I_i\cap J_i\neq\emptyset$, $I_i\subseteq J_i$ for every $i<n$.
 \end{proof}

In general, the collection $\mcc$ is extended in the iteration process. So the structure $(\mcc, \mct, \bfe)$ should be extendable. Moreover, we need for every $a\in [\omega_1]^{<\omega}$ and uncountable non-overlapping $\msa\subseteq [\omega_1]^{<\omega}$, a $\mcc$-candidate $(b, \msb)$ and $I\subseteq N_\msb$ with $a\subseteq b$ and  $\{c[I]: c\in \msb\}\subseteq \msa$. So in the process of iterated forcing, when we want to find an uncountable subset (e.g., 0-homogeneous over some coloring $\pi$) of $\{a\cup a': a'\in \msa\}$, instead of adding generic subsets of $\msa$, we add generic subsets of $\msb$ (with the requirement that finite unions of $a\cup c[I]$'s are 0-homogeneous).

  For this,  we will use the generalized almost disjointness number $\mathfrak{a}_{\omega_1}$ and the generalized bounding number $\mathfrak{b}_{\omega_1}$  on $\omega_1$ (modulo the co-bounded filter).  
  
  Say $\mc{A}\subseteq [\omega_1]^{\omega_1}$ is an almost disjoint family if $|X\cap Y|<\omega_1$  for all $X\neq Y$ in $\mc{A}$. Say an almost disjoint family $\mc{A}\subseteq [\omega_1]^{\omega_1}$ is maximal if for every $X\in [\omega_1]^{\omega_1}$, there exists $Y\in \mc{A}$ with $|X\cap Y|=\omega_1$.
  $\mathfrak{a}_{\omega_1}$ is the minimal size of an uncountable maximal almost disjoint family.
  
  For $f, g \in \omega_1^{\omega_1}$ functions  from  $\omega_1$ to $\omega_1$, say 
  $$f<_{[\omega_1]^{\leq \omega}} g\text{ if }|\{\alpha<\omega_1: f(\alpha)\geq g(\alpha)\}|\leq \omega.$$
   $\mathfrak{b}_{\omega_1}$ is the minimal size of an unbounded subfamily of $(\omega_1^{\omega_1}, <_{[\omega_1]^{\leq \omega}})$ (see \cite{CS}). We recall two facts that are straightforward to check.
   \begin{itemize}
   \item $\mathfrak{b}_{\omega_1}\geq \omega_2$.
   \item $\mathfrak{b}_{\omega_1}$ is not changed by ccc forcing.
   \end{itemize}
   
The following fact follows from the proof of $\mathfrak{b}\leq \mathfrak{a}$. 
\begin{fact}\label{fact b a}
$\mathfrak{b}_{\omega_1}\leq \mathfrak{a}_{\omega_1}$.
\end{fact}
\begin{proof} Fix $\omega_1\leq\kappa<\mathfrak{b}_{\omega_1}$ and an almost disjoint family $\mc{A}=\{X_\alpha: \alpha<\kappa\}\subseteq [\omega_1]^{\omega_1}$.

For $\xi \in [\omega_1, \kappa)$, define $f_\xi: \omega_1\ra \omega_1$ by $f_\xi(\alpha)= \sup(X_\xi\cap X_\alpha)$ for $\alpha<\omega_1$.
 By  $\kappa<\mathfrak{b}_{\omega_1}$, let $f:\omega_1\ra \omega_1$ be a $<_{[\omega_1]^{\leq \omega}}$-upper bound of all $f_\xi$'s. Choose an increasing sequence $\langle x_ \alpha\in   X_ \alpha\setminus f(\alpha): \alpha <\omega_1\rangle$ such that $x_ \alpha> \sup (\bigcup_{\beta< \alpha} (X_\beta\cap X_ \alpha))$.
Then $X=\{x_\alpha: \alpha <\omega_1\}$ witnesses that $\mc{A}$ is not maximal.
\end{proof}

   For $A\in [\omega_1]^{\omega_1}$,  we will find a $\mcc$-candidate $(a, \msa)$ with $\msa_i\subseteq A$ for some $i$. For this, we will need \cite[Lemma 30]{Peng25}. We sketch a proof for completeness.
   
\begin{lem}[\cite{Peng25}]\label{lem Peng 30}
Assume $\varphi_1(\mc{C}, \mc{T})$ and $|\mcc|<\mathfrak{a}_{\omega_1}$. Then for every $A\in [\omega_1]^{\omega_1}$, there exists $B\in [A]^{\omega_1}$ such that for every $X\in \mct$, either $B\subset X$ or $|B\cap X|\leq \omega$.
\end{lem}
\begin{proof}
Consider the subtree $\mct'=\{C\in \mct: |A\cap C|=\omega_1\}$. For $C\in \mct'$, let $succ_{\mct'}(C)$ be the collection of immediate successors of $C$  in $\mct'$. We claim that 
\begin{enumerate}
\item there are $C\in \mct'$ and $B\in [A\cap C]^{\omega_1}$ such that $|B\cap C'|\leq \omega$ for all $C'\in succ_{\mct'}(C)$.
\end{enumerate}
Suppose otherwise. Then for every $C\in \mct'$, $|A\cap C\setminus \bigcup succ_{\mct'}(C)|\leq \omega$ and, by $|succ_{\mct'}(C)|\leq |\mct|<\mathfrak{a}_{\omega_1}$,  $|succ_{\mct'}(C)|\leq \omega$.  So $\mct'$ is a countably branching tree of height $\leq \omega$. Then $\mct'$, and hence $\bigcup \{A\cap C\setminus \bigcup succ_{\mct'}(C): C\in \mct'\}$, is countable. Fix $\alpha\in A\setminus \bigcup \{A\cap C\setminus \bigcup succ_{\mct'}(C): C\in \mct'\}$. Then for $C\in \mct'$, $\alpha\in C$ iff $\alpha\in \bigcup succ_{\mct'}(C)$. Inductively find a chain $C_0\supset C_1\supset\cdots$ in $\mct'$ such that $\alpha\in  C_n$ for all $n$. This contradicts Lemma \ref{lem 7.f}.

Let $C, B$ witness (1).
We check that $B$ is as desired. Fix $X\in \mct$. 

If $X$ is $\supset$-incomparable with $C$, then by (T3), $|B\cap X|\leq |C\cap X|\leq \omega$.

If $C\subseteq X$, then $B\subset X$.

If $X\subset C$, then by (1) and the definition of $\mct'$, $|B\cap X|\leq \omega$.
\end{proof}
   
   \begin{lem}\label{lem 7.j}
  Assume $\varphi_1(\mc{C}, \mc{T},\bfe)$ and $|\mcc|<\mathfrak{a}_{\omega_1}$.  For every $A\in [\omega_1]^{\omega_1}$, there are  $(a, \msa)\in \mcc$ and $\Gamma\in [\omega_1]^{\omega_1}$ such that $\msa_i[\Gamma]\subseteq A$ for some $i<N_\msa$ and $(a, \msa[\Gamma])$ is a $\mcc$-candidate.
\end{lem}
\begin{proof}
By Lemma \ref{lem Peng 30}, find $B\in [A]^{\omega_1}$ such that 
\begin{enumerate}
\item for every $X\in \mct$, either $B\subset X$ or $|B\cap X|\leq \omega$. 
\end{enumerate}
Let $(a, \msa)$ be the $\preceq$-maximal element of $\mcc(B)$.
Fix $i<N_\msb$ and $\Sigma\in [\omega_1]^{\omega_1}$ such that 
\[B= \msa_i[\Sigma].\]
 By (1) and maximality of $(a, \msa)$,
 \begin{enumerate}\setcounter{enumi}{1}
\item for every $\Gamma\in [\omega_1]^{\omega_1}\setminus\{\omega_1\}$, if $\msa_i[\Gamma]\in \mct$, then $|\Sigma\cap \Gamma|=|\msa_i[\Sigma]\cap \msa_i[\Gamma]|\leq \omega$.
\end{enumerate}
Recall by (C2) or Lemma \ref{lem  7.c} (iii), $\msa_i[\Gamma]\in \mct$ iff $\msa_j[\Gamma]\in \mct$. So by (2),
 \begin{enumerate}\setcounter{enumi}{2}
\item for every $X\in \mct$ and $j<N_\msa$, if $X\subset \msa_j$, then $|X\cap \msa_j[\Sigma]|\leq \omega$.
\end{enumerate}

We show that $(a, \msa[\Sigma'])$ satisfies Lemma \ref{lem 7.i} (b) for all $\Sigma'\in [\Sigma]^{\omega_1}$.  Fix $|\msb_k\cap \msa_{j}[\Sigma']|=\omega_1$ for some $(b, \msb)\in \mcc$ and $k,j$. By (T2), $\msa_j$ is $\subseteq$-comparable with $\msb_k$.
Then by (3), $\msa_j\subseteq \msb_k$.  So $\msa_j[\Sigma']\subset \msb_k$ and by (C2), $b\subseteq a$.

Now find $\Gamma\in [\Sigma]^{\omega_1}$ such that $(a, \msa[\Gamma])$ is a $\mcc$-candidate. To see the existence of such $\Gamma $, note that Lemma \ref{lem 7.i} (a) is automatically satisfied.
Then note by (3),
\[{\max}_\preceq(\mcc(\msa_j[\Sigma']))=\{(a, \msa)\}\text{ for all } \Sigma'\in [\Sigma]^{\omega_1}.\]
So Lemma \ref{lem 7.i} (c)-(d) are satisfied by choosing an appropriate subset $\Gamma$. 

Recall that $(a, \msa)\in \mcc$.
Now $(a, \msa)$ and $\Gamma $ are as desired.
\end{proof}

   Now we deal with finite tuples.
   
     \begin{lem}\label{lem 7.k}
  Assume $\varphi_1(\mc{C}, \mc{T},\bfe)$ and $|\mcc|<\mathfrak{a}_{\omega_1}$. For every $a\in [\omega_1]^{<\omega}$ and uncountable non-overlapping $\msa\subseteq [\omega_1]^{N_\msa}$, there are $\mcc$-candidate $(b, \msb)$ and $I\subseteq N_\msb$ such that $a\subseteq b$ and $\{b[I]: b\in \msb\}\subseteq \msa$.
\end{lem}
\begin{proof}
Inductively applying Lemma \ref{lem 7.j} to find  $\langle (a_i, \msa^i)\in\mcc: i<N_\msa\rangle$ and $\Gamma_0\supseteq \cdots\supseteq \Gamma_{N_\msa-1}$ such that
for each $i<N_\msa$, 
\begin{enumerate}
\item for some $j<N_{\msa^i}$ and $\Sigma_i\in [\omega_1]^{\omega_1}$, $\msa^i_j[\Sigma_i]=\msa_i[\Gamma_i]$ and $(a_i, \msa^i[\Sigma_i])$ is a $\mcc$-candidate.
\end{enumerate}
In order to have $\Gamma_i\supseteq \Gamma_{i+1}$, at step $i+1$, we apply Lemma \ref{lem 7.j} to $\msa_{i+1}[\Gamma_i]$.


Let $b=a\cup \bigcup_{i<N_\msa} a_i$ and for $\alpha\in \Gamma_{N_\msa-1}$,  
\[b_\alpha=\bigcup\{a'\in \msa^i: i<N_\msa, \msa_i(\alpha)\in a'\}.\]
 Let $I_\alpha\subseteq |b_\alpha|$ and $\langle J_{\alpha, i}: i<N_\msa\rangle$ be such that
\begin{enumerate}\setcounter{enumi}{1}
\item $b_\alpha[I_\alpha]\in \msa$, $b_\alpha[J_{\alpha, i}]\in \msa^i$ and $\msa_i(\alpha)\in b_\alpha[J_{\alpha, i}]$ for $i<N_\msa$.
\end{enumerate}
By (1) and the fact $\Gamma\subseteq \Gamma_{N_\msa-1}$, 
\begin{enumerate}\setcounter{enumi}{2}
\item $b_\alpha[J_{\alpha, i}]\in \msa^i[\Sigma_i]$.
\end{enumerate}

Now find $\Gamma\in [\Gamma_{N_\msa-1}]^{\omega_1}$, $I$ and $\langle J_i: i<N_\msa\rangle$ such that 
\[\msb=\{b_\alpha: \alpha\in\Gamma\}\text{ is non-overlapping   and $I=I_\alpha$, }J_i=J_{\alpha, i}\]
 whenever $\alpha\in \Gamma, i<N_\msa$. Now by definition of $b_\alpha$, 
\begin{enumerate}\setcounter{enumi}{3}
\item for every $\alpha\in \Gamma$, $b_\alpha=\bigcup_{i<N_\msa} b_\alpha[J_i]$. Hence, $N_\msb=\bigcup_{i<N_\msa} J_i$.
\end{enumerate}
 By (2)-(3),  for every $\alpha\in \Gamma$ and $i<N_\msa$,
\begin{enumerate}\setcounter{enumi}{4}
\item $b_\alpha[I]\in \msa$, $b_\alpha[J_{i}]\in \msa^i[\Sigma_i]$ and $\msa_i(\alpha)\in b_\alpha[J_{i}]$.
\end{enumerate}

Now by (1), (4)-(5) and Lemma \ref{lem 7.l},  for some $\alpha<\omega_1$, $(b, \msb\cap [\omega_1\setminus \alpha]^{<\omega})$ and $I$ are as desired.
\end{proof}

By extending the collection, $\varphi_1$ is preserved while adding a $\mcc$-candidate   to $\mcc$.   It turns out that continuously extending $\mcc$'s   at limit stages  preserves $\varphi_1$ as well (see also the proof of \cite[Lemma 16]{Peng25}).

  \begin{lem}\label{lem 7.m}
Suppose $\nu$ is a limit ordinal, $\langle \mcc^\alpha: \alpha\leq \nu\rangle$ is a sequence of collections such that $\mcc^\alpha\subseteq \mcc^\beta$ for $\alpha<\beta<\nu$. Suppose  $\mcc^\nu=\bigcup_{\alpha<\nu} \mcc^\alpha$ and $\mct^ \nu, \bfe^ \nu $ are induced from $\mcc^ \nu $. If $\varphi_1(\mcc^\alpha, \mct^\alpha, \bfe^\alpha)$ holds for every $\alpha<\nu$, then $\varphi_1(\mcc^\nu, \mct^\nu,\bfe^\nu)$ holds.
\end{lem}
\begin{proof}
For (T1), it suffices to prove that every $A\in \mct^\nu$ has finitely many predecessors.  But if $A$ has infinitely many predecessors, then 
\begin{itemize}
\item either there is a $\subset$-chain $A_0\subset A_1\subset\cdot\cdot\cdot$ in $\mct^\nu$ which induces $A_0(0)>A_1(0)>\cdot\cdot\cdot$ by (T4) for $\mct^\alpha$'s;
\item or there is a $\supset$-chain $A_0\supset A_1\supset\cdot\cdot\cdot$ in $\mct^\nu$ with non-empty intersection which yields a contradiction by proof of Lemma \ref{lem 7.f}.
\end{itemize}
Since neither of above cases can occur, (T1) holds. The rest properties only refer to  finitely many elements and follow from $\varphi_1(\mcc^\alpha, \mct^\alpha, \bfe^\alpha)$ for some $\alpha<\nu$.
\end{proof}

Note that $\varphi_1(\mcc, \mct, \bfe)$ is absolute between models with the same $\omega_1$. So in above lemma, although we describe the condition in one model, the conclusion holds if $\mcc^\alpha$'s are obtained in the iterated forcing. More precisely, the conclusion holds if for some iterated forcing $\langle \mc{P}_\alpha, \dot{\mc{Q}}_\beta: \alpha\leq \nu, \beta<\nu\rangle$,
\begin{itemize}
\item  $\varphi_1(\mcc^\alpha, \mct^\alpha, \bfe^\alpha)$ holds in $V^{\mc{P}_\alpha}$ and $\mc{P}_\nu$ preserves $\omega_1$.
\end{itemize}

When using the damage control structure $(\mcc, \mct, \bfe)$ in a specific problem, we need to verify several properties on $\mcc$-candidates in a forcing extension. In what follows, we present several cases of pulling a $\mcc$-candidate (or a $\mcc'$-candidate for some extension $\mcc'$ of $\mcc$) in a forcing extension back to another $\mcc$-candidate in the ground model.

 \begin{lem}\label{lem 7.n}
  Assume $\varphi_1(\mc{C}, \mc{T},\bfe)$. Suppose 
\begin{enumerate}[(i)]
\item $\mc{P}$ is a ccc poset;
\item for some $p\in \mc{P}$ and $(a, \dot{\msa})$, $p\Vdash (a, \dot{\msa})$ is a $\mcc$-candidate;
\item for every $\alpha<\omega_1$, there are $p_\alpha\leq p$ and $a_\alpha\in[\omega_1\setminus \alpha]^{<\omega}$ such that 
\[p_\alpha\Vdash a_\alpha\in \dot{\msa}.\]
\end{enumerate}
Then for every $\Gamma\in [\omega_1]^{\omega_1}$, there exists $\Sigma\in [\Gamma]^{\omega_1}$ such that $(a, \{a_\alpha: \alpha\in \Sigma\})$ is a $\mcc$-candidate.
\end{lem}
\begin{proof}
First by ccc of $\mc{P}$, find $q\leq p$ such that $q\Vdash |\{\alpha\in \Gamma: p_\alpha\in \dot{G}\}|=\omega_1$ where $\dot{G}$ is the canonical name of the generic filter. Extending $q$ if necessary, we may assume that 
  \begin{enumerate}
  \item for some $\mc{X}\in [\mcc]^{<\omega}$ and $N<\omega$, 
  \[q\Vdash N_{\dot{\msa}}=N\text{ and }\mc{C}_0(a, \dot{\msa})=\mcx,\]  
  and $q$ determines $\mc{C}_1(a, \dot{\msa})$.
  \end{enumerate}
  Note the existence of a finite $\mc{X}$ is guaranteed   by   Lemma \ref{lem 7.g}.

Now find $\Sigma'\in [\Gamma]^{\omega_1}$ such that
\begin{enumerate}\setcounter{enumi}{1}
\item for every $\alpha\in \Sigma'$, $q$ is compatible with $p_\alpha$;
\item $\msa'=\{a_\alpha: \alpha\in \Sigma'\}$ is non-overlapping.
\end{enumerate}
Recall that $\mc{P}$ is ccc. So for every $\alpha<\omega_1$, there exists $\beta<\omega_1$ such that $q\Vdash \dot{\msa}(\alpha)\subseteq \beta$.\footnote{Recall that $\msa(\alpha)$ is the $\alpha$th element of $\msa$.}  Find $\Sigma\in [\Sigma']^{\omega_1}$ such that
\begin{enumerate}\setcounter{enumi}{3}
\item  for every $\alpha<\omega_1$,   $q\Vdash a_{\Sigma(\alpha)}>\dot{\msa}(\alpha)$.
\end{enumerate}

By (1) and (2), 
\begin{enumerate}\setcounter{enumi}{4}
\item for every $(b, \msb)\in \mc{X}$,  $\msa'_i\subseteq \msb_j\text{ for some } i<N, j<N_\msb$.
\end{enumerate}
Fix $A\in \mct\setminus\{\msb_j: (b, \msb)\in \mc{X}, j<N_\msb\}$. Then 
\[q\Vdash |A\cap\dot{\msa}_i|\leq \omega\text{ for every }i<N.\]
Since $\mc{P}$ is ccc, for some $\alpha<\omega_1$, $q\Vdash A\cap\dot{\msa}_i\subseteq \alpha\text{ for every }i<N$.
Together with (2),
\begin{enumerate}\setcounter{enumi}{5}
\item $\msa'_i\cap A\subseteq \alpha$ for every $i<N$.
\end{enumerate}

Now it suffices to show that $\varphi_1(\mcc', \mct', \bfe')$ holds for $\mcc'=\mcc\cup \{(a, \msa')\}$ and $\mct', \bfe'$ induced from $\mcc'$. (C1) is trivial. The non-trivial case for (C2) is   $(b, \msb)\in \mcc$ and $\msa'_i\subseteq \msb_j$ for some $i<N, j<N_\msb$. By (6), $(b, \msb)\in \mc{X}$. Now, (C2) follows from (1) and (2).

By (5) and (6), $\msa'_i$'s are end nodes of $\mct'$.
Together with  the facts that    $(\mct, \supset)$ is a tree, (T1) holds. (T2) is trivial. (T3) follows from (C2) and (6). 

 The non-trivial case of (T4) is $\msb_j\supset \msa'_i$ for some $(b,\msb)\in \mc{X}$ and $i<N, j<N_\msb$. Fix $\alpha<\omega_1$.  Then $\msa'_i(\alpha)=a_{\Sigma(\alpha)}(i)$. By (4),   
 \[q\Vdash a_{\Sigma(\alpha)}(i)>\dot{\msa}_i(\alpha)>\msb_j(\alpha).\]
  So $a_{\Sigma(\alpha)}(i)>\msb_j(\alpha)$ and hence (T4) holds.

Then $\Sigma$ is as desired.
\end{proof}

 \begin{lem}\label{lem 7.o}
  Assume $\varphi_1(\mc{C}, \mc{T},\bfe)$. Suppose 
\begin{enumerate}[(i)]
\item $\mc{P}$ is a ccc poset and $(a,\msa)$ is a $\mcc$-candidate;
\item $\dot{\msa}^*$ is a $\mc{P}$-name of an uncountable subset of $\msa$;
\item for some $p\in \mc{P}$ and $(b, \dot{\msb})$, $p\Vdash (b, \dot{\msb})$ is a $\dot{\mcc}'$-candidate where $\dot{\mcc}'=\mcc\cup \{(a, \dot{\msa}^*)\}$;
\item for every $\alpha<\omega_1$, there are $p_\alpha\leq p$ and $b_\alpha\in[\omega_1\setminus \alpha]^{<\omega}$ such that 
\[p_\alpha\Vdash b_\alpha\in\dot{\msb}.\]
\end{enumerate}
Then for every $\Gamma\in [\omega_1]^{\omega_1}$, there are $q\leq p$ and $\Sigma\in [\Gamma]^{\omega_1}$ such that 
\begin{enumerate}[(i)]\setcounter{enumi}{4}
\item $(b, \{b_\alpha: \alpha\in \Sigma\})$ is a $\mcc$-candidate;
\item $q\Vdash (b, \{b_\alpha: \alpha\in \Sigma, p_\alpha\in \dot{G}\})$ is a $\dot{\mc{C}}'$-candidate where $\dot{G}$ is the canonical name of the generic filter.
\end{enumerate}
\end{lem}
\begin{proof}
First note that $p\Vdash (b, \dot{\msb})$ is a $\mcc$-candidate. Then by Lemma \ref{lem 7.n}, find $\Sigma\in [\Gamma]^{\omega_1}$ such that (v) holds. Then find $q\leq p$ such that 
\[q\Vdash \{\alpha\in \Sigma: p_\alpha\in \dot{G}\}\text{ is uncountable}.\]

Now (vi) follows from (iii) and the fact that
$q\Vdash \{b_\alpha: \alpha\in \Sigma, p_\alpha\in \dot{G}\}\subseteq \dot{\msb}$.
\end{proof}

\medskip

 \begin{lem}\label{lem 7.p}
  Assume $\nu$ is an infinite limit ordinal and $\langle \mc{P}_\alpha, \dot{\mc{Q}}_\beta: \alpha\leq \nu, \beta<\nu\rangle$ is a finite support iteration of ccc posets.
 Suppose 
\begin{enumerate}[(i)]
\item For $\alpha<\nu$, $\Vdash_{\mc{P}_\alpha} \varphi_1(\dot{\mc{C}}^\alpha, \dot{\mc{T}}^\alpha,\dot{\bfe}^\alpha)$ holds;
\item For $\alpha<\beta<\nu$, $\Vdash \dot{\mc{C}}^\alpha\subseteq \dot{\mc{C}}^\beta$ and $\dot{\mc{C}}^\nu=\bigcup_{\xi<\nu} \dot{\mc{C}}^\xi$;
\item for some $p\in \mc{P}_\nu$ and $(a, \dot{\msa})$, $p\Vdash (a, \dot{\msa})$ is a $\dot{\mcc}^\nu$-candidate;
\item for every $\alpha<\omega_1$, there are $p_\alpha\leq p$ and $a_\alpha\in [\omega_1\setminus \alpha]^{<\omega}$ such that 
\[p_\alpha\Vdash a_\alpha\in\dot{\msa}.\]
\end{enumerate}
Then for every $\Gamma\in [\omega_1]^{\omega_1}$, there are $\xi<\nu$, $q\leq p$ and  $\Sigma\in [\Gamma]^{\omega_1}$ such that for the canonical name  of the $\mc{P}_\nu$-generic filter $\dot{G}$,
\begin{enumerate}[(i)]\setcounter{enumi}{4}
\item $q\up\xi\Vdash_{\mc{P}_\xi} (a, \{a_\alpha: \alpha\in \Sigma, p_\alpha\up \xi \in \dot{G}\up\xi\})$ is a $\dot{\mcc}^\xi$-candidate;
\item $q\Vdash (a, \{a_\alpha: \alpha\in \Sigma, p_\alpha\in \dot{G}\})$ is a $\dot{\mc{C}}^\nu$-candidate.
\end{enumerate}
\end{lem}
\begin{proof}
We may assume that $\{supp(p_\alpha): \alpha\in \Gamma\}$ forms a $\Delta$-system with root contained in some $\xi'<\nu$ where $supp(p_\alpha)$ is the support of $p_\alpha$.

By ccc of $\mc{P}_\nu$, find $q\leq p$ such that $q\Vdash |\{\alpha\in \Gamma: p_\alpha\in \dot{G}\}|=\omega_1$. Extending $q$ if necessary, we may assume that 
  \begin{enumerate}
  \item for some $n<\omega$, a $\mc{P}_\nu$-name $\dot{f}$ of an injection from $n$ to  $\dot{\mcc}^\nu$ and $N<\omega$,  
  \[q\Vdash N_{\dot{\msa}}=N\text{ and }\dot{\mc{C}}^\nu_0(a, \dot{\msa})=rang(\dot{f}),\]
 and for $i<n$,  $q$ determines $ \mc{F}(\dot{\ms{B}}, \dot{\msa})$ where $\dot{f}(i)= (b, \dot{\ms{B}})$ for some $b$.
  \end{enumerate}
  
Moreover, assume   for some $\xi'\leq \xi<\nu$, $q\Vdash rang(\dot{f})\subseteq\dot{\mc{C}}^\xi$. So we also view $\dot{f}$ as a $\mc{P}_\xi$-name.

Now find $\Sigma\in [\Gamma]^{\omega_1}$ such that
\begin{enumerate}\setcounter{enumi}{1}
\item for every $\alpha\in \Sigma$, $q$ is compatible with $p_\alpha$;
\item $\{a_\alpha: \alpha\in \Sigma\}$ is non-overlapping;
\item  for every $\alpha<\omega_1$,   $q\Vdash a_{\Sigma(\alpha)}>\dot{\msa}(\alpha)$.
\end{enumerate}
By (2), extending $q$ and enlarging $\xi$, we may assume that 
\[supp(q)\subseteq \xi\text{ and }q\Vdash |\{\alpha\in \Sigma: p_\alpha\in \dot{G}\}|=\omega_1.\]

Let $G$ be a generic filter over $\mc{P}_\nu$ containing $q$. In $V[G\up\xi]$, let $f=(\dot{f})^{G\up_\xi}$ and $\msa'=\{a_\alpha: \alpha\in \Sigma, p_\alpha\up\xi\in G\up\xi\}$.
Note that $p\Vdash (a, \dot{\msa})$ is a $\dot{\mcc}^\xi$-candidate since  $\dot{\mcc}^\xi\subseteq \dot{\mcc}^\nu$. Then the argument of Lemma \ref{lem 7.n} shows that  in $V[G\up\xi]$, $(a, \msa')$ is a $\mcc^\xi$-candidate.

A density argument shows (v).

Finally, (vi) follows from (iii) and the fact that $q\Vdash  \{a_\alpha: \alpha\in \Sigma, p_\alpha\in \dot{G}\}\subseteq\dot{\msa}$. 
\end{proof}

\section{Distinguishing  {\rm P$_{\omega_1}$(K$_n\ra\sigma$-$n$-linked)} and {\rm P$_{\omega_1}$(K$_{n+1}\ra\sigma$-$(n+1)$-linked)}}

Throughout this section, we fix $n\geq 2$. Then we use the damage control structure to construct a model of {\rm P$_{\omega_1}$(K$_n\ra\sigma$-$n$-linked)} in which {\rm P$_{\omega_1}$(K$_{n+1}\ra\sigma$-$(n+1)$-linked)} fails.

We will need a coloring $\pi: [\omega_1]^{n+1}\ra 2 $ with additional properties. For this, we add $\pi$ generically. Let $\mc{Q}$ be the poset consisting of $p=(\pi_p, \sigma_p)$ such that
\begin{itemize}
\item $\pi_p: [D_p]^{n+1}\ra 2$ is a finite map for some $D_p\in [\omega_1]^{<\omega}$;
\item $\sigma_p: [D_p]^{\leq n}\cup \mc{H}^{\pi_p}_0\ra \omega$ with every $\sigma_p^{-1}\{m\}$ $n$-linked. 
\end{itemize}
The order is coordinatewise reverse extension. Suppose  $G$ is generic over $\mc{Q}$ and
\[\pi=\bigcup \{\pi_p: p\in G\},~ \sigma=\bigcup\{\sigma_p: p\in G\}.\]

The following fact is straightforward to verify.
\begin{fact}\label{fact 8.1}
Suppose $\pi, \sigma$ are as above. Then $[\omega_1]^{\leq n}\subseteq \mc{H}^\pi_0$ and  $\mc{H}^\pi_0$ is  $\sigma$-$n$-linked. Moreover,
for every uncountable non-overlapping family $\msa\subseteq [\omega_1]^{N_\msa}$  and every function $h:  N_\msa^{n+1}\ra 2$, there are $a_0<\cdots<a_{n}$ in $\msa$ such that for every $s\in N_\msa^{n+1}$, $\pi(\{a_{i}(s(i)): i\leq n\})=h(s)$.
\end{fact}

We will combine the following additional property of the coloring $\pi$ with the damage control structure       in the iteration.

\begin{defn}
Assume $\varphi_1(\mcc, \mct, \bfe)$. $\psi_1^-(\mcc, \mct, \bfe, \pi, \mcr)$ is the assertion that $\pi: [\omega_1]^{n+1}\ra 2$ is a coloring  and the following statements hold.
\begin{enumerate}[{\rm (R1)}]
\item $\mc{R}$ is a map with domain $\mc{C}\setminus\{(\emptyset, [\omega_1]^1)\}$ such that  for every $(a, \msa)\in dom(\mcr)$, $\mc{R}(a,\msa)=(a', I)$ for some $a'\subseteq a$ and $I\subseteq N_\msa$.
\item For  $\mcr(a, \msa)=(a', I)$, $\bigcup\{a'\cup b[I]:b\in \msa\}$ is 0-homogeneous.
     \end{enumerate}
\end{defn}
By (R2), if $\mcr(a, \msa)=(a', I)$, then $a'\cup b[I]\in \mc{H}^\pi_0$ for every $b\in \msa$.

In the iteration process, we will keep adding 0-homogeneous ($(n+1)$-linked) subsets of $\pi$. More specifically, when we deal with a $\Delta$-system of 0-homogeneous sets $\{p_\alpha: \alpha<\omega_1\}$ with root $\overline{p}$, we first use   Lemma \ref{lem 7.k} to find a $\mcc$-candidate $(a, \msb)$ and $I$ such that $\overline{p}\subseteq a$ and $\{b[I]: b\in \msb\}\subseteq \{p_\alpha\setminus \overline{p}: \alpha<\omega_1\}$. Then force an uncountable $\msa\in [\msb]^{\omega_1}$ such that $\bigcup \{\overline{p}\cup b[I]: b\in \msa\}$ is 0-homogeneous. In this case, $\mcr(a, \msa)=(\overline{p}, I)$
records the positions of the generically added 0-homogeneous subset.

We also need to preserve the property that $\mc{H}^\pi_0$ is not $\sigma$-$(n+1)$-linked. For this, we will need to reserve color 1 for appropriate candidates.  In fact, we will reserve the following stronger property.

  \begin{defn}
 Assume $\varphi_1(\mcc, \mct, \bfe)$. $\psi_1(\mcc, \mct, \bfe, \pi, \mcr)$ is the assertion that  $\psi_1^-(\mcc, \mct, \bfe, \pi, \mcr)$ together with the following statement hold.
  \begin{enumerate}[{\rm (Res)}]
  \item Suppose $(a,\msa)$ is a $\mcc$-candidate  and  $f:  N_\msa^{n+1}\ra 2 $ is a function with the following property.
  \begin{enumerate}[{\rm (Res.1)}]
  \item For every $s\in   N_\msa^{n+1}$, if for some $(b,\msb)\in \mc{C}_0(a, \msa)\setminus\{(\emptyset, [\omega_1]^1)\}$, $\mcr(b, \msb)=(b', I)$ and $rang(s)\subseteq \bigcup\{J[I]: J\in \mc{F}(\msb, \msa)\}$, then $f(s)=0$.
  \end{enumerate}
  \end{enumerate}
 Then there are $a_0<\cdots<a_{n}$ in $\msa$ such that for every $s\in   N_\msa^{n+1}$, $\pi(\{a_{i}(s(i)): i\leq n\})=f(s)$.
 
  Say $f$ is \emph{$\mc{R}$-satisfiable for $(a, \msa)$} if above condition {\rm (Res.1)} is satisfied.  
  \end{defn}
  
 In above definition, note that by (R2), $b'\cup \bigcup \{c[I]: c\in \msb\}$ and hence $\bigcup\{a'[J[I]]: a'\in \msa, J\in \mc{F}(\msb, \msa)\}$ is 0-homogeneous. So condition (Res.1) is necessary and simply states that $f$ takes value 0 at positions that are 0-homogeneous for $\pi$.

We will preserve $\psi_1(\mcc^\alpha, \mct^\alpha, \bfe^\alpha, \pi, \mcr^\alpha)$ in the iteration process. But first, we show that under appropriate assumptions, $\mc{H}^\pi_0$ is not $\sigma$-$(n+1)$-linked.
\begin{lem}\label{lem 8 C-candidate}
Assume $\varphi_1(\mcc, \mct, \bfe)$  and $|\mcc|<\mathfrak{a}_{\omega_1}$. For every stationary set $S\subseteq \omega_1$, there exists $X\in [S]^{\omega_1}$ such that $(\emptyset, [X]^1)$ is a $\mcc$-candidate. Moreover, $\mcc_0(\emptyset, [X]^1)=\{(\emptyset, [\omega_1]^1)\}$.
\end{lem}
\begin{proof}
First find $X'\in [S]^{\omega_1}$ such that $|X'\cap Y|\leq \omega$ for all $Y\in \mct_1$  where $\mct_1$ is the  level 1 of $\mct$. 
By (T3), elements in $\mct_1$ are pairwise almost disjoint (modulo $[\omega_1]^{\leq \omega}$). By (T4), elements in $\mct_1$ are non-stationary.
Let
\[\mct'=\{Y\in \mct_1: |S\cap Y|=\omega_1\}.\]

If $|\mct'|< \omega_1$, then $\bigcup \mct'$ is non-stationary and let $X'=S\setminus \bigcup\mct'$.

If $|\mct'|\geq \omega_1$, then the existence of $X'$ follows from $|\mct'|\leq |\mcc|+\omega<\mathfrak{a}_{\omega_1}$.

Find $X\in [X']^{\omega_1}$ such that $X(\alpha)>\alpha$ for all $\alpha<\omega_1$. Then by Lemma \ref{lem 7.i}, $(\emptyset, [X]^1)$ is a $\mcc$-candidate. 
\end{proof}
\begin{lem}\label{lem 8.2}
Assume $\varphi_1(\mcc, \mct, \bfe)$, $\psi_1(\mcc, \mct, \bfe, \pi, \mcr)$ and $|\mcc|<\mathfrak{a}_{\omega_1}$. Then $\pi$ has no stationary 0-homogeneous subset. In particular,
$\mc{H}^\pi_0$ is not $\sigma$-$(n+1)$-linked.
\end{lem}
\begin{proof}
Fix a stationary set $S\subseteq\omega_1$.
By Lemma \ref{lem 8 C-candidate}, find $X\in [S]^{\omega_1}$ such that $(\emptyset, [X]^1)$ is a $\mcc$-candidate and $\mcc_0(\emptyset, [X]^1)=\{(\emptyset, [\omega_1]^1)\}$.
Then the constant 1 function $f: \{0\}^{n+1}\ra \{1\}$ is $\mcr$-satisfiable for $(\emptyset, [X]^1)$.

Now applying $\psi_1(\mcc, \mct, \bfe, \pi, \mcr)$ to $(\emptyset, [X]^1)$ and $f$,  we get $b\in [X]^{n+1}$ such that $\pi(b)=1$.
This shows that $X$, and hence $S$, is not 0-homogeneous.
\end{proof}

Then we show that $\psi_1(\mcc, \mct, \bfe, \pi, \mcr)$ is preserved in the iteration process (with  a possible extension of $\mcc, \mct, \bfe, \mcr$).
\begin{lem}\label{lem 8.3}
Assume $\varphi_1(\mcc, \mct, \bfe)$ and $\psi_1(\mcc, \mct, \bfe, \pi, \mcr)$. If $\mc{P}$ is a poset with precaliber $\omega_1$, then $\Vdash \psi_1(\mcc, \mct, \bfe, \pi, \mcr)$.
\end{lem}
\begin{proof}
Only (Res) needs a verification. Fix $p\in \mc{P}$, a $\mc{P}$-name  $(a, \dot{\msa})$  and $f:   N_{\dot{\msa}}^{n+1}\ra 2$ such that
\[p\Vdash (a, \dot{\msa}) \text{ is a $\mcc$-candidate and $f$ is $\mcr$-satisfiable for } (a, \dot{\msa}).\]

For every $\alpha<\omega_1$, find $p_\alpha\leq p$ and $a_\alpha\in [\omega_1\setminus\alpha]^{N_{\dot{\msa}}}$ such that
\[p_\alpha\Vdash a_\alpha \text{ is the $\alpha$th element of } \dot{\msa}.\]
First find $\Gamma\in [\omega_1]^{\omega_1}$ such that 
\begin{enumerate}\setcounter{enumi}{1}
\item $\{p_\alpha: \alpha\in \Gamma\}$ is centered.
\end{enumerate}
Then by  Lemma \ref{lem 7.n}, find $\Sigma\in [\Gamma]^{\omega_1}$ such that
\begin{enumerate}\setcounter{enumi}{2}
\item $(a, \msa')$ is a $\mcc$-candidate where $\msa'=\{a_\alpha: \alpha\in \Sigma\}$.
\end{enumerate}

We claim that $f$ is $\mcr$-satisfiable for $(a, \msa')$. To see this, fix $q\leq p$ such that $q\Vdash \{\alpha\in \Sigma: p_\alpha\in \dot{G}\}$ is uncountable where $\dot{G}$ is the canonical name of the generic filter. Then 
\[q\Vdash \{a_\alpha: \alpha\in \Sigma\wedge p_\alpha\in \dot{G}\}\subseteq \dot{\msa}.\]
Consequently, $q\Vdash f \text{ is $\mcr$-satisfiable for } (a, \{a_\alpha: \alpha\in \Sigma\wedge p_\alpha\in \dot{G}\})$. Since $q\Vdash \{a_\alpha: \alpha\in \Sigma\wedge p_\alpha\in \dot{G}\}\subseteq \msa'$,  $f$ is $\mcr$-satisfiable for $(a, \msa')$.

By $\psi_1(\mcc, \mct, \bfe, \pi, \mcr)$, find $a_{\alpha_0}<\cdots<a_{\alpha_{n}}$ in $\msa'$ witnessing (Res) for $(a, \msa')$ and $f$. By (2), fix a common lower bound $r$ of $\{p_{\alpha_i}: i\leq n\}$. Then 
\[r\Vdash a_{\alpha_0},...,a_{\alpha_{n}} \text{ witness (Res) for } (a, \dot{\msa}) \text{ and } f.\]

Now a standard density argument shows $\Vdash \psi_1(\mcc, \mct, \bfe, \pi, \mcr)$.
\end{proof}

Now we introduce the second kind of forcing we shall use. 
\begin{defn}\label{defn PaAaI}
Assume $\varphi_1(\mcc, \mct, \bfe)$ and $(a, \msa)$ is a $\mcc$-candidate. Suppose $a'\subseteq a$ and $I\subseteq N_\msa$ satisfy $a'\cup b[I]\in \mc{H}^\pi_0$ for all $b\in \msa$. Then 
\[\mc{P}_{a, \msa, a', I}\text{ is the poset consisting of $F\in [\msa]^{<\omega}$ such that }a'\cup \bigcup_{b\in F} b[I]\in \mc{H}^\pi_0.\] 
The order is reverse inclusion.
\end{defn}

For the coloring $\pi$ we shall use, $\mc{H}^\pi_0$ is $\sigma$-$n$-linked. Then $\mc{P}_{a, \msa, a', I}$ is $\sigma$-$n$-linked and hence ccc.

\begin{lem}\label{lem 8.4}
Assume $\varphi_1(\mcc, \mct, \bfe)$, $\psi_1(\mcc, \mct, \bfe, \pi, \mcr)$ and $\mc{H}^\pi_0$ is $\sigma$-$n$-linked. Suppose $(a, \msa), a', I$ and $ \mc{P}_{a, \msa, a', I}$ are as above. If $G$ is an uncountable generic filter,  then in $V[G]$, $\varphi_1(\mcc', \mct', \bfe')$ and $\psi_1(\mcc', \mct', \bfe', \pi, \mcr')$ hold where $\mcc'=\mcc\cup \{(a, \bigcup G)\}$, $\mct', \bfe'$ are induced from $\mcc'$  and $\mcr'=\mcr\cup \{((a, \bigcup G), (a', I))\}$.
\end{lem}
\begin{proof}
First note that $(a,\msa)$ is a $\mcc$-candidate and $\bigcup G\subseteq \msa$. Then $(a, \bigcup G)$ is also a $\mcc$-candidate. So $\varphi_1(\mcc', \mct', \bfe')$ holds in $V[G]$.

To see $\psi_1(\mcc', \mct', \bfe', \pi, \mcr')$, first note that (R1) and (R2) follow from the definition of $\mcr'$ and our forcing condition.

Now we check (Res). Fix a $\mcc'$-candidate $(b, \msb)$ and $f:   N_\msb^{n+1}\ra 2$ that is $\mcr'$-satisfiable for $(b, \msb)$. Let 
\[N=N_\msb.\]


Let $\dot{\msb}$ be a $\mc{P}_{a, \msa, a', I}$-name of $\msb$ and $F\in \mc{P}_{a, \msa, a', I}$ such that
\begin{enumerate}
  \item $F\Vdash f$ is $\dot{\mcr}'$-satisfiable for $(b, \dot{\msb})$. Moreover, $F$ determines $\dot{\mc{C}}'_1(b, \dot{\msb})$.
  \end{enumerate}

For every $\alpha<\omega_1$, choose $F_\alpha\leq F$ and $b_\alpha\in [\omega_1]^N$ such that
\begin{enumerate}\setcounter{enumi}{1}
\item $F_\alpha\Vdash b_\alpha \text{ is the $\alpha$th element of } \dot{\msb}$.
\end{enumerate}
Extending $F_\alpha$ if necessary, we may assume that
\begin{enumerate}\setcounter{enumi}{2}
\item if $F\Vdash (a, \bigcup \dot{G})\prec (b, \dot{\msb})$, then $b_\alpha[J]\in F_\alpha$ for every $J\in \mc{F}(\bigcup \dot{G}, \dot{\msb})$.
\end{enumerate} 

Now choose $\Gamma\in [\omega_1]^{\omega_1}$ and $I^*$ such that,
\begin{enumerate}\setcounter{enumi}{3}
\item if $F\Vdash (a, \bigcup \dot{G})\not\prec (b, \dot{\msb})$, then for every $\alpha\in \Gamma$, $b_\alpha\cap (\bigcup F_\alpha)=\emptyset$;
\item $\{a'\cup \bigcup_{c\in F_\alpha}c[I]: \alpha\in \Gamma\}$ is $n$-linked in $\mc{H}^\pi_0$;
\item $|F_\alpha|$ is constant for $\alpha\in \Gamma$ and $\{F_\alpha: \alpha\in \Gamma\}$ is a $\Delta$-system with root $\overline{F}$;
\item  for every $\alpha\in \Gamma$, $c_\alpha[I^*]=b_\alpha$ where $c_\alpha=b_\alpha\cup \bigcup (F_\alpha\setminus \overline{F})$ and $\{c_\alpha: \alpha\in \Gamma\}$ is non-overlapping.
\end{enumerate}
By Lemma \ref{lem 7.l},  find $\Gamma'\in [\Gamma]^{\omega_1}$ such that
\begin{enumerate}\setcounter{enumi}{7}
\item     $(a, \msa')$ is a $\mcc$-candidate where $\msa'=\{\bigcup (F_\alpha\setminus \overline{F}): \alpha\in \Gamma'\}$.
\end{enumerate}
 
 Apply Lemma \ref{lem 7.o} to find  $\Sigma\in [\Gamma']^{\omega_1}$ such that
 \begin{enumerate}\setcounter{enumi}{8}
\item $(b, \msb')$ is a $\mcc$-candidate where $\msb'=\{b_\alpha: \alpha\in \Sigma\}$;
\end{enumerate}

By (3)-(4), (6)-(9) and Lemma \ref{lem 7.l}, omitting a countable subset of $\Sigma$, we may assume that 
 \[(a\cup b, \ms{C})\text{ is a $\mcc$-candidate where }\ms{C}=\{c_\alpha: \alpha\in \Sigma\}.\]

Now define $h:  N_\ms{C}^{n+1}\ra 2$ by for every  $t\in N_\ms{C}^{n+1}$,
\begin{enumerate}\setcounter{enumi}{9}
\item if $t\in (I^*)^{n+1}$, $h(t)=f(s)$ where $t(i)=I^*(s(i))$ for all $i\leq n$;
\item otherwise, $h(t)=0$.
\end{enumerate}
So $h$ copies $f$ on $(I^*)^{n+1}$ and equals 0 on $N_\ms{C}^{n+1}\setminus (I^*)^{n+1}$.\medskip

\textbf{Claim.} $h$ is $\mcr$-satisfiable for $(a\cup b, \ms{C})$. 
\begin{proof}[Proof of Claim.]
Fix $t\in N_\ms{C}^{n+1}$ and $(c, \ms{C}')\prec (a\cup b, \ms{C})$ in $\mcc\setminus\{(\emptyset, [\omega_1]^1)\}$. Assume 
\[\mcr(c, \ms{C}')=(c', J)\text{ and } rang(t)\subseteq \bigcup\{L[J]: L\in \mc{F}(\ms{C}', \ms{C})\}.\]

It suffices to show $h(t)=0$. The non-trivial case is $t\in (I^*)^{n+1}$. Let $s\in N^{n+1}$ be such that $t(i)=I^*(s(i))$ for all $i\leq n$. Note $(\bigcup \ms{C}')\cap (\bigcup \msb')\supseteq \ms{C}_i$ is uncountable. By (9), $(c, \ms{C}')\prec (b, \msb')$. 
We will check that
\[rang(s)\subseteq  \bigcup\{L'[J]: L'\in \mc{F}(\ms{C}', \msb')\}.\]
Then $f(s)=0$ by (1) and hence $h(t)=0$ by (10). To verify above inclusion, fix $i<n+1$.

Let $L\in \mc{F}(\ms{C}', \ms{C})$ satisfy $t(i)\in L[J]$. Then $t(i)\in L\cap I^*$. By (C2) of $\varphi_1(\mcc\cup\{(b, \msb'\})$ (or Lemma \ref{lem 7.i} (a)), $L\subseteq I^*$. Consequently, 
\[L=I^*[L'] \text{ for some } L'\in \mc{F}(\ms{C}', \msb').\]

By definition of $s$, $s(i)\in L'[J]$.
\end{proof}

Applying (Res) to $(a\cup b, \ms{C})$ and $h$, we find witnesses $c_{\alpha_0}<\cdots<c_{\alpha_{n}}$ in $\ms{C}$. 
We check that $F^*=\bigcup_{i< n+1} F_{\alpha_i}\in \mc{P}_{a, \msa, a', I}$. Arbitrarily choose $d\in [a'\cup \bigcup_{c\in F^*} c[I]]^{n+1}$. It suffices to prove that $\pi(d)=0$.

If   for some $x\in [\{\alpha_i: i< n+1\}]^{\leq n}$, $d\subseteq a'\cup \bigcup_{\beta\in x}\bigcup_{c\in  F_\beta} c[I]$, then by (5), $\pi(d)=0$.

Now suppose for some  $t\in N_\ms{C}^{n+1}$, $d(i)= c_{\alpha_{i}}(t(i))$. Then $\pi(d)=h(t)$. We may assume that $h(t)$ is defined according to (10). Let $s$ be such that $t(i)=I^*(s(i))$.
By (3)-(4), $F\Vdash (a, \bigcup \dot{G})\prec (b, \dot{\msb})$ and for every $i$,
\[d(i)\in c[I] \text{ for some $c\in F_{\alpha_i}$ with } c\subseteq b_{\alpha_i}.\]
Note that above $c$ equals $b_{\alpha_i}[J]$ for some $J$ in $\mc{F}(\bigcup \dot{G}, \dot{\msb})$. Consequently, $s(i)\in J[I]$.
Together with (1), $f(s)=0$ and hence $\pi(d)=h(t)=0$. This finishes the verification that
\[F^*=\bigcup_{i\leq n} F_{\alpha_i}\in \mc{P}_{a, \msa, a', I}.\]
Then by (2) and (10), $F^*\Vdash b_{\alpha_0},...,b_{\alpha_{n}}\text{ are in } \dot{\msb}\text{ and witness (Res) for $(b, \dot{\msb})$ and } f$.

Now a density argument shows that (Res) holds. Hence, $\psi_1(\mcc', \mct', \bfe', \pi, \mcr')$ holds in $V[G]$.
\end{proof}

We now check the preservation of $\psi_1(\mcc, \mct, \bfe, \pi, \mcr)$ at limit stages.
  \begin{lem}\label{lem 8.5}
Suppose $\nu$ is a limit ordinal and $\langle \mc{P}_\alpha, \dot{\mc{Q}}_\beta: \alpha\leq \nu, \beta<\nu\rangle$ is a finite support iteration of ccc posets. Moreover, for $\alpha\leq \nu$,
\begin{itemize}
\item $(\mcc^\alpha, \mct^\alpha, \bfe^\alpha, \pi, \mcr^\alpha)\in V^{\mc{P}_\alpha}$;
\item  $\mcc^\xi\subseteq \mcc^\alpha$ for $\xi<\alpha$, $\mcc^\nu=\bigcup_{\xi<\nu} \mcc^\xi$ and $\mcr^\nu=\bigcup_{\xi<\nu} \mcr^\xi$;
\item  $\mct^\alpha, \bfe^\alpha$ are induced from $\mcc^\alpha$ and $\mcr^\alpha$ is a map defined on $\mcc^\alpha\setminus \{(\emptyset, [\omega_1]^1)\}$. 
\end{itemize}
If $\varphi_1(\mcc^\alpha, \mct^\alpha, \bfe^\alpha)$ and $\psi_1(\mcc^\alpha, \mct^\alpha, \bfe^\alpha, \pi, \mcr^\alpha)$ hold in $V^{\mc{P}_\alpha}$ for all $\alpha<\nu$, then $\psi_1(\mcc^\nu, \mct^\nu,\bfe^\nu, \pi, \mcr^\nu)$ holds in $V^{\mc{P}_\nu}$.
\end{lem}
\begin{proof}
First recall by Lemma \ref{lem 7.m}, $\varphi_1(\mcr^\nu, \mct^\nu, \bfe^\nu)$ holds. It suffices to prove that condition (Res) holds in  $V^{\mc{P}_\nu}$.

Fix  in  $V^{\mc{P}_\nu}$, a $\mcc^\nu$-candidate $(a, \msa)$ and an $\mcr^\nu$-satisfiable for $(a, \msa)$ function $f$. Viewing some $V^{\mc{P}_\alpha}$ as the ground model if necessary, we may assume that 
\begin{enumerate}
\item all $(b, \msb)\prec (a, \msa)$ in $\mcc^\nu$ are in $\mcc^0$.
\end{enumerate}

Fix $p\in \mc{P}_\nu$, $N<\omega$ and a $\mc{P}_\nu$-name $\dot{\msa}$ of $\msa$ such that 
\begin{enumerate}\setcounter{enumi}{1}
\item $p$ forces (1) and determines all $(b, \msb)\prec (a, \dot{\msa})$ in $\mcc^0$;
\item $p\Vdash N=N_{\dot{\msa}}\text{ and } f \text{ is } \dot{\mcr}^\nu \text{ satisfiable for } (a, \dot{\msa})$.
\end{enumerate}
For every $\alpha<\omega_1$, find $p_\alpha\leq p$ and $a_\alpha\in [\omega_1]^N$ such that
\[p_\alpha\Vdash a_\alpha\text{ is the $\alpha$th element of } \dot{\msa}.\]
Find $\Gamma\in [\omega_1]^{\omega_1}$ such that 
\begin{enumerate}\setcounter{enumi}{3}
\item $\{supp(p_\alpha): \alpha\in \Gamma\}$ forms a $\Delta$-system with root contained in $\xi'<\nu$.
\end{enumerate}

By Lemma \ref{lem 7.p}, find $\xi\in [\xi', \nu)$, $q\leq p$ and  $\Sigma\in [\Gamma]^{\omega_1}$ such that for the canonical name  of the $\mc{P}_\nu$-generic filter $\dot{G}$,
\begin{enumerate}\setcounter{enumi}{4}
\item $q\up\xi\Vdash_{\mc{P}_\xi} (a, \{a_\alpha: \alpha\in \Sigma, p_\alpha\up \xi \in \dot{G}\up\xi\})$ is a $\dot{\mcc}^\xi$-candidate.
\end{enumerate}

Let $G$ be a $\mc{P}_\nu$-generic filter containing $q$, 
\[\Sigma'=\{\alpha\in \Sigma: p_\alpha\up\xi\in G\up\xi\}\text{ and } \msa'=\{a_\alpha: \alpha\in \Sigma'\}.\]

Work in $V[G\up\xi]$.  By (5),  $(a, \msa')$ is a $\mcc^\xi$-candidate. By (2) and (3), $f$ is $\mcr^\xi$-satisfiable for $(a, \msa')$. Applying $\psi_1(\mcc^\xi, \mct^ \xi, \bfe^ \xi, \pi, \mcr^ \xi)$, we get $a_{\alpha_0}<\cdots<a_{\alpha_{n}}$ in $\msa'$ witnessing (Res) for $(a, \msa')$ and $f$.

By (4) and the fact $\Sigma'\subseteq \Gamma$, $p_{\alpha_0},...,p_{\alpha_{n}}$ have a common lower bound, say $p'$. Now 
\[p'\Vdash a_{\alpha_0}<\cdots<a_{\alpha_{n}} \text{   witness (Res) for $(a, \dot{\msa})$ and } f.\]

Now a density argument shows that (Res) and hence $\psi_1(\mcc^\nu, \mct^\nu,\bfe^\nu, \pi, \mcr^\nu)$ holds in $V^{\mc{P}_\nu}$.
\end{proof}

\begin{proof}[Proof of Theorem \ref{thm K3tosK3 n+1}.]
Start from a model of GCH in which there is a coloring $\pi: [\omega_1]^{n+1}\ra 2$ satisfying the conclusion of Fact \ref{fact 8.1}. Then iteratively force ccc posets with finite support $\langle \mc{P}_\alpha, \dot{\mc{Q}}_\beta: \alpha\leq \omega_2, \beta<\omega_2\rangle$ such that 
\begin{enumerate}
\item for $\beta<\omega_2$, $\Vdash_{\mc{P}_\beta} \dot{\mc{Q}}_\beta$ either has precaliber $\omega_1$ and has size $\leq \omega_1$ or is of form $\mc{P}_{a, \dot{\msa}, a', I}$ (see Definition \ref{defn PaAaI}).
\end{enumerate}

Define $\langle (\dot{\mcc}^\alpha, \dot{\mct}^\alpha, \dot{\bfe}^\alpha, \dot{\mcr}^\alpha): \alpha\leq \omega_2\rangle$ in the following standard way.
 \begin{itemize}
 \item $\dot{\mct}^\alpha$ and $\dot{\bfe}^\alpha$ are induced from $\dot{\mcc}^\alpha$.
\item $\mcc^0=\{(\emptyset, [\omega_1]^1)\}$, $\mcr^0$ is the null function.
\item For limit $\alpha$, $\Vdash_{\mc{P}_\alpha} \dot{\mcc}^\alpha=\bigcup_{\xi<\alpha} \dot{\mcc}^\xi$, $\dot{\mcr}^\alpha=\bigcup_{\xi<\alpha} \dot{\mcr}^\xi$. 
\item Suppose $\alpha=\beta+1$.  $\Vdash_{\mc{P}_\alpha} $ if $\dot{\mc{Q}}_\beta$ has precaliber $\omega_1$, then $\dot{\mcc}^\alpha=\dot{\mcc}^\beta$,  $\dot{\mcr}^\alpha=\dot{\mcr}^\beta$ and if $ \dot{\mc{Q}}_\beta$ has form $\mc{P}_{a, \dot{\msa}, a', I}$, then $\dot{\mcc}^\alpha=\dot{\mcc}^\beta\cup \{(a, \bigcup \dot{G}(\beta))\}$, $\dot{\mcr}^\alpha=\dot{\mcr}^\beta\cup \{((a, \bigcup \dot{G}(\beta)), (a', I)\}$ where $\dot{G}(\beta)$ is the canonical name of the $\dot{\mc{Q}}_\beta$-generic filter over $V^{\mc{P}_\beta}$.
\end{itemize}

Recall  by Fact \ref{fact 8.1},   $\mc{H}^\pi_0$ is $\sigma$-$n$-linked and hence 
\begin{enumerate}\setcounter{enumi}{1}
\item all $\mc{P}_{a, \msa, a', I}$'s  have property K$_n$ and $\mc{P}_{\omega_2}$ is a finite support iteration of posets with property K$_n$. In particular, $\mc{P}_{\omega_2}$ has property K$_n$.
\end{enumerate}

Recall the fact: If $\mc{P}$ has property K$_n$ and $\mc{Q}$ does not have property K$_n$, then $\Vdash_\mc{P} \mc{Q}$ does not have property K$_n$. To see this, suppose $\mc{X}\in [\mc{Q}]^{\omega_1}$ has no uncountable $n$-linked subset. Then $\mc{X}$ has no uncountable $n$-linked subset in $V^\mc{P}$ since otherwise some uncountable $n$-linked subset of $\mc{P}$ would induce an uncountable $n$-linked subset of $\mc{X}$ in $V$.

Now by above fact and (2),
\begin{enumerate}\setcounter{enumi}{2}
\item if $\mc{Q}$ has property K$_n$ in $V^{\mc{P}_{\omega_2}}$ and $\mc{Q}\in V^{\mc{P}_\alpha}$ for some $\alpha<\omega_2$, then $\mc{Q}$ has property K$_n$ in $V^{\mc{P}_\alpha}$.
\end{enumerate}

Arrange the iteration $\langle \mc{P}_\alpha, \dot{\mc{Q}}_\beta: \alpha\leq \omega_2, \beta<\omega_2\rangle$ in a standard bookkeeping way such that
\begin{enumerate}\setcounter{enumi}{3}
\item if $\mc{Q}\in H(\omega_2)^{V^{\mc{P}_{\omega_2}}}$ has precaliber $\omega_1$ in $V^{\mc{P}_\alpha}$ whenever  $\alpha<\omega_2$ is sufficiently large, then $\mc{Q}$ is forced  at some stage $\beta$ (as $\dot{\mc{Q}}_\beta$);
\item if $\{a_\xi\in \mc{H}^\pi_0: \xi<\omega_1\}\in V^{\mc{P}_{\omega_2}}$ is a $\Delta$-system with root $a'$, then there are $\alpha<\omega_2$, a $\mcc^\alpha$-candidate $(a, \msa)\in V^{\mc{P}_\alpha}$ and $I\subseteq N_\msa$ such that $a'\subseteq a$, $\{b[I]: b\in \msa\}\subseteq \{a_\xi\setminus a': \xi<\omega_1\}$ and   $\mc{P}_{a, \msa, a', I}$ is forced at stage $\alpha$.
\end{enumerate}
Note that (5) is guaranteed by Lemma \ref{lem 7.k}.

By (1), our definitions of $\mcc^\alpha$'s and Lemma \ref{lem 7.m}, 
\begin{enumerate}\setcounter{enumi}{5}
\item $\varphi_1(\mcc^\alpha, \mct^\alpha, \bfe^\alpha)$ holds for every $\alpha\leq \omega_2$.
\end{enumerate}
 Then by Lemma \ref{lem 8.3}, Lemma \ref{lem 8.4} and Lemma \ref{lem 8.5}, 
 \begin{enumerate}\setcounter{enumi}{6}
\item $\psi_1(\mcc^\alpha, \mct^\alpha, \bfe^\alpha, \pi, \mcr^\alpha)$ holds in $V^{\mc{P}_\alpha}$ for every $\alpha\leq \omega_2$.
\end{enumerate}

We first show that P$_{\omega_1}($K$_n\ra \sigma$-$n$-linked) holds in $V^{\mc{P}_{\omega_2}}$. For this, fix a poset $\mc{Q}$ having property K$_n$ and size $\omega_1$ in $V^{\mc{P}_{\omega_2}}$. We may assume that $\mc{Q}\in H(\omega_2)$. 

By (3), $\mc{Q}$ has property K$_n$ for all $\alpha<\omega_2$ such that $\mc{Q}\in V^{\mc{P}_\alpha}$. 

By Lemma \ref{lem 6.2}, $(\mc{H}_{\mc{Q}, n})^\omega$ has precaliber $\omega_1$ in $V^{\mc{P}_\alpha}$ for sufficiently large $\alpha<\omega_2$. Together with (4) and Lemma \ref{lem 6.1}, $\mc{Q}$ is $\sigma$-$n$-linked in $V^{\mc{P}_{\omega_2}}$.

We then show that $\mc{H}^\pi_0$ witnesses the failure of P$_{\omega_1}($K$_{n+1}\ra \sigma$-$(n+1)$-linked) in $V^{\mc{P}_{\omega_2}}$.

To see that $\mc{H}^\pi_0$ has property K$_{n+1}$, fix a $\Delta$-system $\{a_\xi\in \mc{H}^\pi_0: \xi<\omega_1\}\in V^{\mc{P}_{\omega_2}}$   with root $a'$. Then by (5), some uncountable $(n+1)$-linked subset of $\{a_\xi\in \mc{H}^\pi_0: \xi<\omega_1\}$ will be added at some stage $\alpha<\omega_2$. Hence, $\mc{H}^\pi_0$ has property K$_{n+1}$.

To see that $\mc{H}^\pi_0$ is not $\sigma$-$(n+1)$-linked, note that $\omega_2=\mathfrak{b}_{\omega_1}\leq \mathfrak{a}_{\omega_1}$ in $V^{\mc{P}_{\alpha}}$ for all $\alpha\leq \omega_2$.
Then by Lemma \ref{lem 8.2}, (6) and (7), $\mc{H}^\pi_0$ is not $\sigma$-$(n+1)$-linked in $V^{\mc{P}_\alpha}$ whenever $\alpha<\omega_2$.
So  $\mc{H}^\pi_0$ is not  $\sigma$-$(n+1)$-linked in $V^{\mc{P}_{\omega_2}}$.

Now $V^{\mc{P}_{\omega_2}}$ is a model of P$_{\omega_1}($K$_n\ra \sigma$-$n$-linked) in which P$_{\omega_1}($K$_{n+1}\ra \sigma$-$(n+1)$-linked)  fails.
\end{proof}

\section{Weakening the weaker property}

In this section, we will use the damage control structure in Section 7 to show Theorem \ref{thm k2 not k3}. Then the following corollary follows from the fact that MA$_{\omega_1}$ is equivalent to $\ms{K}_2$+MA$_{\omega_1}$(K).
\begin{cor}
\begin{enumerate}
\item $\mc{K}_2$ does not imply $\ms{K}_2$.
\item $\mc{K}_2$ does not imply $\mc{K}_3$.
\end{enumerate}
\end{cor}

Recall for $n\geq 2$, $\mc{K}_n$ is the assertion that every ccc coloring $c: [\omega_1]^n\ra 2$ has an uncountable 0-homogeneous subset.

\begin{defn}
For $n\geq 2$, a ccc poset $\mathbb{P}$ has $n$-ccc coloring if $\mc{H}_{\mathbb{P}, n}$ is ccc.\footnote{See Definition \ref{defn hpn}.}
\end{defn}
So $\mc{K}_n$ is equivalent to the property P($n$-ccc coloring$\ra$K$_n$) and MA$_{\omega_1}$($n$-ccc coloring) is equivalent to $\mc{K}_n$+MA$_{\omega_1}$(K$_n$).

The argument in \cite[pp 837]{Todorcevic91} (or the proof of Lemma \ref{lem 6.2}) shows that every powerfully ccc poset has $n$-ccc coloring for $n\geq 2$. So the following proposition follows from Theorem \ref{thm K3toK4}.

\begin{prop}\label{prop K3MAp}
$\mc{K}_3$ implies {\rm MA}$_{\omega_1}$(powerfully ccc). In particular, $\mc{K}_n$ is equivalent to  {\rm MA}$_{\omega_1}$($n$-ccc coloring) for $n\geq 3$.
\end{prop}
By \cite{Peng25}, {\rm MA$_{\omega_1}$(powerfully ccc)} does not imply $\mc{K}_2$. So $\mc{K}_n$ for $n\geq 3$ or P(2-ccc coloring$\ra$K$_3$) is strictly stronger than {\rm MA$_{\omega_1}$(powerfully ccc)}.


For the rest of the section, we prove Theorem \ref{thm k2 not k3}. We will construct a coloring $\pi: [\omega_1]^3\ra 2$ to witness the failure of $\mc{K}_3$ and iteratively force {\rm MA$_{\omega_1}$(2-ccc coloring)}. First, we introduce additional properties of $\pi$ to be preserved in the iteration process.

\subsection{Associating a coloring to the damage control structure}
\begin{defn}
Assume $\varphi_1(\mcc, \mct, \bfe)$. $\psi_2^-(\mcc, \mct, \bfe, \pi, \mcr)$ is the assertion that $\pi: [\omega_1]^{3}\ra 2$ is a coloring and the following statements hold.
\begin{enumerate}[{\rm (R1)}]
\item For every $(a, \msa)\in \mcc$ and $a', a''$ in $\msa$, $(a\cup a', \pi)$ is isomorphic to $(a \cup a'', \pi)$.
\item $\mc{R}$ is a map on $\mc{C}$ such that  for every $(a, \msa)\in\mcc$, $\mc{R}(a,\msa)$ is the collection of functions $f: [|a|+2N_\msa]^3\ra 2$ such that $(|a|+2N_\msa, f)$ is isomorphic to $(a\cup a'\cup a'', \pi)$ for some $a'<a''$ in $\msa$.
\item For every $(a, \msa)\in \mcc$, for every $a^*\subseteq a$ and $I, J\subseteq N_\msa$ with 
\[\forall a'\in \msa ~ a^*\cup a'[I]\in \mc{H}^\pi_0 \wedge a^*\cup a'[J]\in \mc{H}^\pi_0,\]
 there are $a'<a''$ in $\msa$ such that $a^*\cup a'[I]\cup a'[J]\in \mc{H}^\pi_0$.
     \end{enumerate}
\end{defn}
Recall   the definition of isomorphism between functions   in Definition \ref{defn iso}. In (R2), if $|a|+2N_\msa=2$, then $\mcr(a, \msa)$ consists of the null function.

 In above definition, we add one more requirement, (R1), to the collection $\mcc$. Since    (R1) depends on $\pi$, we put this property in $\psi_2^-(\mcc, \mct, \bfe, \pi, \mcr)$ instead of $\varphi_1(\mcc, \mct, \bfe)$.

 The following notation will be frequently used.
 
   \textbf{Notation.} For $a\in [\omega_1]^{<\omega}$, $b\subseteq a$ and $x<\omega$,
  \[a^{-1}[b]=\{i<|a|: a(i)\in b\} \text{ and } x\widetilde{+}I=\{x+k: k\in I\}.\]
  
    (R3) is used in the preservation of ccc for $\mc{H}^\pi_0$. And under the assumption (R2), (R3) can be reformulated as follows.
  \begin{enumerate}[{\rm (R3)$^*$}]
  \item For every $(a, \msa)\in \mcc$, for every $a^*\subseteq a$ and $I, J\subseteq N_\msa$ with 
\[\forall a'\in \msa ~ a^*\cup a'[I]\in \mc{H}^\pi_0 \wedge a^*\cup a'[J]\in \mc{H}^\pi_0,\]
 there exists $f\in \mcr(a, \msa)$ that is constant 0 on $[a^{-1}[a^*]\cup (|a|\widetilde{+} I)\cup (|a|+N_\msa\widetilde{+} J)]^3$.
  \end{enumerate}

We use the following property to reserve all patterns not conflicting requirements recorded by $\mcr$.
  \begin{defn}
 Assume $\varphi_1(\mcc, \mct, \bfe)$. $\psi_2(\mcc, \mct, \bfe, \pi, \mcr)$ is the assertion that  $\psi_2^-(\mcc, \mct, \bfe, \pi, \mcr)$ together with the following statement hold.
  \begin{enumerate}[{\rm (Res)}]
  \item Suppose $(a,\msa)$ is a $\mcc$-candidate with the following property:
  \begin{enumerate}[{\rm (R1)$^*$}]
  \item for every $a', a''$ in $\msa$, $(a\cup a',\pi)$ is isomorphic to $(a\cup a'', \pi)$.
  \end{enumerate}  
  Suppose  $f: [|a|+2N_\msa]^3\ra 2 $ is a function with the following properties.
  \begin{enumerate}[{\rm (Res.1)}]
  \item For $i<2$, $(|a|\cup [|a|+iN_\msa, |a|+(i+1)N_\msa), f)$ is isomorphic to $(a\cup a', \pi)$ for some $a'\in \msa$.
  \item For every  $(b, \msb, I, J)\in \mc{C}_2(a, \msa)$,   there exists $g\in \mcr(b, \msb)$ such that $(a^{-1}[b]\cup (|a|\widetilde{+} I)\cup (|a|+N_\msa\widetilde{+}J), f)$ is isomorphic to $(|b|+2N_\msb, g)$.
  \end{enumerate}
   Then there are $a_0<a_{1}$ in $\msa$ such that $(a\cup a_0\cup a_1, \pi)$ is isomorphic to $(|a|+2N_\msa, f)$.
  \end{enumerate}

  Say $f$ is \emph{$\mc{R}$-satisfiable for $(a, \msa)$} if above conditions {\rm (Res.1)-(Res.2)} are satisfied.  
  \end{defn}

  \textbf{Remark.} Suppose $(a\cup a_0\cup a_1, \pi)$ is isomorphic to $(|a|+2N_\msa, f)$ for some $f$ and $a_0<a_{1}$ in $\msa$. Then on one hand, for $i<2$, 
  \[(|a|\cup [|a|+iN_\msa, |a|+(i+1)N_\msa), f)\text{ is isomorphic to }(a\cup a_i, \pi).\]
   On the other hand, for  $(b, \msb, I, J)\in\mc{C}_2(a, \msa)$, 
  \begin{align*}
   &(a^{-1}[b]\cup (|a|\widetilde{+}I)\cup (|a|+N_\msa\widetilde{+}J), f)   \text{ is isomorphic to } (b\cup a_0[I]\cup a_1[J], \pi)\\
  &   \text{ and hence isomorphic to }  (|b|+2N_\msb, g) \text{ for some } g\in \mcr(b, \msb).
  \end{align*}
  So (Res.1)-(Res.2) are necessary conditions which simply require that $f$ does not conflict the requirements recorded by $\mcr$.

  In what follows, once we define $\pi: [\omega_1]^3\ra 2$, then $\mct, \bfe, \mcr$ are induced from $\mcc$ by (T2), (E1) and (R2) respectively. So we will only describe the procedure of finding $\mcc$.

 We now investigate the preservation of $\psi_2(\mcc, \mct, \bfe, \pi, \mcr)$. First, we apply Lemma \ref{lem 7.po} to find a simple way of forming $\mcr$-satisfiable functions (see also Lemma \cite[Lemma 22]{Peng25}).
  \begin{lem}\label{lem 9.0}
 Assume $\varphi_1(\mcc, \mct, \bfe)$ and $\psi_2^-(\mcc, \mct, \bfe, \pi, \mcr)$. Suppose $(a,\msa)$ is a $\mcc$-candidate with  property {\rm (R1)$^*$} and $f: [|a|+2N_\msa]^3\ra 2$ satisfies {\rm (Res.1)}.
 The following statements are equivalent.
 \begin{enumerate}[(i)]
 \item $f$ is $\mcr$-satisfiable for $(a, \msa)$.
 \item  for every $(b, \msb, I, I')\in \mc{C}_{2, \max}(a, \msa)$, $(a^{-1}[b]\cup (|a|\widetilde{+} I)\cup (|a|+N_\msa\widetilde{+} I', f)$ is isomorphic to $(|b|+2N_\msb, g)$ for some $g\in \mcr(b,\msb)$.
 \end{enumerate}
 \end{lem}
 \begin{proof}
 Only (ii)$\ra$(i) needs a proof. Suppose (ii) holds and we check (Res.2) for $f$.
 
 Fix $(b, \msb, I, I')\in \mc{C}(a, \msa)$ and we check (Res.2) for $(b, \msb)$ and $I, I'$. The case $(b, \msb, I, I')\in \mc{C}_{\max}(a, \msa)$ follows from (ii). So assume $(b, \msb, I, I')\notin \mc{C}_{\max}(a, \msa)$.
 
 By Lemma \ref{lem 7.po} and the fact that $\mc{C}_2(a, \msa)$ is finite, there exists $(b', \msb', J, J') \in \mc{C}_{\max}(a, \msa)$ such that $(b, \msb)\prec (b', \msb')$  and $I\times I'\subseteq J\times J'$. 
 
 By (ii) and (R2),   for some $b_0<b_1$ in $\msb'$,
 \begin{enumerate}
 \item $(a^{-1}[b']\cup (|a|\widetilde{+} J)\cup (|a|+N_\msa\widetilde{+} J', f)$ is isomorphic to $(b'\cup b_0\cup b_1, \pi)$.
 \end{enumerate}
 
 Let $L, L'$ be such that 
 \begin{enumerate}\setcounter{enumi}{1}
 \item $I=J[L]$ and $I'=J'[L']$. 
 \end{enumerate}
 Then 
  \begin{enumerate}\setcounter{enumi}{2}
 \item  $L, L'$ are in $\mc{F}(\msb, \msb')$.
  \end{enumerate}
   To see this, fix $a'\in \msa$. Then $a'[I]\in \msb$ and $a'[J]\in \msb'$. Note $a'[I]=(a'[J])[L]$. By (C2), $c[L]\in \msb$ for some and hence  all $c\in \msb'$. So $L\in \mc{F}(\msb, \msb')$.
 
 By (1) and (2), $(a^{-1}[b]\cup (|a|\widetilde{+} I)\cup (|a|+N_\msa\widetilde{+} I', f)$ is isomorphic to $(b\cup b_0[L]\cup b_1[L'], \pi)$. Then by (3), $b_0[L], b_1[L']$ are in $\msb$.  Together with (R2), (Res.2) holds for $(b, \msb, I, I')$.
 \end{proof}
 
 \textbf{Remark.} By Lemma \ref{lem 7.po}, for $(b, \msb, I, I')\neq (c, \ms{C}, J, J')$ in $\mc{C}_{\max}(\msa)$, $(I\times I')\cap (J\times J')=\emptyset$. So the intersection of their induced domains is
 \begin{align*}
 [a^{-1}[b]\cup (|a|\widetilde{+} I)\cup (|a|+N_\msa\widetilde{+} I')]^3\cap [a^{-1}[c]\cup (|a|\widetilde{+} J)\cup (|a|+N_\msa\widetilde{+} J')]^3\\
 \subseteq \bigcup_{i<2} [|a|\cup [|a|+iN_\msa, |a|+(i+1)N_\msa)]^3.
 \end{align*}
 By (R1)$^*$ and (Res.1), the definition of $f$ on this intersection   is pre-fixed and hence will not conflict each other.

 Then, we introduce an equivalent formulation of (Res) (see also \cite[Lemma 26]{Peng25}).
 
 \begin{lem}\label{lem 9.1}
 Assume $\varphi_1(\mcc, \mct, \bfe)$ and $\psi_2^-(\mcc, \mct, \bfe, \pi, \mcr)$. Then the following statements are equivalent.
 \begin{enumerate}[(i)]
 \item {\rm (Res)}, or equivalently, $\psi_2(\mcc, \mct, \bfe, \pi, \mcr)$.
 \item Suppose $(a,\msa)$ is a $\mcc$-candidate with  property {\rm (R1)$^*$}.
  Suppose $2\leq n<\omega$ and  $f: [|a|+nN_\msa]^3\ra 2 $ is $\mcr$-satisfiable, i.e., $(|a|\cup I\cup J, f)$ is $\mcr$-satisfiable for all $I\neq J$ in $\{[|a|+iN_\msa, |a|+(i+1)N_\msa): i<n\}$.
   Then there are $a_0<\cdots<a_{n-1}$ in $\msa$ such that $(a\cup \bigcup_{i<n} a_i, \pi)$ is isomorphic to $(|a|+nN_\msa, f)$. 
 \end{enumerate}
 \end{lem}
 \begin{proof}
 Only (i)$\ra$(ii) needs a proof. We assume (i) and prove (ii) by induction on $n$. $n=2$ is exactly (Res).
 
 Now suppose (ii) holds for $n=k\geq 2$ and we prove for $n=k+1$.
 
 Fix a $\mcc$-candidate $(a, \msa)$ with property (R1)$^*$  and $\mcr$-satisfiable function $f: [|a|+nN_\msa]^3\ra 2$.
 
 By induction hypothesis for $n=k$, we find $a_{\alpha, 0}<a_{\alpha, 1}<\cdot\cdot\cdot<a_{\alpha, k-1}$ in $\msa\cap [\omega_1\setminus\alpha]^{<\omega}$ for each $\alpha<\omega_1$ such that
\begin{enumerate}
\item $(a\cup \bigcup_{i<k} a_{\alpha, i}, \pi)$ is isomorphic to $(|a|+kN_\msa, f\up [|a|+kN_\msa]^3)$.
\end{enumerate}
For each $\alpha$, let $a_\alpha=\bigcup_{i<k} a_{\alpha, i}$. Then find $\Gamma\in [\omega_1]^{\omega_1}$ such that
\begin{enumerate}\setcounter{enumi}{1}
\item $a_\alpha<a_\beta$ for $\alpha<\beta$ in $\Gamma$.
\end{enumerate}
Then by Lemma \ref{lem 7.l}, omitting a countable part if necessary,
$$(a, \msb)\text{ is a $\mcc$-candidate where }\msb=\{a_\alpha: \alpha\in \Gamma\}.$$
By (1), $(a, \msb)$ satisfies (R1)$^*$.
Then note $N_\msb=kN_\msa$ and
\[\text{for every $b\in \msb$ and $i<k$, }b[[iN_\msa, (i+1)N_\msa)]\in \msa.\]
Now extend $f$ to an $\mcr$-satisfiable for $(a, \msb)$ function $g: [|a|+2N_\msb]^3\ra 2$.
To see the existence of an $\mcr$-satisfiable $g$, choose $\alpha<\beta$ in $\Gamma$. Denote $b=a\cup a_\alpha\cup a_\beta$. Define 
\begin{displaymath}
g(x)=\left\{
\begin{array}{ll}
f(x)  & : ~ x\in [|a|+nN_\msa]^3 \\
\pi(b[x])  & :  ~ \text{otherwise}.
\end{array}
\right.
\end{displaymath}
Then (Res.1) follows from (1). To see(Res.2), fix $(c, \ms{C}, I, J)\in \mcc_2(a, \msb)$. Then for some $i<k$, $J\subseteq [iN_\msa, (i+1)N_\msa)$. So 
\begin{itemize}
\item either $i=0$ and $g$ equals $f$ on $[a^{-1}[c]\cup (|a|\widetilde{+} I)\cup (|a|+N_\msb\widetilde{+} J)]^3$;
\item or $i>0$ and 
$(a^{-1}[c]\cup (|a|\widetilde{+} I)\cup (|a|+N_\msb\widetilde{+} J), g)$ is   isomorphic to $(c\cup a_\alpha[I]\cup a_\beta[J], \pi)$. 
\end{itemize}
Now
(Res.2) follows from (R2) and the fact that $f$ is $\mcr$-satisfiable.


Applying (Res)   to $(a,\msb)$ and $g$, we get $\alpha<\beta$ in $\Gamma$ such that $(a\cup a_\alpha\cup a_\beta, \pi)$ is isomorphic to $(|a|+2N_\msb, g)$. 

By (2),  $a_{\alpha, 0}<\cdot\cdot\cdot<a_{\alpha, k-1}<a_{\beta, 0}$ are in $\msa$. Then by the fact $g$ extends $f$, 
\[(a\cup \bigcup_{i<k} a_{\alpha,i}\cup a_{\beta, 0}, \pi)\text{ is isomorphic to }(|a|+nN_\msa, f).\]
This finishes the induction and the proof of the lemma.
 \end{proof}
 We show that $\pi$ witnesses the failure of $\mc{K}_3$ under appropriate assumptions.
 \begin{lem}\label{lem 9.2}
 Suppose $\varphi_1(\mcc, \mct, \bfe)$, $\psi_2(\mcc, \mct, \bfe, \pi, \mcr)$ and $|\mcc|<\mathfrak{a}_{\omega_1}$. Then $\mc{H}^\pi_0$ is ccc and $\pi$ has no uncountable 0-homogeneous subset.
 \end{lem}
 \begin{proof}
 We first show that $\mc{H}^\pi_0$ is ccc. Fix $\mcx=\{p_\alpha\in \mc{H}^\pi_0: \alpha<\omega_1\}$. We may assume that $\mcx$ forms a $\Delta$-system with root $\overline{p}$.
 
 By Lemma \ref{lem 7.k}, find a $\mcc$-candidate $(a, \msa)$ and $I\subseteq N_\msa$ such that
 \begin{enumerate}
 \item $\overline{p}\subseteq a \text{ and } \{a'[I]: a'\in \msa\}\subseteq \{p_\alpha\setminus \overline{p}: \alpha<\omega_1\}$.
 \end{enumerate}
Replacing $\msa$ by a subset, we may assume that $(a, \msa)$ has property (R1)$^*$.

Define a $\mcr$-satisfiable for $(a, \msa)$ function $f$,  according to Lemma \ref{lem 9.0},   such that

\begin{enumerate}\setcounter{enumi}{1}
\item $f$ satisfies (Res.1) for $(a, \msa)$;
\item for every $(b, \msb, J, J')\in \mc{C}_{2, \max}(a, \msa)$,  
\begin{enumerate}[(\theenumi.\arabic{enumii})]
\item $(a^{-1}[b]\cup (|a|\widetilde{+} J)\cup (|a|+N_\msa\widetilde{+} J', f)$ is isomorphic to $(|b|+2N_\msb, g)$ for some $g\in \mcr(b,\msb)$;
\item $f $ is constant 0 on $[a^{-1}[\overline{p}\cap b]\cup (|a|\widetilde{+} (I\cap J))\cup (|a|+N_\msa\widetilde{+} (I\cap J'))]^3$;
\end{enumerate}
\item $f(x)=0$ if it is not defined according to (2)-(3).
\end{enumerate}
Note that the satisfaction of (3.1)-(3.2)   follows from (R3), (1) and the fact $\mcx\subseteq \mc{H}^\pi_0$.

Then apply (Res) to find $a'<a''$ in $\msa$ such that $(a\cup a'\cup a'', \pi)$ is isomorphic to $(|a|+2N_\msa, f)$. Together with (2)-(4), $\overline{p}\cup a'[I]\cup a''[I]\in \mc{H}^\pi_0$. 

By (1), $\mcx$ contains 2 compatible elements. So $\mc{H}^\pi_0$ is ccc.\medskip

We then show that $\pi$ has no uncountable 0-homogeneous subset. Fix $X\in [\omega_1]^{\omega_1}$. 

By Lemma \ref{lem 7.j},  find $(a, \msa)\in \mcc$, $i<N_\msa$ and $\Gamma\in [\omega_1]^{\omega_1}$ such that $\msa_i[\Gamma]\subseteq X$ and $(a, \msa[\Gamma])$ is a $\mcc$-candidate.

Now it is easy to find an $\mcr$-satisfiable for $(a, \msa[\Gamma])$ function $g: [|a|+3N_\msa]^3\ra 2$, in the sence of Lemma \ref{lem 9.1} (ii), such that
\begin{enumerate}\setcounter{enumi}{4}
\item $g(|a|+i, |a|+N_\msa+i, |a|+2N_\msa+i)=1$.
\end{enumerate}
By Lemma \ref{lem 9.1}, find $a_0<a_1<a_2$ in $\msa[\Gamma]$ such that $(a\cup \bigcup_{j<3} a_j, \pi)$ is isomorphic to $(|a|+3N_\msa, g)$. Then by (5), $\pi(a_0(i), a_1(i), a_2(i))=1$. Since $\{a_0(i), a_1(i), a_2(i)\}\subseteq X$, $X$ is not 0-homogeneous.
 \end{proof}
 
 In the iterated forcing process, we will frequently extend the collection $\mcc$ to contain some $(a, \msa')$ where for some $\mcc$-candidate $(a, \msa)$, $\msa'\in [\msa]^{\omega_1}$ is added generically.  
The following poset will be  frequently used to add such $\msa'$.
 
\begin{defn}\label{defn Paah}
Assume $\varphi_1(\mcc, \mct, \bfe)$ and $\psi_2(\mcc, \mct, \bfe, \pi, \mcr)$. Suppose $(a, \msa)$ is a $\mcc$-candidate satisfying {\rm (R1)$^*$}. Suppose $H$ is a non-empty collection of $\mcr$-satisfiable for $(a, \msa)$ functions. 
Then $\mc{P}_{a, \msa, H}$ is the poset consisting of $p\in [\msa]^{<\omega}$ such that 
\begin{itemize}
\item for every $ a'<a'' $ in $p$, there exists $f\in H$ such that $(a\cup a'\cup a'', \pi)$ is isomorphic to $ (|a|+2N_\msa, f)$.
\end{itemize} 
The order is reverse inclusion.
\end{defn}
\begin{lem}\label{lem 9.3}
Assume $\varphi_1(\mcc, \mct, \bfe)$ and $\psi_2(\mcc, \mct, \bfe, \pi, \mcr)$. For $(a, \msa)$ and $H$ as above, $\mc{P}_{a, \msa, H}$ is  ccc.
\end{lem}
\begin{proof}
Fix $\{p_\alpha\in \mc{P}_{a, \msa, H}: \alpha<\omega_1\}$. Find $\Gamma\in [\omega_1]^{\omega_1}$ such that $\{p_\alpha: \alpha\in \Gamma\}$ forms a $\Delta$-system with root $\overline{p}$. Since the root does not affect compatibility, we may assume that $\overline{p}=\emptyset$.  Find $\Gamma'\in [\Gamma]^{\omega_1}$  and  $n<\omega$ such that
\[|p_\alpha|=n\text{ whenever $\alpha\in \Gamma'$ and $\msa'=\{\bigcup p_\alpha: \alpha\in \Gamma'\}$ is non-overlapping}.\]
Now by Lemma \ref{lem 7.l}, find $\Sigma\in [\Gamma']^{\omega_1}$ such that
\[(a, \msb)\text{ is a $\mcc$-candidate satisfying (R1)$^*$ where }\msb=\{\bigcup p_\alpha: \alpha\in \Sigma\}.\]

Choose $f: [|a|+2nN_\msa]^3\ra 2$ such that
\begin{enumerate}
\item $f$ satisfies (Res.1) for $(a, \msb)$;
\item  for every $i, j<n$, $(|a|\cup I^*_i\cup I^*_{n+j}, f)$ is isomorphic to $(|a|+2N_\msa, g)$ for some $g\in H$ where $I^*_k=[|a|+kN_\msa, |a|+(k+1)N_\msa)$ for $k<2n$.
\end{enumerate}
It is straightforward to check that $f$ is $\mcr$-satisfiable for $(a, \msb)$.

Now applying (Res) to $(a, \msb)$ and $f$, we find $\alpha<\beta$ in $\Sigma$ such that $(a\cup a_\alpha\cup a_\beta, \pi)$ is isomorphic to $(|a|+2N_\msb, f)$. Then $p_\alpha$ is compatible with $p_\beta$. So $\mc{P}_{a, \msa, H}$ is  ccc.
\end{proof}

 We now investigate preservation of $\psi_2(\mcc, \mct, \bfe, \pi, \mcr)$ (or $\psi_2(\mcc', \mct', \bfe', \pi, \mcr')$ with extended $\mcc'$ and induced $\mct', \bfe', \mcr'$) in the iteration process.
 \begin{lem}\label{lem 9.4}
 Assume $\varphi_1(\mcc, \mct, \bfe)$ and $\psi_2(\mcc, \mct, \bfe, \pi, \mcr)$. Suppose $\mc{P}$ is productively ccc. Then $\Vdash_\mc{P} \psi_2(\mcc, \mct, \bfe, \pi, \mcr)$. 
 \end{lem}
 \begin{proof}
 Since $\mc{P}$ preserves $\omega_1$,  $\mc{P}$ preserves $\varphi_1(\mcc, \mct, \bfe)$ and $\psi_2^-(\mcc, \mct, \bfe, \pi, \mcr)$. Now it suffices to check (Res). Fix a $\mc{P}$-name of a $\mcc$-candidate $(a, \dot{\msa})$ with property (R1)$^*$. Fix $p\in \mc{P}$, $N$  and $f: [|a|+2N]^3\ra 2$ satisfying the following properties.
 \begin{enumerate}
 \item $p\Vdash N_{\dot{\msa}}=N$ and $f$ is $\mcr$-satisfiable for $(a, \dot{\msa})$.
 \end{enumerate}
 
 For each $\alpha<\omega_1$, find $p_\alpha\leq p$ and $a_\alpha$ such that
 \[p_\alpha\Vdash a_\alpha \text{ is the $\alpha$th element of } \dot{\msa}.\]
 By Lemma \ref{lem 7.n}, find $\Gamma\in [\omega_1]^{\omega_1}$ such that 
 \[(a, \msa')\text{ is a $\mcc$-candidate satisfying (R1)$^*$ where }\msa'=\{a_\alpha: \alpha\in \Gamma\}.\]
 
 By Lemma \ref{lem 9.3}, $\mc{P}_{a, \msa', \{f\}}$ is ccc. Hence, $\mc{P}\times \mc{P}_{a, \msa', \{f\}}$ is ccc. Find $\alpha<\beta$ in $\Gamma$ such that $(p_\alpha, \{a_\alpha\})$ is compatible with $(p_\beta, \{a_\beta\})$. Let $q\leq p_\alpha, p_\beta$. Then
 \[q\Vdash a_\alpha< a_\beta \text{ are in }\dot{\msa} \text{ and } (a\cup a_\alpha\cup a_\beta, \pi)\text{ is isomorphic to } (|a|+2N, f).\]
 
 Now a density argument shows that $\mc{P}$ preserves $\psi_2(\mcc, \mct, \bfe, \pi, \mcr)$.
 \end{proof}

\begin{lem}\label{lem 9.5}
Assume $\varphi_1(\mcc, \mct, \bfe)$ and $\psi_2(\mcc, \mct, \bfe, \pi, \mcr)$. Suppose $(a, \msa), H$ and $ \mc{P}_{a, \msa, H}$ are as in Definition \ref{defn Paah}.  Suppose moreover that $H$ has the following property.
\begin{enumerate}[(i)]
\item For $a'\in \msa$, $a^*\subseteq a$ and $I, J\subseteq N_\msa$,
if $a^*\cup a'[I]\in \mc{H}^\pi_0$, $ a^*\cup a'[J]\in \mc{H}^\pi_0$,
 then there exists $f\in H$ that is constant 0 on   
$[a^{-1}[a^*]\cup (|a|\widetilde{+} I)\cup (|a|+N_\msa\widetilde{+} J)]^3$.
\end{enumerate}
 If $G$ is an uncountable generic filter,  then $\varphi_1(\mcc', \mct', \bfe')$ and $\psi_2(\mcc', \mct', \bfe', \pi, \mcr')$ hold  in $V[G]$ where $\mcc'=\mcc\cup \{(a, \bigcup G)\}$ and $\mct', \bfe', \mcr'$ are induced from $\mcc'$. Moreover, $\mcr'(a, \bigcup G)=H$.
\end{lem}
\begin{proof}
By $\bigcup G\subseteq \msa$, $\varphi_1(\mcc', \mct', \bfe')$ holds in $V[G]$.

To see $\psi_2(\mcc', \mct', \bfe', \pi, \mcr')$, first note that (R1) of $\mcc'$ follows from (R1) of $\mcc$ and (R1)$^*$ of $(a, \msa)$. Also, (R2) is automatically satisfied. 

Then we prove the moreover part. And
(R3) will follow from the additional property that $H$ satisfies.
Fix $p\in  G$ and $f\in H$.
Recall that $G$ is uncountable. 

Now work in $V$. Note that $\msa'=\{a'\in \msa: p\cup \{a'\}\in \mc{P}_{a, \msa, H}\}$ is uncountable.
 Applying (Res) to $(a, \msa')$ and $f$, we find $a'<a''$ in $\msa'$ such that \[(a\cup a'\cup a'', \pi)\text{ is isomorphic to }(|a|+2N_\msa, f).\]
  Then $p\cup \{a', a''\}$ is an extension of $p$ which forces $f\in \dot{\mcr}'(a, \bigcup \dot{G})$. 
 
 Now a density argument shows that $\mcr'(a, \bigcup G)=H$ in $V[G]$.

Finally we check (Res). 
Fix $p\in G$, $(b, \dot{\msb})$, $N$ and  $f: [|b|+2N]^3\ra 2$ such that
\[p\Vdash (b, \dot{\msb})\text{ is a $\dot{\mcc}'$-candidate, $N=N_{\dot{\msb}}$ and $f$ is $\dot{\mcr}'$-satisfiable for } (b, \dot{\msb}).\]
Extending $p$ if necessary, we may assume that 
\begin{enumerate}
\item $p$ determines $\dot{\mc{C}}'_1(b, \dot{\msb})$.
\end{enumerate}

For every $\alpha<\omega_1$, choose $p_\alpha\leq p$ and $b_\alpha\in [\omega_1]^N$ such that
\[p_\alpha\Vdash b_\alpha \text{ is the $\alpha$th element of } \dot{\msb}.\]
Extending $p_\alpha$ if necessary, we may assume that
\begin{enumerate}\setcounter{enumi}{1}
\item if $p\Vdash (a, \bigcup \dot{G})\prec (b, \dot{\msb})$, then $b_\alpha[J]\in p_\alpha$ for every $J\in \mc{F}(\bigcup \dot{G}, \dot{\msb})$.
\end{enumerate} 

Now choose $\Gamma\in [\omega_1]^{\omega_1}$ and $I^*$ such that,
\begin{enumerate}\setcounter{enumi}{2}
\item if $p\Vdash (a, \bigcup \dot{G})\not\prec (b, \dot{\msb})$, then for every $\alpha\in \Gamma$, $b_\alpha\cap (\bigcup p_\alpha)=\emptyset$;
\item $|p_\alpha|$ is constant for $\alpha\in \Gamma$ and $\{p_\alpha: \alpha\in \Gamma\}$ is a $\Delta$-system with root $\overline{p}$;
\item  for every $\alpha\in \Gamma$, $c_\alpha[I^*]=b_\alpha$ where $c_\alpha=b_\alpha\cup \bigcup (p_\alpha\setminus \overline{p})$;
\item $\{c_\alpha: \alpha\in \Gamma\}$ is non-overlapping.
\end{enumerate}
Note by (1)-(3), for $\alpha\in \Gamma$,
\begin{enumerate}\setcounter{enumi}{6}
\item    for $I\subseteq |c_\alpha|$ with $c_\alpha[I]\in p_\alpha$, either $I\subseteq I^*$, or $I\cap I^*=\emptyset$. Moreover, $\{I\subseteq |c_\alpha|: c_\alpha[I]\in p_\alpha\wedge I\nsubseteq I^*\}$ is a partition of $|c_\alpha|\setminus I^*$.
\end{enumerate}

By Lemma \ref{lem 7.l}, find $\Gamma'\in [\Gamma]^{\omega_1}$ that
\begin{enumerate}\setcounter{enumi}{7}
\item     $(a, \msa')$ is a $\mcc$-candidate where $\msa'=\{\bigcup p_\alpha\setminus \overline{p}: \alpha\in \Gamma'\}$.
\end{enumerate}
 
 Applying Lemma \ref{lem 7.o} to find  $\Gamma''\in [\Gamma']^{\omega_1}$ such that
 \begin{enumerate}\setcounter{enumi}{8}
\item $(b, \msb')$ is a $\mcc$-candidate where $\msb'=\{b_\alpha: \alpha\in \Gamma''\}$.
\end{enumerate}

By (7)-(9)  and Lemma \ref{lem 7.l}, find $\Sigma \in [\Gamma'']^{\omega_1}$ such  that 
\begin{enumerate}\setcounter{enumi}{9}
\item $(c, \ms{C})$ is a $\mcc$-candidate satisfying (R1)$^*$ where $c=a\cup b$ and $\ms{C}=\{c_\alpha: \alpha\in \Sigma\}$.
\end{enumerate}
 Fix ${\xi^*}\in \Sigma$. Let  $I_a, I_b\subseteq |c|$ satisfy 
\[a=c[I_a]\text{ and } b=c[I_b].\]
Then find $h: [|c|+2N_\ms{C}]^3\ra 2$ such that
\begin{enumerate}\setcounter{enumi}{10}
\item $h$ satisfy (Res.1) for $(c, \ms{C})$; 
\item $(I_b\cup (|c|\widetilde{+} I^*) \cup (|c|+N_\ms{C}\widetilde{+} I^*, h)$ is isomorphic to $(|b|+2N, f)$;
\item for $I, I'\subseteq N_\ms{C}$ with $\{c_{\xi^*}[I], c_{\xi^*}[I']\}\subseteq p_{\xi^*}$, $(I_a\cup (|c|\widetilde{+} I)\cup (|c|+N_\ms{C}\widetilde{+} I'), h)$ is isomorphic to $(|a|+2N_\msa, g)$ for some $g\in H$;
\item for $(d, \ms{D}, I, I')\in \mc{C}_{\max}(c, \ms{C})$, $(c^{-1}[d]\cup (|c|\widetilde{+} I)\cup (|c|+ N_\ms{C}\widetilde{+} I'), h)$ is isomorphic to $(|d|+2N_\ms{D}, g)$ for some $g\in \mcr(d, \ms{D})$.
\end{enumerate}
We check that (11)-(14) can be satisfied simultaneously.  

First by (10), define  $h$ partially on $D^*$ so that (11) is satisfied where
\[D^*=\bigcup_{i<2}[|a|\cup [|a|+iN_\ms{C}, |a|+(i+1)N_\ms{C})]^3.\]

Then define $h$ on 
\[[I_b\cup (|c|\widetilde{+} I^*) \cup (|c|+N_\ms{C}\widetilde{+} I^*)]^3\]
so that (12) is satisfied. 

For (13), note by (7), 
\begin{enumerate}\setcounter{enumi}{14}
\item elements in $\mcx=\{I\times I': c_{\xi^*}[I]\in p_{\xi^*}, c_{\xi^*}[I']\in p_{\xi^*}\}$
 are pairwise disjoint and each $I\times I'\in \mcx$ is either contained in or disjoint from $I^*\times I^*$.
 \end{enumerate}
  So define $h$ on
 \[\bigcup\{[I_a\cup (|c|\widetilde{+} I)\cup (|c|+N_\ms{C}\widetilde{+} I']^3: I\times I'\in \mcx, I\times I'\cap I^*\times I^*=\emptyset\}\]
to satisfy (13) for corresponding $I, I'$'s. 

We then check (13) for   $I\times I'\in \mcx$ such that $I\times I'\cap I^*\times I^*\neq \emptyset$. By (15), $I\times I'\subseteq I^*\times I^*$.

Then by (3), $p\Vdash (a, \bigcup \dot{G})\prec (b, \dot{\msb})$. Hence $a\subseteq b$.  Then by (12), $(I_a\cup (|c|\widetilde{+} I)\cup (|c|+N_\ms{C}\widetilde{+} I'), h)$ is isomorphic to $(b^{-1}[a]\cup (|b|\widetilde{+} J)\cup (|b|+N\widetilde{+} J'), f)$ where $J, J'\in \mc{F}(\bigcup \dot{G}, \dot{\msb})$ are such that $I=I^*[J], I'=I^*[J']$.

Now (13) for $I, I'$ follows from  the fact that $\mcr'(b, \bigcup G)=H$ and $f$ is  $\mcr'$-satisfiable for $(b, \dot{\msb}^G)$ in $V[G]$.

Finally we define $h$ to satisfy (14). First note by Lemma \ref{lem 7.po},  elements in
\[\{[c^{-1}[d]\cup (|c|\widetilde{+} I)\cup (|c|+ N_\ms{C}\widetilde{+} I')]^3\setminus D^*: (d, \ms{D}, I, I')\in\mc{C}_{\max}(c, \ms{C})\}\]
are pairwise disjoint.
  So it suffices to define $h$ on corresponding domain induced by each $(d, \ms{D}, I, I')\in\mc{C}_{\max}(c, \ms{C})$ to satisfy (14).

Fix $(d, \ms{D}, I, I')\in\mc{C}_{\max}(c, \ms{C})$. 

First consider the case $(I\times I')\cap (I^*\times I^*)\neq\emptyset$. Note $\ms{D}_i\cap \msb'_j\supseteq \ms{C}_k$ if $k=I(i)=I^*(j)\in I\cap I^*$.  Hence, $(d, \ms{D})\prec (b, \msb')$ and $I\times I'\subseteq I^*\times I^*$. Then (14) for $(d,\ms{D}, I, I')$ follows from (12).

The case $(I\times I')\cap (J\times J')\neq\emptyset$ for some $J\times J'\in \mcx$ follows from a similar argument. Now it is easy to define $h$ for the other cases so that (14) holds.

Now (11)-(14) are all satisfied and hence $h$ is $\mcr$-satisfiable for $(c, \ms{C})$. Apply (Res) to $(c, \ms{C})$ and $h$ to obtain $\alpha<\beta $ in $\Sigma$ such that $(c\cup c_\alpha\cup c_\beta, \pi)$ is isomorphic to $(|c|+2N_\ms{C}, h)$.

By (13), $p_\alpha$ is compatible with $p_\beta$. By (12),
\[p_\alpha\cup p_\beta\Vdash b_\alpha, b_\beta\in \dot{\msb}\text{ and witness (Res) for } (b, \dot{\msb}) \text{ and } f.\] 
Then (Res) follows from a density argument. 
\end{proof}

We show the preservation of $\psi_2(\mcc, \mct, \bfe, \pi, \mcr)$ at limit stages of finite support iteration of ccc posets.
  \begin{lem}\label{lem 9.6}
Suppose $\nu$ is a limit ordinal and $\langle \mc{P}_\alpha, \dot{\mc{Q}}_\beta: \alpha\leq \nu, \beta<\nu\rangle$ is a finite support iteration of ccc posets. Moreover, for $\alpha\leq \nu$,
\begin{itemize}
\item $(\mcc^\alpha, \mct^\alpha, \bfe^\alpha, \pi, \mcr^\alpha)\in V^{\mc{P}_\alpha}$;
\item  $\mcc^\xi\subseteq \mcc^\alpha$ for $\xi<\alpha$ and $\mcc^\nu=\bigcup_{\xi<\nu} \mcc^\xi$;
\item  $\mct^\alpha, \bfe^\alpha$ and $\mcr^\alpha$  are induced from $\mcc^\alpha$.
\end{itemize}
If $\varphi_1(\mcc^\alpha, \mct^\alpha, \bfe^\alpha)$ and $\psi_2(\mcc^\alpha, \mct^\alpha, \bfe^\alpha, \pi, \mcr^\alpha)$ hold in $V^{\mc{P}_\alpha}$ for all $\alpha<\nu$, then $\psi_2(\mcc^\nu, \mct^\nu,\bfe^\nu, \pi, \mcr^\nu)$ holds in $V^{\mc{P}_\nu}$.
\end{lem}

\begin{proof}
By Lemma \ref{lem 7.m}, $\varphi_1(\mcr^\nu, \mct^\nu, \bfe^\nu)$ holds. It suffices to prove   (Res)   in  $V^{\mc{P}_\nu}$.

Fix  in  $V^{\mc{P}_\nu}$, a $\mcc^\nu$-candidate $(a, \msa)$ satisfying (R1)$^*$ and an $\mcr^\nu$-satisfiable for $(a, \msa)$ function $f$. Viewing some $V^{\mc{P}_\alpha}$ as the ground model if necessary, assume 
\begin{enumerate}
\item all $(b, \msb)\prec (a, \msa)$ in $\mcc^\nu$ are in $\mcc^0$.
\end{enumerate}

Fix $p\in \mc{P}_\nu$, $N<\omega$ and a $\mc{P}_\nu$-name $\dot{\msa}$ of $\msa$ such that 
\begin{enumerate}\setcounter{enumi}{1}
\item $p$ forces (1) and determines all $(b, \msb)\prec (a, \dot{\msa})$ in $\mcc^0$;
\item $p\Vdash N=N_{\dot{\msa}}\text{ and } f \text{ is } \dot{\mcr}^\nu \text{ satisfiable for } (a, \dot{\msa})$.
\end{enumerate}
For every $\alpha<\omega_1$, find $p_\alpha\leq p$ and $a_\alpha\in [\omega_1]^N$ such that
\[p_\alpha\Vdash a_\alpha\text{ is the $\alpha$th element of } \dot{\msa}.\]
Find $\Gamma\in [\omega_1]^{\omega_1}$ such that 
\begin{enumerate}\setcounter{enumi}{3}
\item $\{supp(p_\alpha): \alpha\in \Gamma\}$ forms a $\Delta$-system with root contained in $\xi'<\nu$.
\end{enumerate}

By Lemma \ref{lem 7.p}, find $\xi\in [\xi', \nu)$, $q\leq p$ and  $\Sigma\in [\Gamma]^{\omega_1}$ such that for the canonical name  of the $\mc{P}_\nu$-generic filter $\dot{G}$,
\begin{enumerate}\setcounter{enumi}{4}
\item $q\up\xi\Vdash_{\mc{P}_\xi} (a, \{a_\alpha: \alpha\in \Sigma, p_\alpha\up \xi \in \dot{G}\up\xi\})$ is a $\dot{\mcc}^\xi$-candidate.
\end{enumerate}

Let $G$ be a $\mc{P}_\nu$-generic filter containing $q$, 
\[\Sigma'=\{\alpha\in \Sigma: p_\alpha\up\xi\in G\up\xi\}\text{ and } \msa'=\{a_\alpha: \alpha\in \Sigma'\}.\]

Work in $V[G\up\xi]$.  By (5),  $(a, \msa')$ is a $\mcc^\xi$-candidate. By (2) and (3), $f$ is $\mcc^\xi$-satisfiable for $(a, \msa')$. Applying $\psi_2(\mcc^\xi, \mct^ \xi, \bfe^ \xi, \pi, \mcr^ \xi)$, we get $a_{\alpha}<a_{\beta}$ in $\msa'$ witnessing (Res) for $(a, \msa')$ and $f$.

By (4) and the fact $\Sigma'\subseteq \Gamma$, $p_{\alpha}, p_{\beta}$ have a common lower bound, say $p'$. Now 
\[p'\Vdash a_{\alpha}<a_{\beta} \text{   witness (Res) for $(a, \dot{\msa})$ and } f.\]

Now a density argument shows that (Res)  holds in $V^{\mc{P}_\nu}$.
\end{proof}

 \subsection{Forcing at an intermediate stage}
 Assume $\varphi_1(\mcc, \mct, \bfe)$, $\psi_2(\mcc, \mct,\bfe, \pi, \mcr)$ and $|\mcc|<\mathfrak{a}_{\omega_1}$.
We   describe the procedure of dealing with a ccc poset at some intermediate stage. Fix a ccc poset $\mathbb{P}$ of size $\leq \omega_1$. Note that $\mathbb{P}$ preserves $\omega_1$ and hence $\varphi_1(\mcc, \mct, \bfe)$. Our goal is that one of the following outcomes occurs.
\begin{itemize}
\item $\mathbb{P}$ preserves $\psi_2(\mcc, \mct, \bfe, \pi, \mcr)$ and we force with $\mathbb{P}$.
\item Destroy ccc of $\mathbb{P}$.
\item Destroy 2-ccc of $\mathbb{P}$, i.e., add an uncountable antichain to $\mc{H}_{\mathbb{P}, 2}$.
\end{itemize}

Note by Lemma \ref{lem 9.4}, the first outcome will occur if $|\mathbb{P}|=\omega$. So we assume that $\mathbb{P}$ has size $\omega_1$ in Alternatives 2-3 below. We describe our treatment by cases.\bigskip

\textbf{Alternative 1.} $\mathbb{P}$ preserves $\psi_2(\mcc, \mct, \bfe, \pi, \mcr)$.\medskip

We force with $\mathbb{P}$. So the first outcome occurs.\bigskip

\textbf{Alternative 2.} $\mathbb{P}$ destroys  $\psi_2(\mcc, \mct, \bfe, \pi, \mcr)$ and does not add an uncountable 0-homogeneous subset of $\pi$.\medskip

We will need \cite[Lemma 35]{Peng25}. For completeness, we sketch a proof.

  \begin{lem}[\cite{Peng25}]\label{lem P35}
  Suppose $\msa\subset [\omega_1]^{<\omega}$ is uncountable non-overlapping, $2\leq n<\omega$, $\tau: [\omega_1]^n\ra 2$ is a coloring. Then the following statements are equivalent.
  \begin{enumerate}[(i)]
  \item For some $\delta<\omega_1$, for every finite $\mcx\subset \msa\cap [\omega_1\setminus \delta]^{<\omega}$,  there exists  $b\in \mc{H}_0^ \tau $ such that $b\cap a\neq \emptyset$   whenever $a\in \mcx$.
  \item  There is a 0-homogeneous set meeting all but countably many elements of $\msa$.
  \end{enumerate}
  \end{lem}
  \begin{proof}
  Only ``(i)$\Rightarrow$(ii)'' needs a proof. Fix $\delta<\omega_1$ guaranteed by (i).
  
  Fix a non-principal ultrafilter $\mc{U}$ on $\omega$ and a uniform ultrafilter $\mc{V}$ on $\omega_1$.
  
  We first find, for every countable $\mcx\subseteq \msa\cap [\omega_1\setminus \delta]^{<\omega}$, a 0-homogeneous set meeting all elements of $\mcx$: Say $\mcx=\{a_m: m<\omega\}$. For every $i<\omega$, fix $b_i\in \mc{H}_0^ \tau $ meeting all $a_j$ for $j<i$.   
  For $m<\omega$, choose $x_m\in a_m$ such that $\{i: x_m\in b_i\}\in \mc{U}$.   Then $\{x_m: m<\omega\}$ is 0-homogeneous and meeting all elements of $\mcx$.
  
  Now for every $\alpha<\omega_1$, fix a 0-homogeneous set $B_\alpha$ meeting all elements of $\msa\cap [\alpha\setminus \delta]^{<\omega}$.  
    For $a\in \msa\cap [\omega_1\setminus \delta]^{<\omega}$, choose $y_a\in a$ such that $\{\alpha: y_a\in B_\alpha\}\in \mc{V}$.  Then $\{y_a: a\in \msa\cap [\omega_1\setminus \delta]^{<\omega}\}$ is a 0-homogeneous set meeting all elements of $\msa\cap [\omega_1\setminus \delta]^{<\omega}$.
  \end{proof}

  Note that $\varphi_1(\mcc, \mct, \bfe)$ and $\psi_2^-(\mcc, \mct, \bfe, \pi, \mcr)$ are preserved. Hence $\mathbb{P}$ forces the failure of (Res). Fix $p\in \mathbb{P}$, $a\in [\omega_1]^{<\omega}$, a $\mathbb{P}$-name $\dot{\msa}$, $N$ and $f$ such that
\begin{align*}
p\Vdash & (a, \dot{\msa})\text{ is a $\mcc$-candidate satisfying (R1)$^*$, $f$ is $\mcr$-satisfiable for }(a, \dot{\msa}),\\
& N_{\dot{\msa}}=N \text{ and (Res) fails for $(a, \dot{\msa})$ and } f.
\end{align*}
  
  Recall that $p\Vdash $ Lemma \ref{lem P35} (ii) fails for $\dot{\msa}$. 
So by Lemma \ref{lem P35}, for every $\alpha<\omega_1$, there are $p_\alpha\leq p$, $m_\alpha<\omega$ and $a_{\alpha,0}<\cdots<a_{\alpha, m_\alpha-1}$ in $[\omega_1\setminus \alpha]^{<\omega}$ such that $p_\alpha \Vdash \{a_{\alpha, i}: i<m_\alpha\}\subseteq \dot{\msa}$ and
  \begin{align*}
  \text{for every $b\in \mc{H}^\pi_0$, there exists $i<m_\alpha$ with } b\cap a_{\alpha, i}=\emptyset.
  \end{align*}

Find $\Gamma\in [\omega_1]^{\omega_1}$ and $m<\omega$ such that
\begin{enumerate}[({Alt2.}1)]
\item $m_\alpha=m$ for all $\alpha\in \Gamma$ and $(a, \msa')$ is a $\mcc$-candidate satisfying (R1)$^*$ where $\msa'=\{a_\alpha: \alpha\in \Gamma\}$ and $a_\alpha=\bigcup_{i<m} a_{\alpha, i}$.
\end{enumerate}
To get a $\mcc$-candidate $(a, \msa')$, first inductively apply Lemma \ref{lem 7.n} to get $\Gamma'$ such that $(a, \{a_{\alpha, i}: \alpha\in \Gamma'\})$ is a $\mcc$-candidate for every $i<m$. Then apply Lemma \ref{lem 7.l} to get the desired $\Gamma\in [\Gamma']^{\omega_1}$.

Let $H$ be the collection of all $g$ such that $g$ is $\mcr$-satisfiable for $(a, \msa')$ and
\begin{align*}
\text{for some }i<m,~ (|a|\cup (|a|\widetilde{+} I^*_i) \cup (|a|\widetilde{+} I^*_{m+i}), g)\text{ is isomorphic to }(|a|+2N, f)
\end{align*}
where $I^*_i=[iN, (i+1)N)$.

We show that $H$ satisfies the additional property posted in Lemma \ref{lem 9.5}.
\begin{lem}\label{lem 9.8}
For $a'\in \msa'$, $a^*\subseteq a$ and $I, J\subseteq mN$,
if $a^*\cup a'[I]\in \mc{H}^\pi_0$ and $ a^*\cup a'[J]\in \mc{H}^\pi_0$,
 then there exists $g\in H$ that is constant 0 on   
$[a^{-1}[a^*]\cup (|a|\widetilde{+} I)\cup (|a|+mN\widetilde{+} J)]^3$.
\end{lem}
\begin{proof}
Denote $I^*_j=[jN, (j+1)N)$ for $j<2m$.
By (Alt2.1), fix $i<m$ with $I\cap I^*_i=\emptyset$.

First define a partial $g: [|a|+2mN]^3\ra 2$ such that 
\begin{enumerate}
\item (Res.1) holds for $(a, \msa')$ and $(|a|\cup (|a|\widetilde{+} I^*_i) \cup (|a|\widetilde{+} I^*_{m+i}), g)$ is isomorphic to $(|a|+2N, f)$.
\end{enumerate}

Then extend $g$ such that by (R3), for every $(b, \msb, I', J')\in \mc{C}_{\max}(a,\msa')$ with $(I'\times J')\cap (I\times J)\neq\emptyset$, 
\begin{enumerate}\setcounter{enumi}{1}
\item $(a^{-1}[b]\cup (|a|\widetilde{+} I')\cup (|a|+mN\widetilde{+} J'), g)$ is isomorphic to $(|b|+2N_\msb, h)$ for some $h\in \mcr(b, \msb)$;
\item $g$ is constant 0 on $[a^{-1}[a^*\cap b]\cup (|a|\widetilde{+} (I\cap I'))\cup (|a|+mN\widetilde{+} (J\cap J'))]^3$.
\end{enumerate}
Note that (1)-(3) can be satisfied simultaneously since if $I'\cap I\neq\emptyset$, then $I'\not\subseteq I^*_i$ and hence $I'\cap I^*_i=\emptyset$.

Extend $g$ such that (2) holds for all $(b, \msb, I', J')\in \mc{C}_{\max}(a,\msa')$.
Define $g$ to be 0 on the undefined domain. 

Then $g$ is $\mcr$-satisfiable for $(a, \msa')$ by Lemma \ref{lem 9.0}. $g\in H$ by (1). And $g$  is constant 0 on $[a^{-1}[a^*]\cup (|a|\widetilde{+} I)\cup (|a|+mN\widetilde{+} J)]^3$.
\end{proof}

Now we force with $\mc{P}_{a, \msa', H}$. Assume $G$ is an uncountable generic filter. Let $\mcc'=\mcc\cup\{(a, \bigcup G)\}$ and $\mct', \bfe', \mcr'$ are induced from $\mcc'$. 

Then $\varphi_1(\mcc', \mct', \bfe')$ holds since $(a, \msa')$  is a $\mcc$-candidate. $\psi_2(\mcc', \mct', \bfe', \pi, \mcr')$ holds by Lemma \ref{lem 9.5}. Also, by our choice of $H$, $\{p_\alpha: a_\alpha\in \bigcup G\}$ is an uncountable antichain of $\mathbb{P}$. So the second outcome occurs.

This finishes our treatment for Alternative 2.\bigskip

\textbf{Alternative 3.} $\Vdash_\mathbb{P} \pi$ has an uncountable 0-homogeneous subset.\medskip

By Lemma \ref{lem 7.j} and the fact that $|\mcc|<\mathfrak{a}_{\omega_1}$, find $p\in \mathbb{P}$, $(a, \msa)\in \mcc$, $l^*<N_\msa$ and a $\mathbb{P}$-name $\dot{\Gamma}$ in $[\omega_1]^{\omega_1}$ such that
\begin{enumerate}[({Alt3.}1)]
\item $p\Vdash (a, \msa[\dot{\Gamma}])$ is a $\mcc$-candidate and $\msa_{l^*}[\dot{\Gamma}]$ is 0-homogeneous.
\end{enumerate}
It is straightforward to check that 
\[p\Vdash \exists (a, \msa')\in \mcc ~\mc{C}_{0,\max}(a, \msa[\dot{\Gamma}])=\{(a, \msa')\}.\]
 Without loss of generality, we may assume that
\begin{enumerate}[({Alt3.}1)]\setcounter{enumi}{1}
\item $p\Vdash  \mc{C}_{0,\max}(a, \msa[\dot{\Gamma}])=\{(a, \msa)\}$.
\end{enumerate}

For each $\alpha<\omega_1$, find $p_\alpha\leq p$ and $a_\alpha$ such that 
\[p_\alpha\Vdash a_\alpha\in \msa[\dot{\Gamma}]\cap [\omega_1\setminus \alpha]^{N_\msa}.\]
By Lemma \ref{lem 7.n}, find $\Sigma\in [\omega_1]^{\omega_1}$ such that 
\begin{enumerate}[({Alt3.}1)]\setcounter{enumi}{2}
\item $a_\alpha<a_\beta$ for all $\alpha<\beta$ in $\Sigma$ and $(a, \msa')$ is a $\mcc$-candidate satisfying (R1)$^*$ where $\msa'=\{a_\alpha: \alpha\in \Sigma\}$.
\end{enumerate}
Note by (Alt3.2) and (Alt3.3), 
\[\mc{C}_{0,\max}(a, \msa')=\{(a, \msa)\}.\]
 To see this, suppose some $(a, \msa'')\succ (a, \msa)$ is in $\mcc $. Since $\mathbb{P}$ is ccc, for some $\alpha<\omega_1$, $p\Vdash \msa[\dot{\Gamma}]\cap \msa''\subseteq [\alpha]^{N_\msa}$. Then $\msa'\cap \msa''$ is countable. Together with the fact $\msa'\subseteq \msa$ and $\msa''\subseteq \msa$, $(\bigcup \msa')\cap (\bigcup \msa'')$ is countable and hence $(a, \msa'')\not\in \mc{C}_0(a, \msa')$.

\bigskip

\textbf{Subalternative 3.1.}  For every $\alpha<\beta$ in $\Sigma$, 
\begin{align*}
\{\gamma\in \Sigma\setminus (\beta+1): ~& \pi(a_\alpha(l^*), a_\beta(l^*), a_\gamma(l^*))=1 \text{ and}\\
& p_\gamma \text{ is compatible with both $p_\alpha$ and }p_\beta\}
\end{align*}
 is at most countable.\medskip
 
 Find $\Sigma'\in [\Sigma]^{\omega_1}$ such that 
 \begin{itemize}
 \item for every $\alpha<\beta<\gamma$ in $\Sigma'$, if $\pi(a_\alpha(l^*), a_\beta(l^*), a_\gamma(l^*))=1$, then either $p_\gamma$ is incompatible with $p_\alpha$ or $p_\gamma$ is incompatible with $p_\beta$.
 \end{itemize}
 
 Then we use a canonical poset to destroy ccc of $\mathbb{P}$ and show that the poset is ccc and preserves $\psi_2(\mcc, \mct, \bfe, \pi, \mcr)$.
 Let 
 \[\mc{Q}\text{ be the poset consisting of $q\in [\Sigma']^{<\omega}$ such that $\{p_\xi: \xi\in q\}$ is a $\mathbb{P}$-antichain}.\]
  The order is reverse inclusion.
 
 \begin{lem}\label{lem 9.9}
 $\mc{Q}$ is ccc and preserves $\psi_2(\mcc, \mct, \bfe, \pi, \mcr)$.
 \end{lem}
 \begin{proof}
 We will show that $\mc{Q}$ preserves $\psi_2(\mcc, \mct, \bfe, \pi, \mcr)$ and the argument shows ccc of $\mc{Q}$. Note only (Res) needs a verification.
 
 Fix $q\in \mc{Q}$, $b$, $N$, a $\mc{Q}$-name $\dot{\msb}$ and $f: [|b|+2N]^3\ra 2$ such that
\begin{align*}
 q\Vdash & (b, \dot{\msb})\text{ is a $\mcc$-candidate satisfying (R1)$^*$, $N=N_{\dot{\msb}}$ and}\\
 & f\text{ is $\mcr$-satisfiable for }(b, \dot{\msb}).
 \end{align*}
Moreover, assume $q$ determined $\mc{C}_{2, \max}(b, \dot{\msb})$.

For every $\alpha<\omega_1$, find $q_\alpha\leq q$ and $b_\alpha$ such that
\[q_\alpha\Vdash b_\alpha\in \dot{\msb}\cap [\omega_1\setminus \alpha]^N.\]

Find $X\in [\omega_1]^{\omega_1}$ and $n$ such that
\begin{enumerate}
\item $\{q_\alpha: \alpha\in X\}$ forms a $\Delta$-system with root $\overline{q}$ and $|q_\alpha\setminus \overline{q}|=n$  for $\alpha\in X$;
\item $(b, \msb')$ is a $\mcc$-candidate where $\msb'=\{b_\alpha: \alpha\in X\}$;
\item for $\alpha<\beta$ in $X$, $c'_\alpha<c'_\beta$ where $c'_\alpha=\bigcup \{a_\xi: \xi\in q_\alpha\setminus \overline{q}\}$.
\end{enumerate}
Note by (1), (3) and (Alt3.3), 
\begin{enumerate}\setcounter{enumi}{3}
\item $(a, \ms{C}')$ is a $\mcc$-candidate where $\ms{C}'=\{c'_\alpha: \alpha\in X\}$.
\end{enumerate}
By (Alt3.2) and (Alt3.3), $\mc{C}_{0,\max}(a, \ms{C}')=\{(a, \msa)\}$. For $\alpha\in X$, let $c_\alpha=b_\alpha\cup c'_\alpha$.
Find $Y'\in [X]^{\omega_1}$ and $I, J$ such that
\begin{enumerate}\setcounter{enumi}{4}
\item for $\alpha\in Y'$, $c_\alpha[I]=b_\alpha$ and $c_\alpha[J]=c'_\alpha$.
\end{enumerate}
Then by Lemma  \ref{lem 7.l}, find $Y\in [Y']^{\omega_1}$ such that
\begin{enumerate}\setcounter{enumi}{5}
\item $(c, \ms{C})$ is a $\mcc$-candidate satisfying (R1)$^*$ where $c=a\cup b$ and $\ms{C}=\{c_\alpha: \alpha\in Y\}$.
\end{enumerate}
By (1), (3) and (5), for $\alpha\in Y$, $c_\alpha[J]$ is a union of $n$ elements in $\msa'$. Fix a partition 
\[J=\bigcup_{i<n} J_i\]
 of $J$ into $n$ pieces such that each $c_\alpha[J_i]$ is in $\msa'$.  Denote
 \[J^*=\{J_i(l^*): i<n\}.\]
Then for $\alpha\in Y$,
 \[c_\alpha[J^*]=\{a_\xi(l^*): \xi\in q_\alpha\setminus \overline{q}\}.\]

We now define a  function $g: [|c|+(n+2)N_\ms{C}]^3\ra 2$ that is $\mcr$-satisfiable for $(c, \ms{C})$ in the sense of Lemma \ref{lem 9.1} (ii). Denote for $i<n+2$,
\[I_i^*=[|c|+iN_\ms{C}, |c|+(i+1)N_\ms{C}).\]

First define $g$ partially on $\bigcup_{i<j<n+2}[c^{-1}[b]\cup I^*_i[I]\cup I^*_j[I]]^3$ such that

\begin{enumerate}\setcounter{enumi}{6}
\item for every $i<j<n+2$, $(c^{-1}[b]\cup I^*_i[I]\cup I^*_j[I], g)$ is isomorphic to $(|b|+2N, f)$.
\end{enumerate}
Then extend the domain of $g$ to $\bigcup_{i<j<n+2} [|c|\cup I^*_i\cup I^*_j]^3$  such that
\begin{enumerate}\setcounter{enumi}{7}
\item for every $i<j<n+2$, $g\up   [|c|\cup I^*_i\cup I^*_j]^3$ is $\mcr$-satisfiable for $(c, \ms{C})$.
\end{enumerate}
By Lemma \ref{lem 9.0},  we only need that for every $(d, \ms{D}, I', J')\in \mcc_{2, \max}(c, \ms{C})$, $(c^{-1}[d]\cup I^*_i[I']\cup I^*_j[J'], g)$ is isomorphic to $(|d|+2N_\ms{D}, h)$ for some $h\in \mcr(d, \ms{D})$. By (2) and (5), for $(d, \ms{D}, I', J')\in \mcc_{2, \max}(c, \ms{C})$, either $I'\times J'\subseteq I\times I$ or $I'\times J'\cap I\times I=\emptyset$. Together with the fact that $f$ is $\mcr$-satisfiable for $(b, \msb')$, (7) and (8) can be satisfied simultaneously.

Now extend the domain of $g$ to $[|c|+(n+2)N_\ms{C}]^3$ such that
\begin{enumerate}\setcounter{enumi}{8}
\item  $g(x, y,z)=1$ if $(x, y, z)\in \bigcup_{i<j<k<n+2} I^*_i[J^*]\times I^*_j[J^*]\times I^*_k[J^*]$.
\end{enumerate}
Note that the requirement (9) will not conflict the requirements (7)-(8) since their satisfactions  refer to disjoint domains.

Now by (8) and Lemma \ref{lem 9.1}, there are $c_{\alpha_0}<\cdots<c_{\alpha_{n+1}}$ in $\ms{C}$ such that 
 \begin{enumerate}\setcounter{enumi}{9}
\item $(c\cup \bigcup_{i<n+2} c_{\alpha_i}, \pi)\text{ is isomorphic to }(|c|+(n+2)N_\ms{C}, g)$.
\end{enumerate}
We will then find $i<n+1$ such that $q_{\alpha_i}$ is compatible with $q_{\alpha_{n+1}}$. Equivalently, we will find $i<n+1$ such that $\{p_\xi: \xi\in q_{\alpha_i}\cup q_{\alpha_{n+1}}\setminus \overline{q}\}$ is an antichain of $\mathbb{P}$. For this, we will use the following property.
 \begin{enumerate}\setcounter{enumi}{10}
\item For every $\xi\in q_{\alpha_{n+1}}\setminus \overline{q}$, 
\[|\{i<n+1: p_\xi\text{ is compatible with $p_{\xi'}$ for some } {\xi'}\in q_{\alpha_i}\setminus \overline{q}\}|\leq 1.\]
\end{enumerate}
To see this, suppose otherwise. For $i<j<n+1$, ${\xi'}\in q_{\alpha_i}\setminus \overline{q}$ and  ${\xi''}\in q_{\alpha_j}\setminus \overline{q}$, $p_\xi$ is compatible with both $p_{\xi'}$ and $p_{\xi''}$. Then  by our choice of $\Sigma'$ and the fact $\{\xi', \xi'', \xi\}\subset \Sigma'$, 
\[\pi(a_{\xi'}(l^*), a_{\xi''}(l^*), a_\xi(l^*) )=0.\]
Note by our choice of $J^*$, $a_\xi(l^*)\in c_{\alpha_{n+1}}[J^*]$. Similarly, $a_{\xi'}(l^*)\in c_{\alpha_i}[J^*]$ and $a_{\xi''}(l^*)\in c_{\alpha_j}[J^*]$. This is a contradiction since by (9)-(10), 
\[\pi(a_{\xi'}(l^*), a_{\xi''}(l^*), a_\xi(l^*) )=1.\]
This contradiction shows that (11) holds.

By (11) and the fact $|p_{\alpha_{n+1}}\setminus \overline{q}|=n$, there exists $i<n+1$ such that for every $\xi\in q_{\alpha_{n+1}}\setminus \overline{q}$ and every $\xi'\in q_{\alpha_i}\setminus \overline{q}$, $p_\xi$ is incompatible with $p_{\xi'}$. Now $q_{\alpha_i}\cup q_{\alpha_{n+1}}\in \mc{Q}$.

By (5), (7) and (10), 
\[q_{\alpha_i}\cup q_{\alpha_{n+1}}\Vdash b_{\alpha_i}< b_{\alpha_{n+1}}\text{ are in $\dot{\msb}$ and witness (Res) for } (b, \dot{\msb}) \text{ and } f.\] 
This shows (Res) and finishes the proof of the lemma.
 \end{proof}
 
 Then for Subalternative 3.1, we force with $\mc{Q}$. By absoluteness of $\varphi_1(\mcc, \mct, \bfe)$ and above lemma, $\varphi_1(\mcc, \mct, \bfe)$ and $\psi_2(\mcc, \mct, \bfe, \pi, \mcr)$ holds in the forcing extension. On the other hand, $\mc{Q}$ adds an uncountable antichain of $\mathbb{P}$. So the second outcome occurs.\bigskip
 
 Before proceeding to the last case, we would like to point out that in all previous alternatives, 2-ccc coloring is not assumed and we either force with $\mathbb{P}$ or destroy ccc of $\mathbb{P}$. But this cannot succeed for all ccc posets. For example, $\mc{H}^\pi_0$ is ccc, destroys $\psi_2(\mcc, \mct, \bfe, \pi, \mcr)$, and we can neither force with $\mc{H}^\pi_0$ nor destroy ccc of $\mc{H}^\pi_0$. So our last case will destroy  2-ccc coloring only.\bigskip
 
 \textbf{Subalternative 3.2.} $\mathbb{P}$ has 2-ccc coloring and for some $\xi_0<\xi_1$ in $\Sigma$,
\begin{align*}
\Sigma''=\{\gamma\in \Sigma\setminus (\xi_1+1): ~& \pi(a_{\xi_0}(l^*), a_{\xi_1}(l^*), a_\gamma(l^*))=1 \text{ and}\\
& p_\gamma \text{ is compatible with both $p_{\xi_0}$ and }p_{\xi_1}\}
\end{align*}
is uncountable.\medskip

By (Alt3.3), $(b, \{a_\alpha: \alpha\in \Sigma''\})$ is a $\mcc$-candidate where $b=a\cup \{a_{\xi_0}(l^*), a_{\xi_1}(l^*)\}$. By (Alt3.2), $\mc{C}_{1,\max}(b, \{a_\alpha: \alpha\in \Sigma''\})=\{(a, \msa, N_\msa)\}$.

Find $\Sigma^*\in [\Sigma'']^{\omega_1}$ such that
\begin{enumerate}[({Alt3.}1)]\setcounter{enumi}{3}
\item $(b, \msb)$ is a $\mcc$-candidate satisfying (R1)$^*$ and $\mc{C}_{1,\max}(b, \msb)=\{(a, \msa, N_\msa)\}$  where $\msb=\{a_\alpha: \alpha\in \Sigma^*\}$.
\end{enumerate}

We first consider the case that $\mathbb{P}$ has  the following  additional property:  
\begin{itemize}
\item a finite subset of $\mathbb{P}$ is centered if it is pairwise compatible. 
\end{itemize}
Together with (Alt3.1),  the following property holds.
 \begin{enumerate}[({Alt3.}1)]\setcounter{enumi}{4}
 \item  For $\alpha<\beta<\gamma$ in $\Sigma$, if $\{p_\alpha, p_\beta, p_\gamma\}$ is pairwise compatible, then 
 \[\pi(a_\alpha(l^*), a_\beta(l^*), a_\gamma(l^*))=0.\]
 In particular, if $p_\alpha$ is compatible with $p_\beta$ for $\alpha<\beta$ in $\Sigma^*$, then
  \[\pi(a_{\xi_0}(l^*), a_\alpha(l^*), a_\beta(l^*))=0=\pi(a_{\xi_1}(l^*), a_\alpha(l^*), a_\beta(l^*)).\]
 \end{enumerate} 
 Let $H$ be the collection of all $\mcr$-satisfiable for $(b, \msb)$ function $f$ such that
 \[\text{for some $i<2$, }f(|a|+i, |b|+l^*, |b|+N_\msa+l^*)=1.\]
 Here is the role that $H$ will play. Suppose $\{a_\alpha, a_\beta\}\in \mc{P}_{b, \msb, H}$. Then by   definition of $H$ and $\mc{P}_{b, \msb, H}$, $\pi(a_{\xi_i}(l^*), a_\alpha(l^*), a_\beta(l^*))=1$ for some $i<2$. Together with (Alt3.5), $p_\alpha$ is incompatible with $p_\beta$. 
 
So force with $\mc{P}_{b, \msb, H}$ will add an uncountable antichain to $\mathbb{P}$.  We then show that $H$ satisfies the assumption of Lemma \ref{lem 9.5}. 
 
 \begin{lem}\label{lem 9.10}
For $a_\alpha\in \msb$, $b'\subseteq b$ and $I, J\subseteq N_\msa$,
if $b'\cup a_\alpha[I]\in \mc{H}^\pi_0$ and $ b'\cup a_\alpha[J]\in \mc{H}^\pi_0$,
 then there exists $f\in H$ that is constant 0 on   
$[{b}^{-1}[b']\cup (|b|\widetilde{+} I)\cup (|b|+N_\msa\widetilde{+} J)]^3$.
\end{lem}
\begin{proof}
We define $f: [|b|+2N_\msa]^3\ra 2$ by steps.

First define a partial $f$ on $\bigcup_{i<2} [|b|\cup [|b|+iN_\msa, |b|+(i+1)N_\msa)]^3$ such that
 \begin{enumerate}
\item  (Res.1) holds for $(b, \msb)$.
\end{enumerate}

Then by (R3),   define $f$ on $[|a|\cup   [|b|, |b|+2N_\msa)]^3$ such that
 \begin{enumerate}\setcounter{enumi}{1}
\item $(|a|\cup [|b|,   |b|+2N_\msa), f)$ is isomorphic to $(|a|+2N_\msa, g)$ for some $g\in \mcr(a, \msa)$;
\item $f$ is constant 0 on $[b^{-1}[a\cap b']\cup (|b|\widetilde{+} I)\cup (|b|+N_\msa\widetilde{+} J)]^3$.
\end{enumerate}
Note by (Alt3.4),  $\mc{C}_{2,\max}(b, \msb)=\{(a, \msa, N_\msa, N_\msa)\}$. So by Lemma \ref{lem 9.0}, every extension of $f$ is $\mcr$-satisfiable for $(b, \msb)$.\medskip

\textbf{Case 1.} $ \{a_{\xi_0}(l^*), a_{\xi_1}(l^*)\}\not\subseteq b'$.\medskip

Say $a_{\xi_1}(l^*)\not\in b'$. Extend $f$ such that for $i<2$, $x\in [|b|, |b|+N_\msa)$ and $y\in [|b|+N_\msa, |b|+2N_\msa)$, $f(|a|+i, x, y)=i$.

It is straightforward to check that $f$ is as desired.\medskip

\textbf{Case 2.}  $ \{a_{\xi_0}(l^*), a_{\xi_1}(l^*)\}\subseteq b'$. \medskip

By definition of $\Sigma''$ and the fact $\Sigma^*\subseteq \Sigma''$, $l^*\notin I\cup J$. Extend $f$ such that for $i<2$, $x\in |b|\widetilde{+} I$ and $y\in |b|+N_\msa\widetilde{+} J$, $f(|a|+i, x, y)=0$ and $f(|a|, |b|+l^*, |b|+N_\msa+l^*)=1$.

By (1)-(3) and above definition, $f$ is as desired.
\end{proof}
 
 Now we force with $\mc{P}_{b, \msb, H}$. Assume $G$ is an uncountable generic filter.  Let $\mcc'=\mcc\cup\{(b, \bigcup G)\}$ and $\mct', \bfe', \mcr'$ are induced from $\mcc'$. 

Then by Lemma \ref{lem 9.5}, $\varphi_1(\mcc', \mct', \bfe')$   and   $\psi_2(\mcc', \mct', \bfe', \pi, \mcr')$ hold. Also, by our choice of $H$, ccc of $\mathbb{P}$ is destroyed.\medskip

We then consider the general case without assuming the additional property of $\mathbb{P}$.
 
Find, for every $\gamma\in \Sigma^*$, a common lower bound $p_\gamma'$ of $p_{\xi_0}$ and $p_\gamma$.
Define $\tau: [\Sigma^*]^2\ra 2$ by 
\[\tau(\alpha, \beta)=0\text{ iff $p'_\alpha$ is compatible with }p'_\beta.\]
Since $\mathbb{P}$ has 2-ccc coloring, $\mc{H}^\tau_0$ is ccc.

Now we consider the ccc poset $\mc{H}^\tau_0$. Note by (Alt3.4) and Lemma \ref{lem 9.0}, there is an $\mcr$-satisfiable for $(b, \msb)$ function $f: [|b|+2N_\msa]^3\ra 2$ such that
\[f(|a|, |b| +l^*, |b| +N_\msa+l^*)=1.\]

Then by (Alt3.1), 
\[\Vdash_{\mc{H}^\tau_0} \text{ (Res) fails witnessed by } (b, \{a_\alpha\in \msb: \{\alpha\}\in \dot{G}\}) \text{ and }f\]
where $\dot{G}$ is the canonical name of the $\mc{H}^\tau_0$-generic filter. To see this, note if $\tau(\alpha, \beta)=0$, then $\{p_{\xi_0}, p_\alpha, p_\beta\}$ has a common lower bound. Then by (Alt3.1), $\pi(a_{\xi_0}(l^*), a_{\alpha}(l^*), a_\beta(l^*))=0$.  Hence $(b\cup a_\alpha\cup a_\beta, \pi)$ is not isomorphic to $(|b|+2N_\msa, f)$.

So  Alternative 1 does not happen to $\mc{H}^\tau_0$. 

If Alternative 2 or Subalternative 3.1 happens to $\mc{H}^\tau_0$, then the corresponding procedure will destroy ccc of $\mc{H}^\tau_0$ while preserve $\psi_2(\mcc, \mct, \bfe, \pi, \mcr)$ (or $\psi_2(\mcc', \mct', \bfe', \pi, \mcr')$ for extended $\mcc'$ and induced $\mct', \bfe', \mcr'$). Then we would destroy 2-ccc coloring of $\mathbb{P}$.

So we assume that Subalternative 3.2 happens to $\mc{H}^\tau_0$. Note that since $\tau$ is a coloring on pairs and $\mc{H}^\tau_0$ is ccc,  $\mc{H}^\tau_0$ has 2-ccc coloring.

Now,  $\mc{H}^\tau_0$ satisfies the same assumption as $\mathbb{P}$, as well as the following  additional property:  a finite subset of $\mc{H}^\tau_0$ is centered if it is pairwise compatible.
 
 Then our treatment for the first case induces a method of destroying ccc of $\mc{H}^\tau_0$ while preserving $\psi_2(\mcc', \mct', \bfe', \pi, \mcr')$ for some extended $\mcc'$. This destroys 2-ccc coloring of $\mathbb{P}$.

This finishes our treatment for Subalternative 3.2 and all alternatives.\bigskip

 \subsection{Proof of  Theorem \ref{thm k2 not k3}}
Start from a model of GCH and let $\mcc^0=\{(\emptyset, [\omega_1]^1)\}$.
Construct    a coloring $\pi: [\omega_1]^3\ra 2$  by CH (or generically add $\pi$ by forcing) such that 
\[[\omega_1]^2\subseteq \mc{H}^\pi_0 \text{ and }\psi_2(\mcc^0, \mct^0, \bfe^0, \pi, \mcr^0) \text{ holds}.\]
 Note that $\varphi_1(\mcc^0, \mct^0, \bfe^0)$ and $\psi_2^-(\mcc^0, \mct^0, \bfe^0, \pi, \mcr^0)$ are automatically satisfied.

Then iteratively force ccc posets $\langle \mc{P}_\alpha, \dot{\mc{Q}}_\beta: \beta<\omega_2, \alpha\leq \omega_2\rangle$ with finite support, together with $\mc{P}_\alpha$-names $\dot{\mcc}^\alpha, \dot{\mct}^\alpha, \dot{\bfe}^\alpha, \dot{\mcr}^\alpha$ for $\alpha\leq \omega_2$ such that for $\beta<\omega_2$,
 \begin{itemize}
 \item   $\Vdash_{\mc{P}_\beta} \dot{\mc{Q}}_\beta$ has size $\leq \omega_1$ and $|\dot{\mcc}^\beta|<\omega_2$;
 \item  $\Vdash_{\mc{P}_\beta} \varphi_1(\dot{\mcc}^ \beta, \dot{\mct}^ \beta, \dot{\bfe}^ \beta)$ and $\psi_2(\dot{\mcc}^ \beta, \dot{\mct}^ \beta, \dot{\bfe}^ \beta, \pi, \dot{\mcr}^ \beta)$;
 \item if $\beta$ is a limit ordinal, then $\Vdash_{\mc{P}_\beta} \dot{\mcc}^\beta=\bigcup_{\alpha<\beta} \dot{\mcc}^\alpha$.
 \end{itemize}
 
 By Alternatives 1-3 discussed in previous subsection, we may assume that 
 \begin{itemize}
 \item for every ccc poset $\mathbb{P}\in V^{\mc{P}_\alpha}$ of size $\leq \omega_1$ where $\alpha<\omega_2$, if $\mathbb{P}$ has 2-ccc coloring in $V^{\mc{P}_{\omega_2}}$, then $\mathbb{P}$ is forced at some stage $\beta\geq\alpha$.
 \end{itemize}
 This shows that
 \begin{itemize}
 \item {\rm MA$_{\omega_1}$(2-ccc coloring)} holds  in $V^{\mc{P}_{\omega_2}}$.
  \end{itemize}
 Recall that ccc forcing does not change $\mathfrak{b}_{\omega_1}$ and hence $\omega_2=\mathfrak{b}_{\omega_1}\leq \mathfrak{a}_{\omega_1}$ in $V^{\mc{P}_\alpha}$ for all $\alpha\leq \omega_2$. 
 
 So by Lemma \ref{lem 9.2}, 
  \begin{itemize}
 \item $\mc{H}^\pi_0$ is ccc and $\pi$ has no uncountable 0-homogeneous subset in $V^{\mc{P}_\alpha}$ for all $\alpha<\omega_2$ and hence  in $V^{\mc{P}_{\omega_2}}$.
  \end{itemize}
 So in $V^{\mc{P}_{\omega_2}}$, $\mc{K}_3$ fails. This finishes the proof of Theorem \ref{thm k2 not k3}.

 \section{Closing remarks}
 
 For properties of form P($\Phi\ra\Psi$) (or P$_{\omega_1}(\Phi\ra\Psi$)), we concentrate on three types of questions.\medskip
 
 Type 1: How weak can $\Psi$ be in order for P(ccc$\ra\Psi$) (or P$_{\omega_1}$(ccc$\ra\Psi$)) to have the full strength of MA$_{\omega_1}$? By Theorem \ref{thm K3toK4},   $\Psi$ can be  K$_3$.

We list   well-known ccc properties   in the following figure which has the bottom properties in Figure \ref{figure1} as top properties.  The strengths of strengthening properties in Figure \ref{figure3} are unknown.

 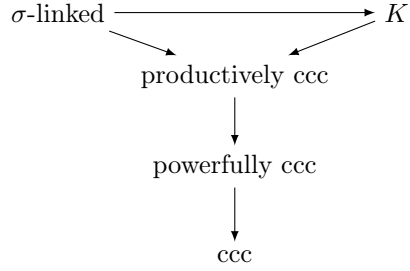
\begin{figure}[h]
\begin{tikzpicture}[auto, node distance=1.2 cm, >=latex]

\node (linked)  {$\sigma$-linked};
\node (K) [right of=linked, xshift=3.3cm] {$K$};
\node (productive) [below right of=linked, xshift=1.5cm] {productively ccc};
\node (powerful) [below of=productive] {powerfully ccc};
\node (ccc) [below of=powerful] {ccc};

 \draw [->] (linked) -- (K);
 \draw [->] (linked) -- (productive);
 \draw [->] (K) -- (productive);
 \draw [->] (productive) -- (powerful);
 \draw [->] (powerful) -- (ccc);

\end{tikzpicture}
\caption{Weak forcing properties}\label{figure3}
\end{figure}

 \begin{question}
 What of the following properties imply {\rm MA}$_{\omega_1}$?
 \begin{enumerate}[(a)]
 \item {\rm P(ccc$\ra$productively ccc)}, or $C(\omega_1)$ that ccc is productive.
 \item {\rm P(ccc$\ra$K)}, or $\ms{K}_2$.
 \item {\rm P$_{\omega_1}$(ccc$\ra\sigma$-linked)}.
 \end{enumerate}
 \end{question}
 These questions are natural and have been asked in the literature, e.g., \cite{TV}, \cite{LT2001}.  
 
 By \cite{Todorcevic86}, each of above 3 properties implies $\mathfrak{b}>\omega_1$. Very few other consequences   are known. See  \cite{Todorcevic91},  \cite{LT2001}, \cite{Yorioka19}, \cite{MY} for more partial results.\medskip
 
 Type 2: How strong can $\Phi$ be in order for P($\Phi\ra$K$_3$) to have the full strength of MA$_{\omega_1}$? By \cite{Peng25}, $\Phi$ must be strictly weaker than powerfully ccc. So we shall search for $\Phi$ not limited in Figure \ref{figure3}. Properties $\mc{K}_n$ (or P($n$-ccc coloring$\ra$K$_3$) in our notation), for $n\geq 2$, might be quite strong. So $n$-ccc coloring, for $n\geq 2$, are natural candidates of $\Phi$. 
  By Theorem \ref{thm k2 not k3}, $\Phi$ cannot be 2-ccc coloring.
 \begin{question}
 For $n\geq 3$, does $\mc{K}_n$ (or {\rm MA$_{\omega_1}$($n$-ccc coloring})) imply {\rm MA}$_{\omega_1}$?
 \end{question}

One might attempt to generalize the proof of Theorem \ref{thm k2 not k3} to distinguish, e.g., $\mc{K}_3$ and $\mc{K}_4$/MA$_{\omega_1}$. We would like to point out one major difficulty in this attempt. The proof of Theorem \ref{thm k2 not k3} makes essential use of the following difference between coloring on pairs and coloring on triples.
\begin{itemize}
\item The root of a $\Delta$-system is usually omitted in analyzing coloring on pairs and plays an important role in coloring on triples.
\end{itemize}
More precisely, in Subalternative 3.2 of Subsection 9.2,  when viewing $\mathbb{P}$ as form $\mc{H}^\tau_0$ for some $\tau: [\omega_1]^2\ra 2$, $\xi_i$'s are in the root and $\tau$-0-homogeneity of $p_\alpha\cup p_\beta$ determines  $\tau$-0-homogeneity of $p_{\xi_i}\cup p_\alpha\cup p_\beta$ which in turn determines $\pi$-0-homogeneity of $\{a_{\xi_i}(l^*), a_\alpha(l^*), a_\beta(l^*)\}$ for $i<2$ and $\alpha, \beta\in \Sigma^*$.

Since the root of a $\Delta$-system plays important roles in colorings on $n$-tuples for all $n\geq 3$, some new difference might be needed to distinguish, e.g., $\mc{K}_3$ and $\mc{K}_4$.

Coloring on $n$-tuples for $n\geq 4$ may encounter another technical difficulty, when the requirement $\mcr(a, \msa)$ collects patterns of $\pi$ on $a$ together with 3 (or more) other inputs from $\msa$. We describe a strengthening, $\varphi_2(\mcc, \mct, \bfe)$, of $\varphi_1(\mcc, \mct, \bfe)$ to overcome this difficulty. 

$\varphi_2(\mcc, \mct, \bfe)$ is the assertion that $\varphi_1(\mcc, \mct, \bfe)$ together with the following statements hold.
 \begin{enumerate}[{\rm (C1)}]\setcounter{enumi}{2}
\item If $(a, \msa)\prec (b, \msb)$ are in $\mcc$, then $a\cup\bigcup \{a'\in \msa: a'\cap b\neq \emptyset\}\subseteq b$.
\item For every $\mcc'\in [\mcc]^{<\omega}$ and every $a\in [\omega_1]^{<\omega}$, there exists a finite $a^*\supseteq a$ such that for every $(b, \msb)\in \mcc'$, $b\cup\bigcup\{b'\in \msb: b'\cap a^*\neq\emptyset\}\subseteq a^*$.
\end{enumerate}
Unlike $\varphi_1(\mcc, \mct, \bfe)$, a candidate satisfying (C4) may need  to be forced.\medskip

Type 3: How strong is P($\Phi\ra\Psi$) (or P$_{\omega_1}(\Phi\ra\Psi$)) for two closely related properties $\Phi$ and $\Psi$? On one hand, investigation of this type of questions clarifies the powerful part of a strong property. On the other hand, in some sence, the strength of P($\Phi\ra\Psi$) (or P$_{\omega_1}(\Phi\ra\Psi$)) measures the difference between $\Phi$ and $\Psi$.  For example,  the difference between K$_n$ and K$_{n+1}$ is considerably  larger than the difference between K$_n$ and $\sigma$-$n$-linked since P(K$_n\ra$K$_{n+1}$) is strictly stronger than P$_{\omega_1}$(K$_n\ra\sigma$-$n$-linked).

For properties above property K, the strengths of P(K$_n\ra$K$_{n+1}$) and P($\sigma$-$n$-linked$\ra$K$_{n+1}$) are clear by Theorem \ref{thm K3toK4} and Theorem \ref{thm sK3toK4}. It is natural to investigate the exact strength of P$_{\omega_1}$(K$_n\ra\sigma$-$n$-linked). 

For this type of questions, very few is known for properties below property K and is closely related to questions of type 1.\bigskip

Above 3 types of questions, as well as the investigation of the paper, concentrate on forcing axioms for $\omega_1$ dense sets. Discussion of  forcing axioms for more dense sets is out of the scope of the current paper. We will only point out that in analyzing  forcing axioms for $\kappa>\omega_1$ dense sets, some idea and technique here may be useful  while the generalization of some treatment may encounter extra difficulty.

The treatment for $\mathfrak{t}$ in Proposition \ref{prop sK3toK4 t} works for larger $\kappa$ as well  (see \cite{Todorcevic86} for the $\mathfrak{b}$ part).
\begin{cor}\label{cor general t}
 For $n\geq 2$, if every $\sigma$-$n$-linked poset $\mathbb{P}$ of size $\leq \kappa$ has an $(n+1)$-linked subset of size $|\mathbb{P}|$, then $\mathfrak{t}>\kappa$.
\end{cor}

Theorem \ref{thm sK3toK4 sK3top} might be generalized to larger $\kappa$ with extra effort.

The generalization of Lemma \ref{lem K3toK4} (or \cite[Corollary 2.7]{TV}) to larger $\kappa$ is likely to encounter some major difficulty. The coding on   $\kappa>\omega_1$ may need extra assumption (see Section 2-3 of \cite{Todorcevic91} for a discussion on $\kappa=\omega_2$).\medskip

The damage control structure can be generalized to larger $\kappa$ as well. But   generalizations may vary for different purposes.

For example, when analyzing a ccc coloring $\pi: [\kappa]^{<\omega}\ra 2$, an appropriate collection $\mc{C}$ may collect $(a, \msa)$ for $a\in [\kappa]^{<\omega}$ and $\msa \in [[\kappa]^{N_\msa}]^{\omega_1}$ with $N_\msa<\omega$. Also, non-overlapping requirement on $\msa$ should be modified accordingly since $\sup(\msa_i)<\sup(\msa_j)$ is possible.

Moreover, when analyzing $\kappa$-cc notions, $\mcc$ may collect $(a, \msa)$ for $a\in [\kappa]^{<\omega}$ and $\msa \in [[\kappa]^{N_\msa}]^{\kappa}$ with $N_\msa<\omega$. And when analyzing $\kappa$ above the continuum,  $\mcc$ may collect $(a, \msa)$ for an infinite $a$ and a family $\msa$ of infinite sets.

 \section*{Acknowledgement}
 I would like to thank Stevo Todorcevic and Justin Moore for their useful suggestions on earlier draft, which have enhanced the paper's presentation.

    \bibliographystyle{plain}

\end{document}